\documentclass[final,notitlepage,8pt]{article}
\usepackage{indentfirst}
\usepackage[normalem]{ulem}
\usepackage{float}

\usepackage{graphicx}
\usepackage{enumitem}[2011/09/28]
\setenumerate{align=left, leftmargin=0pt,labelsep=.5em, labelindent=0\parindent,listparindent=\parindent,itemindent=*}
\usepackage{array}
\usepackage{empheq}
\usepackage{natbib}
\usepackage[all]{xy}

\usepackage{exscale,relsize}
\usepackage{multirow}

\usepackage{caption}[2012/02/19]

\usepackage{longtable}

\usepackage{fouriernc}
\usepackage{fourier}
\usepackage{amssymb,bm,amsmath}
\usepackage{amsfonts}

\usepackage[OT2,T1,T2A]{fontenc}

\usepackage[english,french,german,italian,russian]{babel}
\usepackage{appendix}

\DeclareMathOperator{\D}{d}
\DeclareMathOperator{\I}{Im}
\DeclareMathOperator{\R}{Re}

\bibpunct{[}{]}{,}{n}{}{;}

\def\qed{\hfill$ \blacksquare$}
\def\eor{\hfill$ \square$}

\newtheorem{theorem}{Theorem}[section]

\newtheorem{lemma}[theorem]{Lemma}
\newtheorem{proposition}[theorem]{Proposition}
\newtheorem{corollary}[theorem]{Corollary}

\newenvironment{proof}[1][Proof]{\begin{trivlist}
\item[\hskip \labelsep {\bfseries #1}]}{\end{trivlist}}

\newenvironment{remark}[1][Remark]{\begin{trivlist}
\item[\hskip \labelsep {\bfseries #1}]}{\end{trivlist}}

\begin{document}

\selectlanguage{english}
\title{\textsf{\textbf{ Legendre Functions, Spherical Rotations,\\ and Multiple Elliptic Integrals}}}\author{Yajun Zhou
\\\begin{small}\textsc{Program in Applied and Computational Mathematics, Princeton University, Princeton, NJ 08544}\end{small}
}
\date{}
\maketitle

\def\abstractname{}
\begin{abstract}
     \noindent A closed-form formula is derived for the generalized Clebsch-Gordan integral $ \int_{-1}^1 {[}P_{\nu}(x){]}^2P_{\nu}(-x)\D x$, with $ P_\nu$ being the Legendre function of arbitrary complex degree $ \nu\in\mathbb C$. The finite Hilbert transform of $ P_{\nu}(x)P_{\nu}(-x),-1<x<1$ is evaluated. An analytic proof is provided for a recently conjectured identity $\int_0^1[\mathbf K( \sqrt{1-k^2} )]^{3}\D k=6\int_0^1[\mathbf K(k)]^2\mathbf K( \sqrt{1-k^2} )k\D k=[\Gamma(\frac{1}{4})]^{8}/(128\pi^2) $ involving complete elliptic integrals of the first kind $ \mathbf K(k)$ and Euler's gamma function $ \Gamma(z)$. \end{abstract}
\setcounter{section}{-1}
\section{Introduction}

In a wide range of mathematical and physical contexts, one needs to evaluate \textit{multiple elliptic integrals}  where the integrands involve one or more factors of  complete elliptic integrals.  As a glimpse of some recent applications in number theory, high-energy physics, statistical mechanics and probability theory, we   mention  automorphic Green's functions  \cite{KontsevichZagier},  multi-loop Feynman diagrams
\cite{Bailey2008}, lattice Green's functions \cite{Zucker2011} and short uniform random walks \cite{BNSW2011,BSWZ2012,Wan2012}. Various  techniques, of algebraic, geometric, combinatoric or analytic flavor, have been developed to express these multiple elliptic integrals in terms of familiar mathematical constants and special values of Euler's gamma function \cite{KontsevichZagier,Bailey2008,Zucker2011,BNSW2011,BSWZ2012,Wan2012}.

When the integrands contain the products of three or more complete elliptic integrals, the evaluation can become analytically challenging. For example, the following relation\begin{align*}\frac{[\Gamma(\frac{1}{4})]^{8}}{128\pi^{2}}\overset{?}{=}\int_0^1\left[\mathbf K\left( \sqrt{1-k^2} \right)\right]^{3}\D k\overset{?}{=}6\int_0^1[\mathbf K(k)]^2\mathbf K\left( \sqrt{1-k^2} \right)k\D k\end{align*}has been recently conjectured by J.~G. Wan in \cite{Wan2012} based on numerical evidence.

 In this work, we solve Wan's conjecture and compute a related family of   multiple elliptic integrals with sophisticated appearance. We employ an  approach that is probably overlooked in previous literature: connecting elliptic integrals to spherical functions and spherical rotations. Concretely speaking, the complete elliptic integral of the first kind $ \mathbf K(k):=\int_0^{\pi/2}(1-k^2\sin^2\theta)^{-1/2}\D \theta$ is related to fractional degree Legendre functions $ P_{-1/2},P_{-1/3},P_{-1/4},P_{-1/6}$, where \begin{align*}P_\nu(1)=1;\quad P_\nu(\cos\theta):=\frac{2}{\pi}\int_0^\theta\frac{\cos\frac{(2\nu+1)\beta}{2}}{\sqrt{2(\cos\beta-\cos\theta)}}\D\beta,\quad \theta\in(0,\pi),\nu\in\mathbb C\end{align*}extend the familiar Legendre polynomials \[P_\ell(x)=\frac{1}{2^\ell\ell!}\frac{\D^\ell}{\D x^\ell}[(x^2-1)^\ell],\quad \ell\in\mathbb Z_{\geq0}\]to arbitrary complex degree $\nu\in\mathbb C $.
By a modest generalization of the ``spherical harmonic coupling'' for Legendre polynomials (Clebsch-Gordan theory) to Legendre functions of arbitrary degree, we are able to prove the integral identity\begin{align*}\int_{-1}^1 {[}P_{\nu}(x){]}^2P_{\nu}(-x)\D x=\lim_{z\to\nu}\frac{1+2\cos(\pi z)}{3}\frac{\pi\Gamma\left(\frac{z+1 }{2}\right) \Gamma\left(\frac{3 z+2 }{2}\right)}{\left[ \Gamma\left(\frac{1-z}{2}\right)\right]^2\left[ \Gamma\left(\frac{z+2 }{2}\right)\right]^3 \Gamma\left(\frac{3 z+3 }{2}\right)},\quad\nu\in\mathbb  C\end{align*}
and compute the finite Hilbert transform\begin{align*}\mathcal
P\int_{-1}^1\frac{2P_{\nu}(\xi)P_{\nu}(-\xi)}{\pi(x-\xi)}\D \xi=\frac{[P_{\nu}(x)]^2-[P_{\nu}(-x)]^2}{\sin(\nu\pi)},\quad\nu\in\mathbb  C\smallsetminus\mathbb Z.
\end{align*} The pair of formulae above, along with a couple of geometric transformations related to rotations on spheres,  allow us to evaluate a slew of multiple elliptic integrals, which embody Wan's conjectural identity as a special case.

The article is organized as follows. In \S\ref{sec:Hansen_Heine}, we review some classical results that express fractional degree Legendre functions as complete elliptic integrals, and give a new proof of Hobson's coupling formula for two Legendre functions, based on scaling limits. We carry on the scaling analysis in \S\ref{sec:G_CG}, so as to interpolate the standard Clebsch-Gordan integrals for Legendre polynomials into a closed-form expression for $ \int_{-1}^1 {[}P_{\nu}(x){]}^2P_{\nu}(-x)\D x$ with arbitrary complex degree $\nu\in\mathbb C$. This is followed by a geometric intermezzo on spherical rotations in \S\ref{sec:rotations},  providing tools to  transform one multiple elliptic integral into another, as well as  preparing some integral identities for   \S\ref{sec:Tricomi}. The finite Hilbert transform (also known as the Tricomi transform) of the function  $ P_{\nu}(x)P_{\nu}(-x)$ is  studied in \S\ref{sec:Tricomi}, in the spirit of  ``spherical harmonic coupling''. Lastly, \S\ref{sec:outlook} serves as a perspective on subsequent works on multiple elliptic integrals related to, or motivated by Legendre functions.  \section{\label{sec:Hansen_Heine}Fractional Degree Legendre Functions and Hansen-Heine Scaling Limits}The Legendre  functions of the first kind can be defined via the Mehler-Dirichlet integral\begin{align}\label{eq:MD_Pnu}P_\nu(1)=1;\quad P_\nu(\cos\theta):=\frac{2}{\pi}\int_0^\theta\frac{\cos\frac{(2\nu+1)\beta}{2}}{\sqrt{2(\cos\beta-\cos\theta)}}\D\beta,\quad \theta\in(0,\pi),\nu\in\mathbb C.\end{align}Alternatively, we may characterize  $ P_\nu(x),-1<x\leq1$ as the unique $ C^2(-1,1)$ solution to \begin{align}\label{eq:diff_def_Pnu}\frac{\D}{\D x}\left[ (1-x^2 )\frac{\D f(x)}{\D x}\right]+\nu(\nu+1) f(x)=0,\quad f(1)=1.\end{align}By either definition, we have  $ P_\nu(x)\equiv P_{-\nu-1}(x)$.
Later in this work, we shall also use the Legendre  functions of the second kind, given by \begin{align}\label{eq:def_Qnu}Q_\nu(x):=\frac{\pi}{2\sin(\nu\pi)}[\cos(\nu\pi)P_\nu(x)-P_{\nu}(-x)],\quad \nu\in\mathbb C\smallsetminus\mathbb Z;\quad Q_n(x):=\lim_{\nu\to n}Q_\nu(x),\quad n\in\mathbb Z_{\geq0}\end{align}for $ -1<x<1$. By convention, the Legendre function of the second kind $ Q_n(x)$ is undefined for any negative integer degree $ n\in\mathbb Z_{<0}$.
While both $ P_\nu(x)$ and $ Q_\nu(x)$ solve the same Legendre differential equation of degree $ \nu\in\mathbb C\smallsetminus\mathbb Z_{<0}$, the function $ Q_\nu(x)$ blows up logarithmically as $ x\to1-0^+$, unlike the natural boundary condition for $ P_\nu(1)\equiv1$.

Additionally, the Legendre  function of the first kind can be represented as the Laplace integral under some restrictions:\begin{align}P_\nu(z):=\frac{1}{2\pi}\int_0^{2\pi}e^{\nu\log(z+i\sqrt{1-z^{2}}\cos\phi)}\D\phi\equiv\frac{1}{2\pi}\int_0^{2\pi}e^{\nu\log(z-i\sqrt{1-z^{2}}\cos\phi)}\D\phi,\quad \R z>0,\nu\in\mathbb C.\label{eq:Laplace_Pnu}\end{align}

\subsection{Kleiber-Ramanujan Correspondence}
Highlights of the relation $ P_{-1/2}(\cos\theta)=(2/\pi)\mathbf K(\sin(\theta/2))$ date back to J. Kleiber's  posthumous publication  \cite{Kleiber1893} in 1893, where systematic studies were carried out on connections between complete elliptic integrals and Legendre functions $ P_{(2n+1)/2},n\in\mathbb Z$, the latter of which are useful for solving differential equations in toroidal coordinates \cite[Ref.][\S\S253-258]{Hobson1931}.

In his seminal paper \cite{RamanujanPi} and legendary notebooks \cite[Ref.][Chap.~33]{RN5}, S.~Ramanujan developed rich and elegant theories relating  $ P_{-1/3}, P_{-1/4},P_{-1/6}$ to complete elliptic integrals of the first kind.

We recapitulate the aforementioned classical results in the following short lemma. \begin{lemma}[Some Fractional Degree Legendre Functions]\label{lm:P_nu_spec}\begin{enumerate}[label=\emph{(\alph*)}, ref=(\alph*), widest=a]\item We have the following identities for $ 0\leq\theta<\pi$:\begin{align}P_{-1/2}(\cos\theta)={}&\frac{2}{\pi}\mathbf K\left( \sin\frac{\theta}{2}\right)=\frac{2}{\pi}\int_0^{\pi/2}\frac{\D\psi}{\sqrt{\smash[b]{1-\sin^2(\theta/2)\sin^2\psi}}},\label{eq:P_half_sin2}\\ P_{-1/4}(\cos\theta)={}&\frac{2}{\pi}\frac{1}{\sqrt{\smash[b]{1+\sin(\theta/2)}}}\mathbf K\left( \sqrt{\frac{2\sin(\theta/2)}{1+\sin(\theta/2)}} \right)=\frac{2}{\pi}\int_0^{\pi/2}\frac{\D\psi}{\sqrt{\smash[b]{1-\sin^2(\theta/2)\sin^4\psi}}},\label{eq:P_quarter_sin4}\end{align}\item For a parameter $ p\in[0,1)$, the following relations \begin{subequations}\begin{align}P_{-1/6}\left( 1-2\frac{27p^{2}(1+p)^2}{4(1+p+p^{2})^{3}} \right)={}&\sqrt{\frac{1+p+p^{2}}{1+2p}}P_{-1/2}\left( 1-2\frac{p(2+p)}{1+2p} \right)=\frac{2}{\pi}\sqrt{\frac{1+p+p^{2}}{1+2p}}\mathbf K\left( \sqrt{\frac{p(2+p)}{1+2p}} \right),\label{eq:P_sixth_a}\\P_{-1/6}\left( 2\frac{27p^{2}(1+p)^2}{4(1+p+p^{2})^{3}}-1 \right)={}&\sqrt{\frac{1+p+p^{2}}{1+2p}}P_{-1/2}\left( 2\frac{p(2+p)}{1+2p} -1\right)=\frac{2}{\pi}\sqrt{\frac{1+p+p^{2}}{1+2p}}\mathbf K\left( \sqrt{\frac{1-p^{2}}{1+2p}} \right);\label{eq:P_sixth_b}\end{align}\end{subequations}\begin{subequations}hold for degree $ \nu=-1/6$, and the transformations\begin{align}P_{-1/3}\left( 1-2\frac{27p^{2}(1+p)^2}{4(1+p+p^{2})^{3}} \right)={}&\frac{1+p+p^2}{\sqrt{1+2p}}P_{-1/2}\left( 1-2\frac{p^{3}(2+p)}{1+2p} \right)=\frac{2}{\pi}\frac{1+p+p^2}{\sqrt{1+2p}}\mathbf K\left( \sqrt{\frac{p^3(2+p)}{1+2p}} \right),\label{eq:P_third_a}\\P_{-1/3}\left( 2\frac{27p^{2}(1+p)^2}{4(1+p+p^{2})^{3}} -1\right)={}&\frac{1+p+p^2}{\sqrt{3+6p}}P_{-1/2}\left( 2\frac{p^{3}(2+p)}{1+2p} -1\right)=\frac{2}{\pi}\frac{1+p+p^2}{\sqrt{3+6p}}\mathbf K\left( \sqrt{\frac{(1-p)(1+p)^{3}}{1+2p}} \right).\label{eq:P_third_b}\end{align}\end{subequations}apply to  degree $ \nu=-1/3$.\end{enumerate}\end{lemma}\begin{proof}
\begin{enumerate}[label=(\alph*),widest=a]\item\label{itm:P_half_quarter}  Let $ \mathbf E(k)=\int_0^{\pi/2}\sqrt{\smash[b]{1-k^2\sin^2\phi}}\D\phi$ be the complete elliptic integral of the second kind, then we have the differential relations \begin{align*}\frac{\D\mathbf K(k)}{\D k}=\frac{\mathbf E(k)}{k(1-k^2)}-\frac{\mathbf K(k)}{k},\quad \frac{\D\mathbf E(k)}{\D k}=\frac{\mathbf E(k)-\mathbf K(k)}{k}.\end{align*} It is thus  routine to check that the proposed forms of elliptic integrals in Eqs.~\ref{eq:P_half_sin2} and \ref{eq:P_quarter_sin4} satisfy the said Legendre differential equations and natural boundary conditions (Eq.~\ref{eq:diff_def_Pnu}).

The last integral in Eq.~\ref{eq:P_half_sin2} is the standard definition for the complete elliptic integral of the first kind.  With a variable substitution $ \vartheta=\arctan(\sqrt{1+\sin(\theta/2)}\tan\psi)$ \cite[Ref.][p.~105]{RN3}, one can convert the last expression in Eq.~\ref{eq:P_quarter_sin4} into the standard form. \item

As in the proof of part \ref{itm:P_half_quarter}, one can verify the relations between the respective fractional degree Legendre functions and complete elliptic integrals by a routine, though time-consuming procedure of Legendre differential equations.

Equations~\ref{eq:P_sixth_a} and \ref{eq:P_sixth_b} have appeared in Ramanujan's work \cite[Ref.][pp.~161-163]{RN5} on the elliptic function theory for signature~6.
As suggested by \cite{RN5}, Eqs.~\ref{eq:P_sixth_a}, \ref{eq:P_sixth_b} can be constructed from first principles, using   a hypergeometric transformation that identifies $P_{-1/6}(\cos\theta)={_2F_1}\left( \left.\begin{smallmatrix} 1/6,5/6\\1 \end{smallmatrix}\right|\sin^2\frac{\theta}{2} \right),0\leq\theta\leq\pi/2$ with the Fricke-Klein function $ _2F_1\left( \left.\begin{smallmatrix} 1/12,5/12\\1 \end{smallmatrix}\right|\sin^2\theta\right),0\leq\theta\leq\pi/2$  \cite[Ref.][p.~334]{FrickeKlein}, the latter of which has a well-known connection to complete elliptic integrals of the first kind~\cite{KontsevichZagier}.

We may combine Eqs.~\ref{eq:P_sixth_a} and \ref{eq:P_sixth_b} into a single formula\begin{align}P_{-1/6}\left( \frac{x(9-x^{2})}{(3+x^{2})^{3/2}} \right)=\sqrt{\frac{\sqrt{3+x^2}}{2}}P_{-1/2}(x)=\frac{2}{\pi}\sqrt{\frac{\sqrt{3+x^2}}{2}}\mathbf K\left( \sqrt{\frac{1-x}{2}} \right),\quad x\in(-1,1],\label{eq:P_sixth_unified}\end{align}a form that is handier for conversion between integrals involving $ P_{-1/6}$ and $\mathbf K $.

Ramanujan's studies of the elliptic function theory for signature~3  bring us Eqs.~\ref{eq:P_third_a}  and \ref{eq:P_third_b}. One may consult \cite[Ref.][pp.~112-114]{RN5} for the detailed theoretical arguments leading to the discovery of the algebraic relations between $ P_{-1/3}$ and $\mathbf K$.
 \qed\end{enumerate}\end{proof}

According to the recursion relation for Legendre functions $ (x^2-1)\D P_\nu(x)/\D x=\nu x P_\nu(x)-\nu P_{\nu-1}(x)$ and the symmetry $ P_{\nu}(x)=P_{-\nu-1}(x)$, one can relate some  fractional degree Legendre functions to complete elliptic integrals of  the first and the second kinds, such as \begin{align*}P_{1/2}(\cos\theta)={}&\frac{2}{\pi}\left[ 2 \mathbf E\left( \sin\frac{\theta}{2} \right)-\mathbf K\left( \sin\frac{\theta}{2} \right)\right],\quad P_{3/2}(\cos\theta)=\frac{2}{3\pi} \left[ 8\cos\theta   \mathbf E\left( \sin\frac{\theta}{2} \right)- (4 \cos \theta +1)\mathbf K\left( \sin\frac{\theta}{2} \right)\right],\notag\\P_{1/4}(\cos\theta)={}&\frac{2}{\pi}\left[ 2\sqrt{\smash[b]{1+\sin(\theta/2)}} \mathbf E\left( \sqrt{\frac{2\sin(\theta/2)}{1+\sin(\theta/2)}} \right)-\frac{1}{\sqrt{\smash[b]{1+\sin(\theta/2)}}}\mathbf K\left( \sqrt{\frac{2\sin(\theta/2)}{1+\sin(\theta/2)}} \right)\right],\notag\\P_{3/4}(\cos\theta)={}&\frac{2}{3\pi}\left[ 2\sqrt{\smash[b]{1+\sin(\theta/2)}} \mathbf E\left( \sqrt{\frac{2\sin(\theta/2)}{1+\sin(\theta/2)}} \right)-\frac{1-2\cos\theta}{\sqrt{\smash[b]{1+\sin(\theta/2)}}}\mathbf K\left( \sqrt{\frac{2\sin(\theta/2)}{1+\sin(\theta/2)}} \right)\right].\end{align*}\subsection{Hansen-Heine Scaling Limits and Hobson Coupling Formula}
By taking scaling limits of certain identities  for Legendre polynomials, one may extend their validity to Legendre functions of arbitrary complex degree. In this manner, one can compute some multiple elliptic integrals by interpolating Legendre polynomial identities, the latter of which often have clear motivations from the harmonic analysis on spheres.

As we may recall,  for positive integers $ m$, the associated Legendre functions of degree $ \nu$ and order $m$ are defined by \begin{align*}&&P^m_\nu(x):={}&(-1)^m(1-x^2)^{m/2}\frac{\D^m P_\nu(x)}{\D x^m},& P^{-m}_\nu(x):={}&(-1)^{m}\frac{\Gamma(\nu-m+1)}{\Gamma(\nu+m+1)}P^m_\nu(x);&&\notag\\&&Q^m_\nu(z):={}&(z^2-1)^{m/2}\frac{\D^m Q_\nu(z)}{\D z^m},&Q^{-m}_\nu(z):={}&\frac{\Gamma(\nu-m+1)}{\Gamma(\nu+m+1)}Q^m_\nu(x)\end{align*} One also writes $ P_\nu^0=P_\nu^{\vphantom{0}}$ and $ Q_\nu^0=Q_\nu^{\vphantom{0}}$. The spherical harmonics \[Y_{\ell m}(\theta,\phi):=\sqrt{\frac{2\ell+1}{4\pi}\frac{(\ell-m)!}{(\ell+m)!}}P^m_\ell(\cos\theta)e^{im\phi},\quad \ell\in\mathbb Z_{\geq0};m\in\mathbb Z\cap[-\ell,\ell]\] carry the Condon-Shortley
phase $ \overline {Y_{\ell m}(\theta,\phi)}=(-1)^m Y_{\ell,-m}(\theta,\phi)$.

In  the following, we show that the studies of Legendre functions of arbitrary degree can be built on   the Mehler-Dirichlet formula for spherical harmonics (see item 8.927 in \cite{GradshteynRyzhik} or \S4.5.4 in \cite{MOS}), which couples two normal vectors  $ \bm n_{1}(\theta_{1},\phi_{1})=(\sin\theta_{1}\cos\phi_{1},\sin\theta_{1}\sin\phi_{1},\cos\theta_{1})$ and  $ \bm n_{2}(\theta_{2},\phi_{2})=(\sin\theta_{2}\cos\phi_{2},\sin\theta_{2}\sin\phi_{2},\cos\theta_{2})$ on the unit sphere:  \begin{align}\label{eq:Mehler_Dirichlet}4\pi\sum _{\ell=0}^{\infty } \sum_{m=-\ell}^\ell\frac{\overline{Y_{\ell m}(\theta_{1},\phi_{1})}Y_{\ell m}(\theta_{2},\phi_{2})\cos[(\ell+\frac{1}{2})\beta]}{2\ell+1}=\sum _{\ell=0}^{\infty }P_\ell(\bm n_1\cdot\bm n_2) \cos[(\ell+\tfrac{1}{2})\beta]=\frac{\theta_{H}(\cos\beta-\bm n_1\cdot\bm n_2)}{\sqrt{2(\cos\beta-\bm n_1\cdot\bm n_2)}},\quad |\bm n_{1}|=|\bm n_{2}|=1.\end{align}Here, the Heaviside theta function is defined as $ \theta_H(x)=(1+\frac{x}{|x|})/2, x\in \mathbb R\smallsetminus\{0\}$, and the range of the angle $ \beta$ is $ (0,\pi)$.
\begin{lemma}[Hobson Coupling Formula]\label{lm:Legendre_interpolation}

 For $\theta_1,\theta_2\in[0,\pi),\nu\in\mathbb C$, the Legendre function of the first kind satisfies the Hobson coupling formula
\begin{align}
\frac{1}{2\pi}\int_0^{2\pi}P_\nu(\cos\theta_1\cos\theta_2+\sin\theta_1\sin\theta_2\cos\phi)\D\phi=\begin{cases}P_{\nu}(\cos\theta_1)P_\nu(\cos\theta_2), & \theta_{1}+\theta_2\leq\pi\ \\
P_{\nu}(-\cos\theta_1)P_\nu(-\cos\theta_2), & \theta_{1}+\theta_2\geq\pi
\end{cases}\label{eq:LegendreP_nu}
\end{align}  which specializes to  the following  integral identities for $ \theta_1,\theta_2\in[0,\pi)$:
\begin{align}
\frac{1}{4}\int_0^{2\pi}\mathbf K\left( \sin\frac{\varTheta}{2} \right)\D\phi={}&\begin{cases}\mathbf K\left( \sin\frac{\theta_1}{2} \right)\mathbf K\left( \sin\frac{\theta_2}{2} \right), & \theta_{1}+\theta_2\leq\pi \\[6pt]
\mathbf K\left( \cos\frac{\theta_1}{2} \right)\mathbf K\left( \cos\frac{\theta_2}{2} \right), & \theta_{1}+\theta_2\geq\pi
\end{cases}\label{eq:LegendreP_half}\\\frac{1}{4}\int_0^{2\pi}\mathbf K\left( \sqrt{\frac{2\sin(\varTheta/2)}{1+\sin(\varTheta/2)}} \right)\frac{\D\phi}{\sqrt{1+\sin(\varTheta/2)}}={}&\begin{cases}\frac{1}{\sqrt{\smash[b]{1+\sin(\theta_{1}/2)}}\sqrt{\smash[b]{1+\sin(\theta_{2}/2)}}}\mathbf K\left( \sqrt{\frac{2\sin(\theta_{1}/2)}{1+\sin(\theta_{1}/2)}} \right)\mathbf K\left( \sqrt{\frac{2\sin(\theta_{2}/2)}{1+\sin(\theta_{2}/2)}} \right), & \theta_{1}+\theta_2\leq\pi \\[6pt]
\frac{1}{\sqrt{\smash[b]{1+\cos(\theta_{1}/2)}}\sqrt{\smash[b]{1+\cos(\theta_{2}/2)}}}\mathbf K\left( \sqrt{\frac{2\cos(\theta_{1}/2)}{1+\cos(\theta_{1}/2)}} \right)\mathbf K\left( \sqrt{\frac{2\cos(\theta_{2}/2)}{1+\cos(\theta_{2}/2)}} \right), & \theta_{1}+\theta_2\geq\pi
\end{cases}\label{eq:LegendreP_quarter}
\end{align}
where $ \cos\varTheta=\cos\theta_1\cos\theta_2+\sin\theta_1\sin\theta_2\cos\phi$. In particular, we have
\begin{align}\left[ \mathbf K\left( \sin\frac{\theta}{2} \right) \right]^2={}&\int_0^{\pi/2}\mathbf K(\sin\theta\cos\phi)\D\phi,\!\!\!&\forall \theta\in&\left[0,\frac{\pi}{2}\right],\label{eq:LegendreP_half_spec}\tag{\ref{eq:LegendreP_half}$ ^\dagger$}\\ \mathbf K\left( \sin\frac{\theta}{2} \right) \mathbf K\left( \cos\frac{\theta}{2} \right)={}&\int_0^{\pi/2}\mathbf K\left( \sqrt{1-\sin\vphantom{1}^{2}\theta\cos\vphantom{1}^{2}\phi} \right)\D\phi,\!\!\!& \forall \theta\in&\left[0,\frac{\pi}{2}\right],\label{eq:LegendreP_half_spec'}\tag{\ref{eq:LegendreP_half}$ ^{\ddagger}$}\\\frac{1}{1+\sqrt{\vphantom{1}u}}\left[\mathbf K\left( \sqrt{\frac{2\sqrt{\vphantom{1}u}}{1+\sqrt{\vphantom{1}u}}} \right)\right]^2={}&\int_0^{\pi/2}\mathbf K\left( \sqrt{\frac{4\sqrt{u(1-u)}\cos\phi}{1+2\sqrt{u(1-u)}\cos\phi}} \right)\frac{\D\phi}{\sqrt{\smash[b]{1+2\sqrt{u(1-u)}\cos\phi}}},& \forall u\in&\left[0,\frac{1}{2}\right],\tag{\ref{eq:LegendreP_quarter}$ ^\dagger$}\label{eq:LegendreP_quarter_spec}\\\frac{\sqrt{2}}{1+\sqrt{\vphantom{1}u}}\mathbf K\left( \sqrt{\frac{2\sqrt{\vphantom{1}u}}{1+\sqrt{\vphantom{1}u}}} \right)\mathbf K\left( \sqrt{\frac{1-\sqrt{\vphantom{1}u}}{1+\sqrt{\vphantom{1}u}}} \right)={}&\int_0^{\pi/2}\mathbf K\left( \sqrt{\vphantom{\frac{\sqrt2}{\sqrt2}}\smash{\frac{2\sqrt{\smash[b]{1-4u(1-u)\cos^{\raisebox{0em}{$_2$}}\!\phi}}}{1+\sqrt{\smash[b]{1-4u(1-u)\cos^{\raisebox{0em}{$_2$}}\!\phi}}}}} \right)\frac{\D\phi}{\sqrt{\vphantom{{\vec W}}\smash{1+\sqrt{\smash[b]{1-4u(1-u)\cos^{\raisebox{0em}{$_2$}}\!\phi}}}}},\!\!\!& \forall u\in&\left( 0,1 \right).\tag{\ref{eq:LegendreP_quarter}$ ^{\ddagger}$}\label{eq:LegendreP_quarter_spec'}\end{align}\end{lemma}
\begin{proof}When $ \nu=n$ is a non-negative integer, one can verify Eq.~\ref{eq:LegendreP_nu} by resorting to Eq.~\ref{eq:Mehler_Dirichlet}:\begin{align*}&\frac{1}{2\pi}\int_0^{2\pi}P_n(\cos\theta_1\cos\theta_2+\sin\theta_1\sin\theta_2\cos\phi)\D\phi\notag\\={}&\frac{1}{2\pi}\int_0^{2\pi}\left[ \frac{2}{\pi}\int_0^\pi\frac{\theta_{H}(\cos\beta-\cos\theta_1\cos\theta_2-\sin\theta_1\sin\theta_2\cos\phi)\cos[(n+\frac{1}{2})\beta]\D\beta}{\sqrt{\cos\beta-\cos\theta_1\cos\theta_2-\sin\theta_1\sin\theta_2\cos\phi}} \right]\D\phi\notag\\={}&\frac{1}{2\pi}\int_0^{2\pi}\left[\frac{2}{\pi}\int_0^\pi4\pi\sum _{\ell=0}^{\infty } \sum_{m=-\ell}^\ell\frac{\overline{Y_{\ell m}(\theta_{1},0)}Y_{\ell m}(\theta_{2},\phi)\cos[(\ell+\frac{1}{2})\beta]\cos[(n+\frac{1}{2})\beta]}{2\ell+1}\D\beta\right]\D\phi=P_{n}(\cos\theta_1)P_n(\cos\theta_2),\end{align*}where spherical harmonic modes other than $ Y_{n0}$ will have vanishing contribution upon integration in the angular variables  $ \beta$ and $ \phi$. Here, due to the parity relation $ P_n(-x)=(-1)^nP_n(x)$ for non-negative integers $n$, one always has $P_{n}(\cos\theta_1)P_n(\cos\theta_2) =P_{n}(-\cos\theta_1)P_n(-\cos\theta_2)$, irrespective of the sum $ \theta_1+\theta_2$.

Fixing two positive  parameters $s $ and $t $, one may consider a neighborhood of the origin in the complex $ z$-plane such that $ \R\cos(s z)>0,\R\cos(tz)>0,\R[\cos(sz)\cos(tz)+\sin(sz)\sin(tz)\cos\phi]>0$. Then, given an arbitrary $ \nu\in\mathbb C$, we can apply the Laplace integral representation for the Legendre functions (Eq.~\ref{eq:Laplace_Pnu}) to evaluate the following scaling limit:\begin{align*}f_{s,t;\nu}(z):={}&P_{\nu/z}(\cos(sz))P_{\nu/z}(\cos(tz))-\int_0^{2\pi}P_{\nu/z}(\cos(sz)\cos(tz)+\sin(sz)\sin(tz)\cos\phi)\frac{\D\phi}{2\pi}\notag\\={}&\int_0^{2\pi}e^{\frac{\nu}{z}\log(\cos(sz)+i\sin(sz)\cos\phi_{1})}\frac{\D\phi_{1}}{2\pi}\int_0^{2\pi}e^{\frac{\nu}{z}\log(\cos(tz)+i\sin(tz)\cos\phi_{2})}\frac{\D\phi_{2}}{2\pi}\notag\\&-\int_0^{2\pi}\left\{\int_0^{2\pi}e^{\frac{\nu}{z}\log(\cos(sz)\cos(tz)+\sin(sz)\sin(tz)\cos\phi+i\sqrt{1-[\cos(sz)\cos(tz)+\sin(sz)\sin(tz)\cos\phi]^{2}}\cos\phi_3)}\frac{\D\phi_3}{2\pi}\right\}\frac{\D\phi}{2\pi}\notag\\\to{}&\int_0^{2\pi}e^{i\nu s\cos\phi_{1}}\frac{\D\phi_{1}}{2\pi}\int_0^{2\pi}e^{i\nu t\cos\phi_{2}}\frac{\D\phi_{2}}{2\pi}-\int_0^{2\pi}\left\{\int_0^{2\pi}e^{i\nu\sqrt{\smash[b]{s^{2}+t^2-2st\cos\phi}}\cos\phi_3}\frac{\D\phi_3}{2\pi}\right\}\frac{\D\phi}{2\pi}=0,\quad \text{as }z\to0,\end{align*}where the last equality turns out to be a special case of the Sonine-Gegenbauer formula for Bessel functions  \cite[Ref.][\S12.1]{Watson1995Bessel}:
\begin{align*}J_{0}(\nu s)J_0(\nu t)=\frac{1}{2\pi}\int_0^{2\pi}J_0(\nu\sqrt{\smash[b]{s^{\raisebox{0em}{$_2$}}+t^{\raisebox{0em}{$_2$}}-2st\cos\phi}})\D\phi.\end{align*}

 Thus, with fixed parameters $ s$, $t$ and $ \nu$, the function $ f_{s,t;\nu}(z)$ is analytic near the origin in the complex $ z$-plane, and $ f_{s,t;\nu}(z)=0$ on the set  $ \{0\}\cup\{\nu/n|n\in\mathbb Z\cap[0,+\infty)\}$, which contains an accumulation point $ \{0\}$ in the domain of analyticity for $ f_{s,t;\nu}(z)$. Therefore, the function $ f_{s,t;\nu}(z)$ must vanish identically in a certain neighborhood of  the origin, the spatial extent of which depends on the nature of the positive parameters $ s$ and $t$.

Now suppose that $ s+t<\pi$, then $ \cos s\cos t+\sin s\sin t\cos\phi\geq\cos(s+t)>-1$. Accordingly,  for a complex number $ z$ sitting in a certain open neighborhood of the unit interval $ [0,1]$, the expression $ \cos(sz)\cos(tz)+\sin(sz)\sin(tz)\cos\phi$ will miss the value $ -1$, which is the logarithmic branch point of the Legendre functions $ P_\mu$ of non-integer degree $ \mu\notin\mathbb Z$. Hence, when $s>0$, $t>0$ and $s+t<\pi$, we have a univalent complex-analytic function $ f_{s,t;\nu}(z)$ in an open  neighborhood of the unit interval $ [0,1]$, where resides a non-isolated set of zeros. By the principle of analytic continuation, we may conclude that $ f_{s,t;\nu}(1)=0$ for $ s>0$, $ t>0$ and $ s+t<\pi$, which extends, by continuity, to the $ s+t=\pi$ scenario as well as those situations where  $st=0$. This proves  Eq.~\ref{eq:LegendreP_nu} for the cases of $ \theta_1+\theta_2\leq\pi$. The rest of  Eq.~\ref{eq:LegendreP_nu} follows from the invariance of the integral\begin{align*}\frac{1}{2\pi}\int_0^{2\pi}P_\nu(\cos\theta_1\cos\theta_2+\sin\theta_1\sin\theta_2\cos\phi)\D\phi\end{align*}under the equatorial reflections $ (\theta_1,\theta_2)\mapsto(\pi-\theta_1,\pi-\theta_2)$.

Setting $ \nu=-1/2$ in  Eq.~\ref{eq:LegendreP_nu}, we obtain  Eq.~\ref{eq:LegendreP_half}; setting $ \nu=-1/4$ or $ \nu=-3/4$ in  Eq.~\ref{eq:LegendreP_nu}, one verifies   Eq.~\ref{eq:LegendreP_quarter}. When $ \theta_1=\theta_2$, Eq.~\ref{eq:LegendreP_half}  (resp.~\ref{eq:LegendreP_quarter}) becomes Eq.~\ref{eq:LegendreP_half_spec} (resp.~\ref{eq:LegendreP_quarter_spec}).
When $ \theta_1=\pi-\theta_2=\theta$, one may specialize Eq.~\ref{eq:LegendreP_half} into  Eq.~\ref{eq:LegendreP_half_spec'}. A similar procedure on    Eq.~\ref{eq:LegendreP_quarter} would bring us the evaluation\begin{align*}&\int_0^{\pi/2}\mathbf K\left( \sqrt{\vphantom{\frac{\sqrt2}{\sqrt2}}\smash{\frac{2\sqrt{\smash[b]{1-4u(1-u)\cos^{\raisebox{0em}{$_2$}}\!\phi}}}{1+\sqrt{\smash[b]{1-4u(1-u)\cos^{\raisebox{0em}{$_2$}}\!\phi}}}}} \right)\frac{\D\phi}{\sqrt{\vphantom{{\vec W}}\smash{1+\sqrt{\smash[b]{1-4u(1-u)\cos^{\raisebox{0em}{$_2$}}\!\phi}}}}}\notag\\={}&\frac{1}{\sqrt{1+\sqrt{\vphantom{1}u}}\sqrt{1+\sqrt{1-u}}}\mathbf K\left( \sqrt{\frac{2\sqrt{\vphantom{1}u}}{1+\sqrt{\vphantom{1}u}}} \right)\mathbf K\left( \sqrt{\frac{2\sqrt{1-u}}{1+\sqrt{1-u}}} \right),\quad \forall u\in\left( 0,\frac12 \right],\end{align*}   while the substitution $ u=(1-t)^2/(1+t)^2$ and Landen's transformation \cite[Ref.][item 8.126]{GradshteynRyzhik} allow us to see that\begin{align*}\frac{1}{\sqrt{1+\sqrt{1-u}}}\mathbf K\left( \sqrt{\frac{2\sqrt{1-u}}{1+\sqrt{1-u}}} \right)=\frac{\sqrt{1+t}}{1+\sqrt{t}}\mathbf K\left( \frac{2\sqrt[4] t}{1+\sqrt{t}} \right)=\sqrt{1+t}\mathbf K(\sqrt{t})=\sqrt{\frac{2}{1+\sqrt{\vphantom{1}u}}}\mathbf K\left( \sqrt{\frac{1-\sqrt{\vphantom{1}u}}{1+\sqrt{\vphantom{1}u}}} \right).\end{align*}This verifies Eq.~\ref{eq:LegendreP_quarter_spec'} for $u\in(0,1/2]$. Furthermore, the right-hand side of  Eq.~\ref{eq:LegendreP_quarter_spec'}  is intact under the transformation $ u\mapsto1-u$, whereas the symmetric extension of its left-hand side hinges on the identity\begin{align}
{\frac{1}{1+\sqrt{\vphantom{1}u}}}\mathbf K\left( \sqrt{\frac{1-\sqrt{\vphantom{1}u}}{1+\sqrt{\vphantom{1}u}}} \right)\mathbf K\left( \sqrt{\frac{2\sqrt{\vphantom{1}u}}{1+\sqrt{\vphantom{1}u}}} \right)={\frac{1}{1+\sqrt{1-u}}}\mathbf K\left( \sqrt{\frac{1-\sqrt{1-u}}{1+\sqrt{1-u}}} \right)\mathbf K\left( \sqrt{\frac{2\sqrt{1-u}}{1+\sqrt{1-u}}} \right),\quad 0<u<1,\label{eq:PuP1_u_symm}
\end{align}which is an equivalent formulation for the products of two Landen's transformations:
\begin{align}\frac{1+t}{2}\mathbf K(\sqrt{1-t})\mathbf K(\sqrt{t})=\frac{1+t}{(1+\sqrt{t})^{2}}\mathbf K\left( \frac{1-\sqrt{t}}{1+\sqrt{t}} \right)\mathbf K\left( \frac{2\sqrt[4]{t}}{1+\sqrt{t}} \right),\quad 0<t<1,\tag{\ref{eq:PuP1_u_symm}$'$}\end{align}
with the  correspondence of variables being $u=(1-t)^2/(1+t)^2 $.
\qed\end{proof}\begin{remark}\begin{enumerate}[label=(\arabic*)]\item Using Eq.~\ref{eq:Mehler_Dirichlet}, we may readily adapt our proof of the product formula in Eq.~\ref{eq:LegendreP_nu}  to associated Legendre functions:\begin{align}
\frac{1}{2\pi}\frac{\Gamma(\nu+m+1)}{\Gamma(\nu- m+1)}\int_0^{2\pi}P_\nu(\cos\theta_1\cos\theta_2+\sin\theta_1\sin\theta_2\cos\phi)\cos m\phi\D\phi=\begin{cases}P_{\nu}^{m}(\cos\theta_1)P_\nu^{m}(\cos\theta_2), & \theta_{1}+\theta_2\leq\pi\ \\
P_{\nu}^{m}(-\cos\theta_1)P_\nu^{m}(-\cos\theta_2), & \theta_{1}+\theta_2\geq\pi
\end{cases}\tag{\ref{eq:LegendreP_nu}$ ^m$}\label{eq:LegendreP_nu_m}
\end{align}  where $ \theta_1,\theta_2\in[0,\pi)$, $ \nu\in\mathbb C$, $ m\in\mathbb Z$. Here, the  corresponding scaling limit is the Sonine-Gegenbauer formula\[J_{m}(\nu s)J_m(\nu t)=\frac{1}{2\pi}\int_0^{2\pi}J_0(\nu\sqrt{\smash[b]{s^{\raisebox{0em}{$_2$}}+t^{\raisebox{0em}{$_2$}}-2st\cos\phi}})\cos m\phi\D\phi.\] (Doubtlessly, the cases of\begin{align*}P^{1}_{-1/2}(\cos\theta)={}&\frac{2}{\pi\sin\theta}\left[ \mathbf E\left( \sin\frac{\theta}{2}\right) -\mathbf K\left( \sin\frac{\theta}{2}\right)\cos^2\frac{\theta}{2}\right],\notag\\P^{1}_{-1/4}(\cos\theta)=P^{1}_{-3/4}(\cos\theta)={}&\frac{\sqrt{\smash[b]{1+\sin(\theta/2)}}}{\pi\sin\theta}\left[ \mathbf E\left( \sqrt{\frac{2\sin(\theta/2)}{1+\sin(\theta/2)}}\right) -\mathbf K\left( \sqrt{\frac{2\sin(\theta/2)}{1+\sin(\theta/2)}}\right)\left( 1-\sin\frac{\theta}{2} \right)\right]\end{align*}  will then lead to  generalizations of Eqs.~\ref{eq:LegendreP_half} and \ref{eq:LegendreP_quarter} into integral formulae involving complete elliptic integrals of the second kind.) Effectively, we can combine Eq.~\ref{eq:LegendreP_nu_m}  with Eq.~\ref{eq:LegendreP_nu} to develop a Fourier  expansion in $ \phi$ for the function $ P_\nu(\cos\theta_1\cos\theta_2+\sin\theta_1\sin\theta_2\cos\phi)$ (necessarily under the constraints $ 0\leq\theta_1\leq\pi,0\leq\theta_2\leq\pi,0\leq\theta_1+\theta_2\leq\pi$),  which turns up as a uniformly convergent series formerly known to E. W. Hobson  \cite[Ref.][\S226]{Hobson1931}: \begin{align*}P_\nu(\cos\theta_1\cos\theta_2+\sin\theta_1\sin\theta_2\cos\phi)=P_{\nu}(\cos\theta_1)P_\nu(\cos\theta_2)+2\sum_{m=1}^\infty\frac{\Gamma(\nu-m+1)}{\Gamma(\nu+m+1)}P_{\nu}^m(\cos\theta_1)P_\nu^m(\cos\theta_2)\cos m\phi.\end{align*}Hence we refer to Eq.~\ref{eq:LegendreP_nu}  as the Hobson coupling formula. Hobson's original approach draws on the Fourier series of $ (z+\sqrt{z^2-1}\cos\phi)^\nu$ for complex-valued $\nu$, which  is  not a method based on scaling limits.\item The representation of Bessel functions $ J_m$ as the scaling limit of (associate) Legendre functions $ P^m_\nu$  is a special case of Hansen's formula \cite[Ref.][\S5.7]{Watson1995Bessel}. Heine extended the limit procedure so that one may connect Bessel functions $ Y_m$ and the (associate) Legendre functions $ Q^m_\nu$  in a similar fashion \cite[Ref.][\S5.71]{Watson1995Bessel}. The Hansen-Heine scaling limit procedure ($ m=0$) will be used again during the proof of Eq.~\ref{eq:T_nu_nu_CG}, namely, \[\int_{-1}^1 {[}P_{\nu}(x){]}^2P_{\nu}(-x)\D x=\lim_{z\to\nu}\frac{1+2\cos(\pi z)}{3}\frac{\pi\Gamma\left(\frac{z+1 }{2}\right) \Gamma\left(\frac{3 z+2 }{2}\right)}{\left[ \Gamma\left(\frac{1-z}{2}\right)\right]^2\left[ \Gamma\left(\frac{z+2 }{2}\right)\right]^3 \Gamma\left(\frac{3 z+3 }{2}\right)},\quad\nu\in\mathbb  C\] in Proposition~\ref{prop:gen_CG}.        \item
The  integral formula in Eq.~\ref{eq:LegendreP_half_spec'} is usually attributed to Glasser \cite{Glasser1976}, and has been mentioned recently in the work of Bailey \textit{et al.} \cite{Bailey2008}. We will recover Eq.~\ref{eq:LegendreP_half_spec'} later in Proposition~\ref{prop:Ramanujan} (see \S\ref{subsec:Ramanujan}), using a method independent of the Hobson coupling formula for Legendre functions.\item The Hobson coupling formula is sensitive to the geometric constraint $ 0\leq\theta_1+\theta_2\leq\pi$. Thus, it is not  possible to directly  extrapolate the formulae in this lemma to an integral representation for  $ [P_\nu(x)]^2,-1<x<0$ where $ \nu$ is arbitrary. Practically, this obstacle can be overcome in various ways. For $ \nu=-1/2,-1/3,-1/4,-1/6$, the angular restriction in the Hobson coupling formula  can be circumvented by using Ramanujan's  integral representation for $ [\mathbf K(k)]^2,0\leq k<1$ (see Proposition~\ref{prop:Ramanujan}). For generic $ \nu$, Hobson's approach yields an integral representation for $ P_{\nu}(x)P_\nu(-x),-1<x<1$ and $ [P_\nu(x)]^2,0\leq x<1$, and the finite Hilbert transform in  Proposition~\ref{prop:P_nu_T} (see \S\ref{subsec:Tricomi_comp}) will enable us to express $ [P_\nu(x)]^2-[P_\nu(-x)]^2,0\leq x<1$ in terms of $ P_{\nu}(\xi)P_\nu(-\xi),-1<\xi<1$.   \eor \end{enumerate}
\end{remark}

\section{\label{sec:G_CG}Generalized Clebsch-Gordan Integrals}As the $ C^2(-1,1)$ solutions to the  Legendre differential equation
of arbitrary degree  $ \nu\in\mathbb C\smallsetminus\mathbb Z_{<0}$:\begin{align}
\frac{\D}{\D x}\left[ (1-x^2 )\frac{\D f(x)}{\D x}\right]+\nu(\nu+1) f(x)=0\label{eq:LegendreDiff_nu}
\end{align}are exhausted by linear combinations of $ P_\nu(x)$ and $ Q_\nu(x)$, a special case in  Appell's theory of third order ordinary differential equations \cite{Appell1881} then tells us that the $ C^3(-1,1)$ solutions to \begin{align}
\frac{\D}{\D x}\left\{(1-x^2 )\frac{\D}{\D x}\left[ (1-x^2 )\frac{\D f(x)}{\D x}\right]+4\nu(\nu+1)(1-x^2 ) f(x)\right\}+4\nu(\nu+1)
xf(x)=0,\quad \nu\in\mathbb C\smallsetminus\mathbb Z_{<0}\label{eq:LegendreSqrDiff_nu}\end{align} belong to a three-dimensional space spanned by $ [P_\nu(x)]^2$, $ P_\nu(x)Q_\nu(x)$ and $ [Q_\nu(x)]^2$ (or by an equivalent basis set $ \{[P_\nu(x)]^2,P_\nu(x)P_\nu(-x),[P_\nu(-x)]^2\}$ when $ \nu\notin\mathbb Z$).

The  ``interpolation'' from Legendre polynomials $ P_\ell(x),\ell\in\mathbb Z_{\geq0}$ to non-integer degree Legendre functions  $ P_{\nu}(x),\nu\notin\mathbb Z$ and the ``scaling limit'' for large degrees will  enable us to  extend some Clebsch-Gordan integral formulae to generic Legendre functions, as elaborated in the proposition below.

\begin{proposition}[Generalized Clebsch-Gordan Integrals]\label{prop:gen_CG}
 Define \begin{align*}T_{\mu,\nu}:=\int_{-1}^1 P_\mu(x)P_\nu(x)P_\nu(-x)\D x,\quad \mu,\nu\in\mathbb C ,\end{align*}then we have the symmetry \begin{align}T_{\mu,\nu}=T_{-\mu-1,\nu}=T_{\mu,-\nu-1}=T_{-\mu-1,-\nu-1}, \label{eq:T_mu_nu_symm} \end{align}the recursion relation  \begin{align}
(\mu +1)^2  [(\mu+1)^{2} -(2 \nu+1 )^{2}]T_{\mu+1,\nu}-\mu ^2 [\mu^{2} -(2 \nu+1 )^{2}]T_{\mu-1,\nu}={}&\frac{4 (2 \mu +1) \sin (  \mu\pi ) \sin (  \nu\pi )}{\pi ^2},\label{eq:T_mu_nu_rec}
\end{align}the representation in terms of generalized hypergeometric series\begin{align}T_{\mu,\nu}={}&\frac{2}{\pi^2}\frac{\sin(\mu\pi)\sin(\nu\pi)}{\mu(\mu+1)}\left[ \, \frac{1}{\nu}\,{_4F_3}\left(\left.\begin{array}{c}
1,\frac{1-\mu}{2},\frac{\mu+2 }{2},-\nu \\[4pt]
\frac{2-\mu }{2},\frac{\mu+3 }{2},1-\nu \\
\end{array}\right| 1\right) -\frac{1}{\nu+1} \,{ _4F_3}\left(\left.\begin{array}{c}
1,\frac{1-\mu}{2},\frac{\mu+2 }{2},\nu+1\ \\[4pt]
\frac{2-\mu }{2},\frac{\mu+3 }{2},\nu+2 \\
\end{array}\right| 1\right) \right],\quad \mu,\nu\notin\mathbb  Z\label{eq:T_mu_nu_CG}\\T_{\mu,\nu}={}&\frac{2}{\pi}\frac{\sin(\mu\pi)\cos(\nu\pi)}{2\nu+1}\left[ \, \frac{1}{\mu}\,{_5F_4}\left(\left.\begin{array}{c}
\frac{1}{2},\frac{1}{2},-\frac{\mu }{2},-\nu ,1+\nu \\[4pt]
1,\frac{2-\mu }{2},\frac{1-2\nu}{2} ,\frac{2\nu+3}{2} \\
\end{array}\right| 1\right) -\frac{1}{\mu+1}\, {_5F_4}\left(\left.\begin{array}{c}
\frac{1}{2},\frac{1}{2},\frac{1+\mu}{2},-\nu ,1+\nu\ \\[4pt]
1,\frac{3+\mu}{2},\frac{1-2\nu}{2} ,\frac{3+2\nu}{2}
\end{array}\right| 1\right) \right],\quad \mu,\nu\notin\mathbb Z,\nu\neq-\frac{1}{2}\tag{\ref{eq:T_mu_nu_CG}$ ^*$}\label{eq:T_mu_nu_CG_star} \end{align}along with some generalized Clebsch-Gordan integral formulae:{\allowdisplaybreaks\begin{align}T_{\mu,n}={}&(-1)^{n}\frac{\pi   \Gamma\left(\frac{2n+1-\mu }{2}\right)\Gamma\left(\frac{2n+2+\mu }{2}\right)}{ \left[\Gamma\left(\frac{1-\mu}{2}\right)\right]^2 \left[\Gamma \left(\frac{\mu+2 }{2}\right)\right]^2 \Gamma\left(\frac{2n+2-\mu }{2}\right)\Gamma\left(\frac{2n+3+\mu }{2}\right)},\quad  \mu\notin\mathbb Z, n\in\mathbb Z_{\geq0}\label{eq:T_mu_n_CG}\tag{\ref{eq:T_mu_nu_CG}$_{(\mu,n)}$}\\T_{2m,\nu}={}&\frac{\pi\cos(\nu\pi)   \Gamma\left(\frac{2\nu+1-2m }{2}\right)\Gamma(\nu+1+m)}{ \left[\Gamma\left(\frac{1-2m}{2}\right)\right]^2 \left[\Gamma(m{+1})\right]^2 \Gamma(\nu+1-m)\Gamma\left(\frac{2\nu+3+2m }{2}\right)},\quad m\in\mathbb Z_{\geq0},\nu+\frac{1}{2}\notin \mathbb Z\tag{\ref{eq:T_mu_nu_CG}$ _{(2m,\nu)}$}\label{eq:T_2m_nu_CG}\\T_{2m,n+\frac{1}{2}}={}&\begin{cases}-\dfrac{(-1)^{n}}{\pi}\dfrac{   [\Gamma(m+\frac{1}{2})]^2\Gamma(m-n-\frac{1}{2})\Gamma(m+n+\frac{3}{2})}{  \left[\Gamma(m{+1})\right]^2 \Gamma(m-n)\Gamma(m+n+2)}, & m\in\mathbb Z_{\geq0},n\in \mathbb Z,m-n\notin\mathbb Z_{\leq 0},m+n+2\notin \mathbb Z_{\leq 0} \\
0, & m\in\mathbb Z_{\geq0}, (m-n\in\mathbb Z_{\leq 0}\text{ or } m+n+2\in \mathbb Z_{\leq 0})\\
\end{cases}\quad \tag{\ref{eq:T_mu_nu_CG}$_{(2m,n+\frac{1}{2})}$}\label{eq:T_2m_n_half_CG}\\T_{\nu,\nu}={}&\lim_{z\to\nu}\frac{1+2\cos(\pi z)}{3}\frac{\pi\Gamma\left(\frac{z+1 }{2}\right) \Gamma\left(\frac{3 z+2 }{2}\right)}{\left[ \Gamma\left(\frac{1-z}{2}\right)\right]^2\left[ \Gamma\left(\frac{z+2 }{2}\right)\right]^3 \Gamma\left(\frac{3 z+3 }{2}\right)},\quad\nu\in\mathbb  C\label{eq:T_nu_nu_CG}\tag{\ref{eq:T_mu_nu_CG}$ _{(\nu,\nu)}$}\\T_{2\nu-1,\nu}={}&\lim_{z\to\nu}\frac{\sin(\pi z)\sin(2\pi z)}{\pi^{2}z^2},\quad\nu\in\mathbb  C\label{eq:T_2nu1_nu_CG}\tag{\ref{eq:T_mu_nu_CG}$ _{(2\nu-1,\nu)}$}\\T_{2\nu+2,\nu}={}&-\lim_{z\to\nu}\frac{\sin(\pi z)\sin(2\pi z)}{\pi^{2}(z+1)^2},\quad\nu\in\mathbb  C.\label{eq:T_2nu2_nu_CG}\tag{\ref{eq:T_mu_nu_CG}$ _{(2\nu+2,\nu)}$}\end{align}}
\end{proposition}\begin{proof}The symmetry of Legendre functions with respect to degree $ P_{\nu}(x)=P_{-\nu-1}(x),\forall\nu\in\mathbb C$ explains the related property of $ T_{\mu,\nu}$  (Eq.~\ref{eq:T_mu_nu_symm}).

 Judging from the familiar recursion relations for Legendre functions\begin{align*}&&(2\mu+1)(1-x^{2})\frac{\D P_\mu(x)}{\D x}={}&\mu(\mu+1)[P_{\mu-1}(x)-P_{\mu+1}(x)];&(2\mu+1) x P_\mu(x)={}&(\mu+1)P_{\mu+1}(x)+\mu P_{\mu-1}(x)&&\end{align*}and the  differential equations given in Eqs.~\ref{eq:LegendreDiff_nu} and \ref{eq:LegendreSqrDiff_nu}, we may carry out the following computations: {\allowdisplaybreaks\begin{align*}&\nu(\nu+1)[(\mu+1)T_{\mu+1,\nu}+\mu T_{\mu-1,\nu}]\notag\\={}&\nu(\nu+1)\int_{-1}^{1}[(\mu+1)P_{\mu+1}(x)+\mu P_{\mu-1}(x)]P_{\nu}(x) P_\nu(-x)\D x=(2\mu+1)\nu(\nu+1)\int_{-1}^1x P_\mu(x)P_{\nu}(x) P_\nu(-x)\D x\notag\\={}&-\frac{2\mu+1}{4}\int_{-1}^1P_\mu(x)\frac{\D}{\D x}\left\{(1-x^2 )\frac{\D}{\D x}\left[ (1-x^2 )\frac{\D(P_{\nu}(x) P_\nu(-x)) }{\D x}\right]+4\nu(\nu+1)(1-x^2 )P_{\nu}(x) P_\nu(-x)\right\}\D x\notag\\={}&\frac{2\mu+1}{4}\int_{-1}^1\left\{(1-x^2 )\frac{\D}{\D x}\left[ (1-x^2 )\frac{\D(P_{\nu}(x) P_\nu(-x)) }{\D x}\right]+4\nu(\nu+1)(1-x^2 )P_{\nu}(x) P_\nu(-x)\right\}\frac{\D P_\mu(x)}{\D x}\D x\notag\\={}&\frac{2\mu+1}{4}\int_{-1}^1\left\{(1-x^2 )\frac{\D}{\D x}\left[ (1-x^2 )\frac{\D(P_{\nu}(x) P_\nu(-x)) }{\D x}\right]\right\}\frac{\D P_\mu(x)}{\D x}\D x+\mu(\mu+1)\nu(\nu+1)(T_{\mu-1,\nu}- T_{\mu+1,\nu})\notag\\={}&\frac{2\mu+1}{4}\left[\mu(\mu+1)\int_{-1}^1(1-x^2 )P_\mu(x)\frac{\D(P_{\nu}(x) P_\nu(-x)) }{\D x}{\D x}-\frac{4\sin(\mu\pi)\sin(\nu\pi)}{\pi^{2}}\right]+\mu(\mu+1)\nu(\nu+1)(T_{\mu-1,\nu}- T_{\mu+1,\nu}),\end{align*}}where we have used in the last line the properties $ P_\nu(1)=1,\forall\nu\in\mathbb C$ and \begin{align*}\lim_{x\to-1+0^{+}}(1-x^2)\frac{\D P_\nu(x)}{\D x}=\frac{2\sin(\nu\pi)}{\pi},\quad \forall\nu\in\mathbb C.\end{align*}We may then proceed with the simplification\begin{align*}&(2\mu+1)\int_{-1}^1(1-x^2 )P_\mu(x)\frac{\D(P_{\nu}(x) P_\nu(-x)) }{\D x}{\D x}\notag\\={}&-(2\mu+1)\int_{-1}^1(1-x^2 )P_{\nu}(x) P_\nu(-x)\frac{\D P_\mu(x)}{\D x}\D x+2(2\mu+1)\int_{-1}^1xP_\mu(x)P_{\nu}(x) P_\nu(-x)\D x\notag\\={}&\mu(\mu+1)(T_{\mu+1,\nu}-T_{\mu-1,\nu})+2(\mu+1)T_{\mu+1,\nu}+2\mu T_{\mu-1,\nu},\end{align*}so as to confirm the recursion formula of $ T_{\mu,\nu}$ (Eq.~\ref{eq:T_mu_nu_rec}).

For non-negative integers $ m$ and $n$ meeting the requirement $ m\leq 2n$, we might recall from  the standard Clebsch-Gordan theory  that \begin{align}\label{eq:T_m_n_CG}\tag{\ref{eq:T_mu_nu_CG}$_{(m,n)}$}T_{m,n}:={}&\int_{-1}^1P_m(x)P_n(x)P_n(-x)\D x=2(-1)^{n}\begin{pmatrix}m & n & n \\
0 & 0 & 0 \\
\end{pmatrix}^2=(-1)^{n}\frac{\pi   \Gamma\left(\frac{2n+1-m }{2}\right)\Gamma\left(\frac{2n+2+m }{2}\right)}{ \left[\Gamma\left(\frac{1-m}{2}\right)\right]^2 \left[\Gamma\left(\frac{m+2 }{2}\right)\right]^2 \Gamma\left(\frac{2n+2-m }{2}\right)\Gamma\left(\frac{2n+3+m }{2}\right)},\end{align}where $\left( \begin{smallmatrix}j_1&j_2& j_3\\m_1& m_2& m_3\end{smallmatrix}\right)$ is the Wigner $ 3$-$j$ symbol. The right-hand side expression in  Eq.~\ref{eq:T_m_n_CG} is regarded as zero if $m$ is a positive odd integer, as anticipated from the  relations  $ P_m(x)=-P_{m}(-x)$ and $\Gamma(\frac{1-m}{2})=\infty $ in such circumstances.  We have $T_{m,n}=\int_{-1}^1P_m(x)P_n(x)P_n(-x)\D x=0 $ for non-negative integers $ m$ and $n$ satisfying $ m>2n$, as indicated by the respective poles of the gamma factors in the denominator on the right-hand side of  Eq.~\ref{eq:T_m_n_CG}.

Actually, even without the prior knowledge of the standard Clebsch-Gordan theory, one can start from the initial condition $ T_{0,n}=\int_{-1}^1P_n(x)P_n(-x)\D x=2(-1)^{n}/(2n+1),n\in\mathbb Z_{\geq0}$, and build  Eq.~\ref{eq:T_m_n_CG} inductively from the recursion relation given by Eq.~\ref{eq:T_mu_nu_rec}.

The interpolation formula (known as Dougall's expansion in \cite[Ref.][p.~167]{HTF1})\begin{align*}P_{\nu}(x)={}&2\int_0^\pi\frac{\theta_{H}(\cos\beta-x)\cos\frac{(2\nu+1)\beta}{2}}{\sqrt{2(\cos\beta-x)}}\frac{\D\beta}{\pi}=2\int_0^\pi \sum_{\ell=0}^\infty P_\ell(x)\cos\frac{(2\ell+1)\beta}{2}\cos\frac{(2\nu+1)\beta}{2}\frac{\D\beta}{\pi}\notag\\={}&\sum^\infty_{\ell=0} P_\ell(x)\left[ \frac{\sin(\ell-\nu)\pi}{(\ell-\nu)\pi} +\frac{\sin(\ell+\nu+1)\pi}{(\ell+\nu+1)\pi}\right]\end{align*}allows us to sum  Eq.~\ref{eq:T_m_n_CG} as \begin{align*}T_{\mu,n}:={}&\int_{-1}^1P_\mu(x)P_n(x)P_n(-x)\D x=\sum_{\ell=0} ^\infty T_{\ell,n}\left[ \frac{\sin(\ell-\mu)\pi}{(\ell-\mu)\pi} +\frac{\sin(\ell+\mu+1)\pi}{(\ell+\mu+1)\pi}\right]\end{align*}for  $ \mu\in\mathbb C\smallsetminus\mathbb Z,n\in\mathbb Z_{\geq0}$. To identify the last term in the equation above with  the expression in Eq.~\ref{eq:T_mu_n_CG}, we note that the sum over $ \ell$  in fact truncates after finite terms, which reveals the expression \begin{align*}\frac{1}{\sin(\mu\pi)}\left\{(-1)^{n}\frac{\pi   \Gamma\left(\frac{2n+1-\mu }{2}\right)\Gamma\left(\frac{2n+2+\mu }{2}\right)}{ \left[\Gamma\left(\frac{1-\mu}{2}\right)\right]^2 \left[\Gamma\left(\frac{\mu+2 }{2}\right)\right]^2 \Gamma\left(\frac{2n+2-\mu }{2}\right)\Gamma\left(\frac{2n+3+\mu }{2}\right)}-\sum_{\ell=0} ^\infty T_{\ell,n}\left[ \frac{\sin(\ell-\mu)\pi}{(\ell-\mu)\pi} +\frac{\sin(\ell+\mu+1)\pi}{(\ell+\mu+1)\pi}\right]\right\}\end{align*}as a rational function of $\mu$, for whatever integer $n$. As  $ \mu$ approaches any integer $m$, it is clear that such a rational function tends to zero with $ O(\mu-m)$ convergence rate,  so it must be identically vanishing.

By Eq.~\ref{eq:LegendreP_nu} along with the Mehler-Dirichlet theory, we have the following computation when $ \nu$ is not an  integer:\begin{align*}P_\nu(x)P_\nu(-x)={}&\int_0^{2\pi}P_\nu((1-x^{2})\cos\phi-x^2)\frac{\D\phi}{2\pi}=2\int_0^{2\pi}\left[\int_0^\pi\frac{\theta_{H}(\cos\beta-(1-x^{2})\cos\phi+x^2)\cos\frac{(2\nu+1)\beta}{2}}{\sqrt{2(\cos\beta-(1-x^{2})\cos\phi+x^2)}}\frac{\D\beta}{\pi}\right]\frac{\D\phi}{2\pi}\notag\\={}&2\int_0^{2\pi}\left[\int_0^\pi \sum_{\ell=0}^\infty P_\ell((1-x^{2})\cos\phi-x^2)\cos\frac{(2\ell+1)\beta}{2}\cos\frac{(2\nu+1)\beta}{2}\frac{\D\beta}{\pi}\right]\frac{\D\phi}{2\pi}\notag\\={}&\sum^\infty_{\ell=0} P_\ell(x)P_\ell(-x)\left[ \frac{\sin(\ell-\nu)\pi}{(\ell-\nu)\pi} +\frac{\sin(\ell+\nu+1)\pi}{(\ell+\nu+1)\pi}\right],\end{align*}where we have integrated over  $ \beta$ and $ \phi$ separately in the last step. Thus, we can perform a decomposition\[T_{\mu,\nu}=\sum^\infty_{\ell=0}T_{\mu,\ell}\left[ \frac{\sin(\ell-\nu)\pi}{(\ell-\nu)\pi} +\frac{\sin(\ell+\nu+1)\pi}{(\ell+\nu+1)\pi}\right]=\frac{\sin(\nu\pi)(\tau_{\mu,\nu}+\tau_{\mu,-\nu-1})}{\left[\Gamma\left(\frac{1-\mu}{2}\right)\right]^2 \left[\Gamma\left(\frac{\mu+2 }{2}\right)\right]^2 },\]where \[\tau_{\mu,\lambda}:=\sum_{\ell=0}^\infty\frac{   \Gamma\left(\frac{2\ell+1-\mu }{2}\right)\Gamma\left(\frac{2\ell+2+\mu }{2}\right)}{ \Gamma\left(\frac{2\ell+2-\mu }{2}\right)\Gamma\left(\frac{2\ell+3+\mu }{2}\right)}\frac{1}{\lambda-\ell}=\frac{2\sin(\mu\pi)}{\lambda\pi^{2}}\left[\Gamma\left(\frac{1-\mu}{2}\right)\right]^2 \left[\Gamma\left(\frac{\mu+2 }{2}\right)\right]^2\,{_4F_3}\left(\left.\begin{array}{c}
1,\frac{1-\mu}{2},\frac{\mu+2 }{2},-\lambda\ \\[4pt]
\frac{2-\mu }{2},\frac{\mu+3 }{2},1-\lambda\ \\
\end{array}\right| 1\right)\]follows directly from the definition of $ _4F_3$. This verifies Eq.~\ref{eq:T_mu_nu_CG}.

One can prove  Eq.~\ref{eq:T_2m_nu_CG} by directly computing\begin{align}T_{0,\nu}={}&\frac{\sin(\nu\pi)}{\pi}\sum_{n=0}^\infty (-1)^{n}T_{0,n}\left( \frac{1}{\nu-n} -\frac{1}{\nu+n+1}\right)=\frac{\sin(\nu\pi)}{\pi}\sum_{n=0}^\infty \frac{2}{2n+1}\left( \frac{1}{\nu-n} -\frac{1}{\nu+n+1}\right)=\frac{2\cos(\nu\pi)}{2\nu+1}\label{eq:T_0_nu}\tag{\ref{eq:T_mu_nu_CG}$ _{(0,\nu)}$}\end{align}and invoking the recursion relation in Eq.~\ref{eq:T_mu_nu_rec}.
Once  Eq.~\ref{eq:T_2m_nu_CG} is given, and $ T_{2m+1,\nu}=0,m\in\mathbb Z_{\geq0}$ is obvious from symmetry considerations, we can interpolate $ T_{\mu,\nu}$ as follows:\begin{align*}T_{\mu,\nu}={}&\sum_{\ell=0}^\infty T_{\ell,\nu}\left[ \frac{\sin(\ell-\mu)\pi}{(\ell-\mu)\pi} +\frac{\sin(\ell+\mu+1)\pi}{(\ell+\mu+1)\pi}\right]=\frac{\sin(\mu\pi)}{\pi}\sum_{m=0}^\infty T_{2m,\nu}\left(\frac{1}{\mu-2m} -\frac{1}{2m+\mu+1}\right)\notag\\={}&\sin(\mu\pi)\cos(\nu\pi)\sum_{m=0}^\infty \frac{   \Gamma\left(\frac{2\nu+1-2m }{2}\right)\Gamma(\nu+1+m)}{ \left[\Gamma\left(\frac{1-2m}{2}\right)\right]^2 \left[\Gamma(m{+1})\right]^2 \Gamma(\nu+1-m)\Gamma\left(\frac{2\nu+3+2m }{2}\right)}\left( \frac{1}{\mu-2m} -\frac{1}{2m+\mu+1}\right),\end{align*} which in turn, directly accounts for the hypergeometric summation in Eq.~\ref{eq:T_mu_nu_CG_star} concerning $ _5F_4$. Exploiting  again Euler's reflection formula  $ \Gamma(z)\Gamma(1-z)=\pi/{\sin(\pi z)}$, one can rewrite the expression for $ T_{2m,\nu}$ as (taking appropriate limits when the resulting fractions assume  indeterminate forms)\begin{align*}T_{2m,\nu}=-\frac{\sin(\nu\pi)}{\pi}\frac{   [\Gamma(m+\frac{1}{2})]^2\Gamma(m-\nu)\Gamma(m+\nu+1)}{  \left[\Gamma(m{+1})\right]^2 \Gamma(m-\nu+\frac{1}{2})\Gamma(m+\nu+\frac{3}{2})},\quad m\in\mathbb Z_{\geq0},\end{align*}which specializes to Eq.~\ref{eq:T_2m_n_half_CG}.

To prove Eq.~\ref{eq:T_nu_nu_CG}, we show that the function \begin{align}\mathfrak T(z):=\frac{1}{\sin^2(\pi z)}\left\{T_{z,z}-\frac{1+2\cos(\pi z)}{3}\frac{\pi\Gamma\left(\frac{z+1 }{2}\right) \Gamma\left(\frac{3 z+2 }{2}\right)}{\left[ \Gamma\left(\frac{1-z}{2}\right)\right]^2\left[ \Gamma\left(\frac{z+2 }{2}\right)\right]^3 \Gamma\left(\frac{3 z+3 }{2}\right)}\right\}\equiv\mathfrak T(-z-1)\label{eq:T_van_def}\end{align}is analytic over the whole complex $ z$-plane, with asymptotic  behavior $ \mathfrak T(z)=O(N^{-9/4}),z\in C_N,N\to+\infty$. Here,    $N$ is a positive integer and the square contour   $ C_N$    has vertices  $(1-4N)/4-iN$, $(1+4N)/4-iN $, $(1+4N)/4+iN $ and $(1-4N)/4-iN $.

First, to verify that $ \mathfrak T(z)$ is an entire function, we only need to check that the expression inside the braces of Eq.~\ref{eq:T_van_def} has vanishing derivative whenever $z$ is an integer, so that the numerator of Eq.~\ref{eq:T_van_def} encounters (at least) a second-order zero at every integer $ z=n\in\mathbb Z$. At the positive odd integers $ z=2n+1,n\in\mathbb Z_{\geq0}$, one may directly compute\begin{align*}\left.\frac{\partial T_{z,z}}{\partial z}\right|_{z=2n+1}={}&2\int_{-1}^1\left.\frac{\partial P_{z}(x)}{\partial z}\right|_{z=2n+1}P_{2n+1}(x)P_{2n+1}(-x)\D x+\int_{-1}^1[P_{2n+1}(-x)]^{2}\left.\frac{\partial P_{z}(x)}{\partial z}\right|_{z=2n+1}\D x=\left.\frac{\partial T_{z,2n+1}}{\partial z}\right|_{z=2n+1} =0\intertext{from Eq.~\ref{eq:T_mu_n_CG}, which coincides with the behavior \begin{align*}\frac{1+2\cos(\pi z)}{3}\frac{\pi\Gamma\left(\frac{z+1 }{2}\right) \Gamma\left(\frac{3 z+2 }{2}\right)}{\left[ \Gamma\left(\frac{1-z}{2}\right)\right]^2\left[ \Gamma\left(\frac{z+2 }{2}\right)\right]^3 \Gamma\left(\frac{3 z+3 }{2}\right)}=O((z-(2n+1))^2)\end{align*}attributed to the factor $ [\Gamma((1-z)/2)]^2=O((z-(2n+1))^{-2})$.   Meanwhile, at the non-negative even integers  $ z=2n,n\in\mathbb Z_{\geq0}$,}\left.\frac{\partial T_{z,z}}{\partial z}\right|_{z=2n}={}&2\int_{-1}^1\left.\frac{\partial P_{z}(x)}{\partial z}\right|_{z=2n}P_{2n}(x)P_{2n}(-x)\D x+\int_{-1}^1[P_{2n}(-x)]^{2}\left.\frac{\partial P_{z}(x)}{\partial z}\right|_{z=2n}\D x=3\left.\frac{\partial T_{z,2n}}{\partial z}\right|_{z=2n} \notag\\={}&3T_{2n,2n}\left.\frac{\partial}{\partial z}\right|_{z=2n}\log\frac{   \Gamma\left(\frac{4n+1-z }{2}\right)\Gamma\left(\frac{4n+2+z }{2}\right)}{ \left[\Gamma\left(\frac{1-z}{2}\right)\right]^2 \left[\Gamma \left(\frac{z+2 }{2}\right)\right]^2 \Gamma\left(\frac{4n+2-z }{2}\right)\Gamma\left(\frac{4n+3+z }{2}\right)}\end{align*}is also a result of Eq.~\ref{eq:T_mu_n_CG}.
As we have \begin{align*}3\left.\frac{\partial}{\partial z}\right|_{z=\nu}\log\frac{   \Gamma\left(\frac{2\nu+1-z }{2}\right)\Gamma\left(\frac{2\nu+2+z }{2}\right)}{ \left[\Gamma\left(\frac{1-z}{2}\right)\right]^2 \left[\Gamma \left(\frac{z+2 }{2}\right)\right]^2 \Gamma\left(\frac{2\nu+2-z }{2}\right)\Gamma\left(\frac{2\nu+3+z }{2}\right)}-\left.\frac{\partial}{\partial z}\right|_{z=\nu}\log\frac{\Gamma\left(\frac{z+1 }{2}\right) \Gamma\left(\frac{3 z+2 }{2}\right)}{\left[ \Gamma\left(\frac{1-z}{2}\right)\right]^2\left[ \Gamma\left(\frac{z+2 }{2}\right)\right]^3 \Gamma\left(\frac{3 z+3 }{2}\right)}=-2\pi\tan\frac{\nu\pi}{2},\end{align*}which vanishes when $ \nu=2n$ is a non-negative even number,  it is clear that the numerator of Eq.~\ref{eq:T_van_def} has $ O((z-2n)^2)$ behavior for $ n\in\mathbb Z_{\geq0}$.

Then, we investigate the asymptotic behavior of  $ \mathfrak T(z)$ as $ z\to\infty$. For any point $ z$ sitting on the right half of the contour $  C_N$  (where $N$ is a large positive integer) satisfying $ \R z>-\frac{1}{2}$, we have a uniform asymptotic formula \begin{align}T_{ z, z}={}&\left[1+O\left(\frac{1}{N}\right)\right]\int_0^{\pi/2}\left\{J_0\left( \left( z+\frac{1}{2}\right)\theta \right)[1+\cos(\pi  z)]+Y_0\left( \left(z+\frac{1}{2}\right)\theta \right)\sin(\pi  z)\right\}J_0\left( \left(z+\frac{1}{2}\right)\theta \right)\times\notag\\&\times\left[J_0\left( \left(  z+\frac{1}{2}\right)\theta \right)\cos(\pi  z)+Y_0\left( \left(z+\frac{1}{2}\right)\theta \right)\sin(\pi  z)\right]\sqrt{\frac{\theta^{3}}{\sin\theta}}\D \theta.\label{eq:T_z_z_est}\end{align}Here, we have used the following transformations for $ -\pi<\arg\nu<\pi$:\begin{align*}T_{\nu,\nu}={}&\int_0^\pi [P_\nu(\cos\theta)]^{2}P_\nu(-\cos\theta)\sin\theta\D\theta=\int_0^{\pi/2}[P_\nu(\cos\theta)+P_\nu(-\cos\theta)]P_\nu(\cos\theta)P_\nu(-\cos\theta)\sin\theta\D\theta\notag\\={}&\int_0^{\pi/2}\left\{P_\nu(\cos\theta)[1+\cos(\nu\pi)]-\frac{2}{\pi}Q_\nu(\cos\theta)\sin(\nu\pi)\right\}P_\nu(\cos\theta)\left[P_\nu(\cos\theta)\cos(\nu\pi)-\frac{2}{\pi}Q_\nu(\cos\theta)\sin(\nu\pi)\right]\sin\theta\D\theta,\end{align*}along with the asymptotic formulae (see \cite[Ref.][Chap.~12, Eqs.~12.18 and 12.25]{Olver1974} or \cite[Ref.][Eqs.~43 and 46]{Jones2001}):\begin{align*}P_{\nu}(\cos\theta)=\sqrt{\frac{\theta}{\sin\theta}}J_0\left( \left(\nu+\frac{1}{2}\right)\theta \right)+O\left( \frac{1}{2\nu+1}  \right),\quad Q_\nu(\cos\theta)=-\frac{\pi}{2}\sqrt{\frac{\theta}{\sin\theta}}Y_0\left( \left(\nu+\frac{1}{2}\right)\theta \right)+O\left( \frac{1}{2\nu+1}  \right)\end{align*}where the error bounds are uniform for the angular variables $ \theta\in(0,\pi/2]$, so long as $ |\nu|\to\infty,-\pi<\arg \nu<\pi$ \cite{Olver1974,Jones2001}. For the case with  degree   $ \nu=(1+4N)/4$ where $ N$ is a large positive integer, we  split the integral on the right-hand side of   Eq.~\ref{eq:T_z_z_est} into two parts:\begin{align*}\int_0^{\pi/2}(\cdots)\D\theta=\int_0^{1/\sqrt{2\nu+1}}(\cdots)\D\theta+\int_{1/\sqrt{2\nu+1}}^{\pi/2}(\cdots)\D\theta.\end{align*}For the first portion, we may deduce\begin{align*}&\int_0^{1/\sqrt{2\nu+1}}\left\{J_0\left( \left( \nu+\frac{1}{2}\right)\theta \right)[1+\cos(\nu\pi)]+Y_0\left( \left(\nu+\frac{1}{2}\right)\theta \right)\sin(\nu\pi  )\right\}J_0\left( \left(\nu+\frac{1}{2}\right)\theta \right)\times\notag\\&\times\left[J_0\left( \left(  \nu+\frac{1}{2}\right)\theta \right)\cos(\nu\pi  )+Y_0\left( \left(\nu+\frac{1}{2}\right)\theta \right)\sin(\nu\pi  )\right]\sqrt{\frac{\theta^{3}}{\sin\theta}}\D\theta=\frac{8[1+\cos(\nu\pi)][1+2\cos(\nu\pi)]}{3\sqrt{3}\pi(1+2\nu)^2}\left[1+O\left( \frac{1}{\sqrt[4]{2\nu+1}} \right)\right]\end{align*}from three integral formulae valid for $ a>0$ \cite[Ref.][items~2.12.42.25, 2.13.22.11, 2.13.25.1]{PBMVol2}:\footnote{The first integral here belongs to the  generalized Weber-Schafheitlin type \cite[Ref.][\S13.46]{Watson1995Bessel}. The proof of the second integral can be found in \cite{GervoisNavelet1984}. The third integral can be derived from the first one, in conjunction with    Nicholson's formula \cite[Ref.][\S13.73]{Watson1995Bessel} $ [J_0(x)]^2+[Y_0(x)]^2=\frac{8}{\pi^2}\int_0^{\infty}K_0(2x\sinh t)\D t$ and a special case of the modified Weber-Schafheitlin integral \cite[Ref.][\S13.45]{Watson1995Bessel} $ \int_0^\infty J_{0}(x)K_0(2x\sinh t)x\D x=(1+4\sinh^2 t)^{-1}$.}\begin{align*}\int_0^\infty [J_0(ax)]^3x\D x=\frac{2}{\sqrt{3}\pi a^2},\quad \int_0^\infty [J_0(ax)]^2 Y_0(ax)x\D x=0,\quad \int_0^\infty J_0(ax)[Y_0(ax)]^2x\D  x=\frac{2}{3\sqrt{3}\pi a^2},\end{align*} as well as the familiar asymptotic behavior of Bessel functions for large arguments $x\gg1 $:\begin{align*}J_{0}(x)\sim\sqrt{\frac{2}{\pi x}}\cos\left( x-\frac{\pi}{4} \right),\quad Y_0(x)\sim\sqrt{\frac{2}{\pi x}}\sin\left( x-\frac{\pi}{4} \right),\end{align*}which allows us to verify \begin{align*}&\int^\infty_{1/\sqrt{2\nu+1}}\left\{J_0\left( \left( \nu+\frac{1}{2}\right)x \right)[1+\cos(\nu\pi)]+Y_0\left( \left(\nu+\frac{1}{2}\right)x \right)\sin(\nu\pi  )\right\}J_0\left( \left(\nu+\frac{1}{2}\right)x\right)\times\notag\\&\times\left[J_0\left( \left(  \nu+\frac{1}{2}\right)x \right)\cos(\nu\pi  )+Y_0\left( \left(\nu+\frac{1}{2}\right)x \right)\sin(\nu\pi  )\right]x\D x=O\left( \frac{1}{(2\nu+1)^{9/4}} \right)\end{align*}via integration by parts.\footnote{To wit, we  need to estimate \begin{align*}&\int_{1/\sqrt{2\nu+1}}^\infty\frac{\sin(a(2\nu+1)x+b)}{\sqrt{x}}\D x=\frac{\sin b}{a\sqrt{2\nu+1}}\int_{\sqrt{2\nu+1}}^\infty\frac{1}{\sqrt{y}}\frac{\partial}{\partial y}\sin(ay)\D y-\frac{\cos b}{a\sqrt{2\nu+1}}\int_{\sqrt{2\nu+1}}^\infty\frac{1}{\sqrt{y}}\frac{\partial}{\partial y}\cos(ay)\D y\notag\\={}&\frac{\cos(a\sqrt{2\nu+1}+b)}{a(2\nu+1)^{3/4}} +\frac{\sin b}{2a\sqrt{2\nu+1}}\int_{\sqrt{2\nu+1}}^\infty\frac{\sin(ay)}{y^{3/2}}\D y-\frac{\cos b}{2a\sqrt{2\nu+1}}\int_{\sqrt{2\nu+1}}^\infty\frac{\cos(ay)}{y^{3/2}}\D y=O\left( \frac{1}{(2\nu+1)^{3/4}} \right),\end{align*}where the coefficients $ a,b\in\mathbb R $ arise from the product-to-sum formulae of trigonometric functions applied to asymptotic expansions for the products of Bessel functions.} Meanwhile, using the  asymptotic forms of Bessel functions and integration by parts, one can also verify that\footnote{One might wish to compare this with the direct application of a hypergeometric asymptotic formula for Legendre functions (see \cite[Ref.][p.~295]{Hobson1931} or \cite[Ref.][p.~162]{HTF1}):\begin{align*}P_{\nu}(\cos\theta)=\frac{\Gamma(\nu+1)}{\Gamma(\nu+\frac{3}{2})}\sqrt{\frac{2}{\pi\sin\theta}}\left[ \cos\left( \left(\nu+\frac{1}{2}\right)\theta-\frac{\pi}{4} \right)+O\left( \frac{1}{\nu} \right) \right],\quad \text{where }\theta\in[\varepsilon,\pi-\varepsilon],\varepsilon>0.\end{align*}}\begin{align*}&\int^{\pi/2}_{1/\sqrt{2\nu+1}}\left\{J_0\left( \left( \nu+\frac{1}{2}\right)\theta \right)[1+\cos(\nu\pi)]+Y_0\left( \left(\nu+\frac{1}{2}\right)\theta \right)\sin(\nu\pi  )\right\}J_0\left( \left(\nu+\frac{1}{2}\right)\theta \right)\times\notag\\&\times\left[J_0\left( \left(  \nu+\frac{1}{2}\right)\theta \right)\cos(\nu\pi  )+Y_0\left( \left(\nu+\frac{1}{2}\right)\theta \right)\sin(\nu\pi  )\right]\sqrt{\frac{\theta^{3}}{\sin\theta}}\D\theta\notag\\\sim&\left[ \frac{4}{\pi(2\nu+1)} \right]^{3/2}\int^{\pi-(1/\sqrt{2\nu+1})}_{1/\sqrt{2\nu+1}}\cos^2\left(  \left( \nu+\frac{1}{2}\right)\theta-\frac{\pi}{4} \right)\cos\left(  \left( \nu+\frac{1}{2}\right)(\pi-\theta)-\frac{\pi}{4} \right)\frac{\D\theta}{\sqrt{\sin\theta}}=O\left( \frac{1}{(2\nu+1)^{9/4}} \right).\end{align*} Thus,  the asymptotic expansion\begin{align*}T_{\nu,\nu}=\frac{8[1+\cos(\nu\pi)][1+2\cos(\nu\pi)]}{3\sqrt{3}\pi(1+2\nu)^2}\left[1+O\left( \frac{1}{\sqrt[4]{2\nu+1}} \right)\right]\end{align*}
is established for  $\nu=(1+4N)/4>0$.
For an arbitrary point $ z$ satisfying $ \R z>-1/2$, we may modify the  aforementioned procedure  by deforming the integral path joining the points $ 0$ and $ \pi/2$ into the sum of two line segments:\begin{align*}\int_0^{\pi/2}(\cdots)\D\theta=\int_0^{\sqrt{|2z+1|}/(2z+1)}(\cdots)\D\theta+\int_{\sqrt{|2z+1|}/(2z+1)}^{\pi/2}(\cdots)\D\theta,\end{align*}followed by the integrations $ \int_0^\infty(\cdots)\D x$  and $  \int_{\sqrt{|2z+1|}/(2z+1)}^\infty(\cdots)\D x$ along  contours that run to infinity in such a manner that $ (2z+1)x\to+\infty$ along the positive real axis. This results in\begin{subequations}
\begin{align}T_{z,z}={}&\frac{8[1+\cos(\pi z)][1+2\cos(\pi z)]}{3\sqrt{3}\pi(1+2z)^2}\left[1+O\left( \frac{1}{\sqrt[4]{N}} \right)\right],\quad z\in C_N,\R z>-\frac{1}{2},\label{eq:T_z_z_asympt_contour_right}\intertext{where the error bound is uniform along the contour $ C_N$, on which both  $ |\tan(\pi z)|$ and $ |\cot(\pi z)|$ are uniformly bounded.
After reflection $z\mapsto-z-1 $, we obtain}T_{z,z}={}&\frac{8[1-\cos(\pi z)][1-2\cos(\pi z)]}{3\sqrt{3}\pi(1+2z)^2}\left[1+O\left( \frac{1}{\sqrt[4]{N}} \right)\right],\quad z\in C_N,\R z<-\frac{1}{2}.\label{eq:T_z_z_asympt_contour_left}\end{align}   \end{subequations}We may combine Eqs.~\ref{eq:T_z_z_asympt_contour_right} and \ref{eq:T_z_z_asympt_contour_left} into a single formula\begin{align*}T_{z,z}=\frac{1+2\cos(\pi z)}{3}\frac{\pi\Gamma\left(\frac{z+1 }{2}\right) \Gamma\left(\frac{3 z+2 }{2}\right)}{\left[ \Gamma\left(\frac{1-z}{2}\right)\right]^2\left[ \Gamma\left(\frac{z+2 }{2}\right)\right]^3 \Gamma\left(\frac{3 z+3 }{2}\right)}\left[1+O\left( \frac{1}{\sqrt[4]{N}} \right)\right],\quad z\in C_N,\end{align*}which also implies the  bound estimate $\mathfrak T(z)=O(N^{-9/4})$ uniformly applicable to all the points  $z$  on the rectangular contour $ C_N $, thanks to the inequalities \[\sup_{z\in C_N} |\cot(\pi z)|\leq\max(1,\coth\pi)=\coth\pi,\quad\sup _{z\in C_N} \frac{1}{|\sin(\pi z)|}\leq\max\left(\sqrt{2},\frac{1}{\sinh\pi}\right)=\sqrt{2},\]for  $N\in\mathbb Z_{>0}$. By Cauchy's integral formula,  we  have the identity \begin{align*}\mathfrak T(\nu)={}&\lim_{N\to\infty}\oint_{C_N}\frac{\mathfrak T(z)}{z-\nu}\frac{\D z}{2\pi i}=0,\quad \forall\nu\in\mathbb C,\end{align*}which proves Eq.~\ref{eq:T_nu_nu_CG}.

Both Eqs.~\ref{eq:T_2nu1_nu_CG} and \ref{eq:T_2nu2_nu_CG} follow directly from the recursion relation (Eq.~\ref{eq:T_mu_nu_rec}).
 \qed
\end{proof}

\begin{remark}  As in our proof of    Eq.~\ref{eq:T_nu_nu_CG}, we can use Bessel functions to establish  the following asymptotic behavior for $ |z|\to\infty,\R z>-1/2$:\begin{align*}T_{\rho z,z}\sim{}&\frac{8}{\pi}\frac{\cos(\rho\pi z+\pi z  )+\cos(\pi z  )}{\rho\sqrt{4-\rho^{2}}(1+2z)^{2}}+\frac{16}{\pi^{2}}\frac{\sin(\rho\pi z)\sin(\pi z)}{\rho\sqrt{4-\rho^{2}}(1+2z)^{2}}\left(\pi-\arcsin\sqrt{1-\frac{\rho^{2}}{4}}\right),&& 0<\rho<2;\\T_{\rho z,z}\sim{}&\frac{8}{\pi}\frac{\sin(\rho\pi z-\pi z  )-\sin(\pi z)}{\rho\sqrt{\rho^{2}-4}(1+2z)^{2}}-\frac{16}{\pi^{2}}\frac{\sin(\rho\pi z)\sin(\pi z)}{\rho\sqrt{\rho^{2}-4}(1+2z)^{2}}\sinh^{-1}\sqrt{\frac{\rho^{2}}{4}-1},&&\rho>2.\intertext{along with an asymptotic reflection formula $ T_{\rho(-z-1),-z-1}\sim e^{i\pi(1-\rho)\arg z/|\arg z|}T_{\rho z,z},0<|\arg z|<\pi$ for $ \rho>0,|z|\to\infty$. Put differently, for large $ |z|$ and $ -\pi<\arg z<\pi$, we have an ``asymptotic trigonometric modulation'' of the standard Clebsch-Gordan integral formulae:}T_{\rho z,z}\sim{}&\frac{\pi  \Gamma \left(\frac{1+(2-\rho)z}{2}\right) \Gamma \left(\frac{(\rho  +2)z+2}{2}\right)}{\left[\Gamma \left(\frac{1-\rho z}{2}\right)\right]^2 \left[\Gamma \left(\frac{2+ \rho z }{2}\right)\right]^2\Gamma \left(\frac{2+(2-\rho)  z}{2}\right)  \Gamma \left(\frac{(\rho+2)  z+3}{2}\right)}\times\notag\\&\times\left[ \frac{\cos(\rho\pi z+\pi z  )+\cos(\pi z  )}{1+\cos(\rho\pi z)} +\frac{2}{\pi}\frac{\sin(\rho\pi z)\sin(\pi z)}{1+\cos(\rho\pi z)}\left(\pi-\arcsin\sqrt{1-\frac{\rho^{2}}{4}}\right)\right],&& 0<\rho<2;\notag\\T_{\rho z,z}\sim{}&\frac{\pi  \Gamma \left(\frac{1+(2-\rho)z}{2}\right) \Gamma \left(\frac{(\rho  +2)z+2}{2}\right)}{\left[\Gamma \left(\frac{1-\rho z}{2}\right)\right]^2 \left[\Gamma \left(\frac{2+ \rho z }{2}\right)\right]^2\Gamma \left(\frac{2+(2-\rho)  z}{2}\right)  \Gamma \left(\frac{(\rho+2)  z+3}{2}\right)}\times\notag\\&\times\left[ \frac{\sin(\rho\pi z-\pi z  )-\sin(\pi z)}{1+\cos(\rho\pi z)} -\frac{2}{\pi}\frac{\sin(\rho\pi z)\sin(\pi z)}{1+\cos(\rho\pi z)}\sinh^{-1}\sqrt{\frac{\rho^{2}}{4}-1}\right]\cot\frac{(\rho-2)\pi z}{2},&&\rho>2.\end{align*}For $ \rho\neq1$, the right-hand side of the penultimate (resp.~last) asymptotic formula diverges at $ z=-2/(2+\rho)$ (resp.~$ z=-1/\rho$), so the asymptotic analysis does not result in exact  forms of $ T_{\rho z,z}$ ($ \rho\neq1$) for  finitely sized $|z|$.  \eor\end{remark}\begin{corollary}[Some Multiple Elliptic Integrals in Clebsch-Gordan Forms]We have the integral formulae{\allowdisplaybreaks
\begin{align}\frac{\pi^2}{4}T_{0,-1/2}={}&2\int_0^1\mathbf K(\sqrt{t})\mathbf K(\sqrt{1-t})\D t=\frac{\pi^{3}}{4},\label{eq:T_zero_minus_half}\\\frac{\pi^2}{4}T_{0,1/2}={}&2\int_0^1[2\mathbf E(\sqrt{t})-\mathbf K(\sqrt{t})][2\mathbf E(\sqrt{1-t})-\mathbf K(\sqrt{1-t})]\D t=0,\label{eq:T_zero_half}\\
\frac{\pi^{3}}{8}T_{-1/2,-1/2}={}&2\int_{0}^1[\mathbf K(\sqrt{1-t})]^2\mathbf K(\sqrt{t})\D t=\frac{[\Gamma(\frac{1}{4})]^{8}}{192\pi^{2}},\label{eq:T_half_half}\\\frac{\pi^{3}}{8}T_{-1/3,-1/3}={}&\frac{27}{2}\int_0^1\frac{ (1-p^2) p(2+p)}{\sqrt{3+6 p} (1+p+p^2)}\left[ \mathbf K\left( \sqrt{\frac{p^3(2+p)}{1+2p}} \right) \right]^{2}\mathbf K\left( \sqrt{1-\frac{p^3(2+p)}{1+2p}} \right)\D p\notag\\={}&\frac{27}{6}\int_0^1\frac{ (1-p^2) p(2+p)}{\sqrt{1+2 p} (1+p+p^2)}\left[ \mathbf K\left( \sqrt{1-\frac{p^3(2+p)}{1+2p}} \right) \right]^{2}\mathbf K\left( \sqrt{\frac{p^3(2+p)}{1+2p}} \right)\D p=\frac{3\sqrt{3}[\Gamma(\frac13)]^9}{256\pi^2},\label{eq:T_third_third}\\\frac{\pi^{3}}{8}T_{-1/4,-1/4}={}&\int_0^1\frac{4}{(1+\sqrt{\vphantom{1}u})^{3/2}}\left[\mathbf K\left( \sqrt{\frac{1-\sqrt{\vphantom{1}u}}{1+\sqrt{\vphantom{1}u}}} \right)\right]^2\mathbf K\left( \sqrt{\frac{2\sqrt{\vphantom{1}u}}{1+\sqrt{\vphantom{1}u}}} \right)\D u=4\sqrt{2}\int_{0}^1\frac{(1-t)[\mathbf K(\sqrt{t})]^2\mathbf K(\sqrt{1-t})}{(1+t)^{3/2}}\D t\notag\\={}&\int_0^1\left(\frac{2}{1+\sqrt{\vphantom{1}u}}\right)^{3/2}\left[\mathbf K\left( \sqrt{\frac{2\sqrt{\vphantom{1}u}}{1+\sqrt{\vphantom{1}u}}} \right)\right]^2\mathbf K\left( \sqrt{\frac{1-\sqrt{\vphantom{1}u}}{1+\sqrt{\vphantom{1}u}}} \right)\D u=4\int_{0}^1\frac{(1-t)[\mathbf K(\sqrt{1-t})]^2\mathbf K(\sqrt{t})}{(1+t)^{3/2}}\D t=\frac{[\Gamma(\frac{1}{8})\Gamma(\frac{3}{8})]^{2}}{24},\label{eq:T_quarter_quarter}\\\frac{\pi^{3}}{8}T_{-1/6,-1/6}={}&\frac{27}{2\sqrt{2}}\int_{-1}^1\left[\mathbf K\left( \sqrt{\frac{1-x}{2}} \right)\right]^2\mathbf K\left( \sqrt{\frac{1+x}{2}} \right)\frac{1-x^2}{(3+x^2)^{7/4}}\D x=\frac{27}{4}\int_{0}^1\frac{t(1-t)[\mathbf K(\sqrt{1-t})]^2\mathbf K(\sqrt{t})}{(1-t+t^{2})^{7/4}}\D t=\frac{[\Gamma (\frac{1}{4})]^{4}}{8\sqrt{2\sqrt{3}}},\label{eq:T_sixth_sixth}\\\frac{\pi^{3}}{8}T_{1/2,-3/4}={}&2\sqrt{2}\int_0^1\frac{2\mathbf E(\sqrt{\vphantom{1}u})-\mathbf K(\sqrt{\vphantom{1}u})}{1+\sqrt{\vphantom{1}u}}\mathbf K\left( \sqrt{\frac{2\sqrt{\vphantom{1}u}}{1+\sqrt{\vphantom{1}u}}} \right)\mathbf K\left( \sqrt{\frac{1-\sqrt{\vphantom{1}u}}{1+\sqrt{\vphantom{1}u}}} \right)\D u=\sqrt{2}\pi,\label{eq:T_half_quarter}\\\frac{\pi^{3}}{8}T_{3/2,-1/4}={}&\frac{2\sqrt{2}}{3}\int_0^1\frac{8(1-2u)\mathbf E(\sqrt{\vphantom{1}u})-(5-8u)\mathbf K(\sqrt{\vphantom{1}u})}{1+\sqrt{\vphantom{1}u}}\mathbf K\left( \sqrt{\frac{2\sqrt{\vphantom{1}u}}{1+\sqrt{\vphantom{1}u}}} \right)\mathbf K\left( \sqrt{\frac{1-\sqrt{\vphantom{1}u}}{1+\sqrt{\vphantom{1}u}}} \right)\D u=-\frac{\sqrt{2}\pi}{9}. \label{eq:T_sesqui_quarter}
\end{align}}\end{corollary}\begin{proof}Special cases of Eqs.~\ref{eq:T_2m_n_half_CG}, \ref{eq:T_nu_nu_CG} and \ref{eq:T_2nu2_nu_CG} lead to Eqs.~\ref{eq:T_zero_minus_half}, \ref{eq:T_zero_half}, \ref{eq:T_half_half}, \ref{eq:T_third_third}, \ref{eq:T_quarter_quarter}, \ref{eq:T_sixth_sixth}, \ref{eq:T_half_quarter} and \ref{eq:T_sesqui_quarter}.\qed\end{proof}

\section{\label{sec:rotations}Spherical Rotations and Multiple Elliptic Integrals}
In this section, we shall focus on the transformations of multiple elliptic integrals using variable substitutions motivated by rotations on a unit sphere. In effect, our attention will be momentarily restricted to two  special Legendre functions $P_{-1/2} $ and $ P_{-1/4}$.

The transformation methods in this section not only provide evaluations of certain Cauchy principal values:\begin{align*}\mathcal P\int_{-1}^1\frac{ \mathbf K(\sqrt{(1-\xi)/2})\D\xi}{\pi(x-\xi)\sqrt{1+\xi}},\quad -1<x<1;\qquad\mathcal P\int_{-1}^1\mathbf K\left( \sqrt{\frac{1+\xi}{2}} \right)\mathbf K\left( \sqrt{\frac{1-\xi}{2}} \right)\frac{ 2\D \xi}{\pi (x-\xi)},\quad -1<x<1\quad\text{etc.}\end{align*}
which will be used in \S\ref{sec:Tricomi},
but also pave way for further applications in our  subsequent works on elliptic integrals.
\subsection{Beltrami Rotations}
 The technique of spherical rotation, which  traces back to Beltrami's work on the Abel integral equations \cite[Ref.][p.~328]{Beltrami1880}, has found applications in the geometrically-motivated evaluation of certain definite integrals related to Bessel functions (see \cite[Ref.][\S3.33, \S12.12 and \S12.14]{Watson1995Bessel}). The  lemma below is a simple realization of the Beltrami rotation in the case of  the complete elliptic integrals $ \mathbf K(k)$ and $ \mathbf E(k)$.\begin{lemma}[Beltrami Transformations]\label{lm:sph_rot}\begin{enumerate}[label=\emph{(\alph*)}, ref=(\alph*), widest=a]\item We have the following integral identities\begin{align}\mathbf K(k)={}&\frac{2}{\pi}\int_0^1\frac{\mathbf K(\sqrt{1-\kappa^2})\D \kappa}{1-k^2\kappa^2},& 0<k<1,\label{eq:Beltrami}\\\mathbf K(\xi)={}&\frac{2}{\pi}\int_0^1\frac{\sqrt{1-\xi^{2}}\mathbf K(\sqrt{1-\kappa^2})\D \kappa}{1-\xi^2(1-\kappa^2)},& 0<\xi<1,\label{eq:mB}\tag{\ref{eq:Beltrami}$'$}\\\mathbf K(r)={}&\frac{2}{\pi}\int_0^1\frac{(1+r)\mathbf K(\sqrt{1-\kappa^2})\D \kappa}{(1+r)^{2}-4r\kappa^2},& 0<r<1,\label{eq:LB}\tag{\ref{eq:Beltrami}$_L$}\\\mathbf K(\eta)={}&\frac{2}{\pi}\int_0^1\frac{(1-\eta)\mathbf K(\sqrt{1-\kappa^2})\D \kappa}{(1-\eta)^{2}+4\eta\kappa^2},& 0<\eta<1.\label{eq:mLB}\tag{\ref{eq:Beltrami}$'_L$}\end{align}\item For $ 0<k<1$, we have \begin{align}\frac{\mathbf K(k)-\mathbf E(k)}{k^{2}}={}&\frac{2}{\pi}\int_{0}^1\frac{\mathbf E(\sqrt{1-\kappa^{2}})\D\kappa}{1-k^{2}\kappa^{2}},\label{eq:KE_D_B}\\\mathbf E(k)={}&\frac{4(1-k^2)}{\pi}\int_{0}^1\frac{\mathbf E(\sqrt{1-\kappa^{2}})\D\kappa}{(1-k^{2}\kappa^{2})^2},\label{eq:EB}\\\frac{\mathbf E(k)}{\sqrt{1-k^{2}}}={}&\frac{4(1-k^2)}{\pi}\int_{0}^1\frac{\mathbf E(\sqrt{1-\kappa^{2}})\D\kappa}{[1-k^{2}(1-\kappa^{2})]^2}.\label{eq:EB'}\tag{\ref{eq:EB}$'$}\end{align} \end{enumerate}

\end{lemma}\begin{proof}\begin{enumerate}[label=(\alph*), widest=a]\item\label{itm:Beltrami_a} We start by writing the complete elliptic integral of the first kind as an integral on the unit sphere $ S^2$: \begin{align*}\mathbf K(k)=\frac{1}{2}\int_0^\pi\frac{\D\theta}{\sqrt{1-k^2\sin^2\theta}}=\frac{1}{4\pi}\int_0^\pi\sin\theta\D\theta\int_0^{2\pi}\D\phi\frac{1}{\sin\theta(1-k\sin\theta\cos\phi)}=\int_{S^2}\frac{1}{\sqrt{1-Z^2}(1-kX)}\frac{\D\sigma}{4\pi}.\end{align*} As we rotate about the $ Y$-axis by a right angle, effectively swapping the r\^oles of  the $X$- and $Z$-axes, we may rewrite the  integral above as\begin{align*}\mathbf K(k)={}&\int_{S^2}\frac{1}{\sqrt{1-X^2}(1-kZ)}\frac{\D\sigma}{4\pi}=\int_{S^2}\frac{1}{\sqrt{1-X^2}(1-k^{2}Z^{2})}\frac{\D\sigma}{4\pi}=\frac{1}{4\pi}\int_0^\pi\sin\theta\D\theta\int_0^{2\pi}\D\phi\frac{1}{\sqrt{\smash[b]{1-\sin^{2}\theta\cos^{2}\phi}}(1-k^{2}\cos^{2}\theta)},\end{align*}where we have combined the values of the integrand at points  $ \pm Z$ and restored the spherical coordinates. Integrating over the azimuthal angle $ \phi$, we obtain\begin{align*}\mathbf K(k)=\frac{1}{\pi}\int_0^\pi\frac{\mathbf K(\sin\theta)\sin\theta\D\theta}{1-k^{2}\cos^{2}\theta},\end{align*}which is equivalent to Eq.~\ref{eq:Beltrami}.

  Now that we have proved the relation\begin{align*}\mathbf K(k)=\frac{2}{\pi}\int_0^1\frac{\mathbf K(\sqrt{1-\kappa^2})}{1-k^2\kappa^2}\D\kappa,\quad 0<k<1,\end{align*}which analytically continues to $ k\in\mathbb C\smallsetminus[1,+\infty)$, we may derive Eqs.~\ref{eq:mB}, \ref{eq:LB} and \ref{eq:mLB}, respectively,  from the imaginary modulus transformation $ \mathbf K(\xi)=\mathbf K(i\xi/\sqrt{\smash[b]{1-\xi^{2}}})/\sqrt{\smash[b]{1-\xi^{2}}}$, Landen's transformation $ \mathbf K(r)=\mathbf K(2\sqrt{r}/(1+r))/(1+r)$, and a combination thereof $ \mathbf K(\eta)=\mathbf K(2i\sqrt{\smash[b]{\eta}}/(1-\eta))/(1-\eta)$.

\item

We adapt our derivation of part~\ref{itm:Beltrami_a}  to accommodate for the complete elliptic integrals of the second kind $ \mathbf E(k),0<k<1$:\begin{align*}\frac{\mathbf K(k)-\mathbf E(k)}{k^{2}}={}&\frac{1}{2}\int_0^\pi\frac{\sin^2\theta\D\theta}{\sqrt{1-k^2\sin^2\theta}}=\frac{1}{4\pi}\int_0^\pi\sin\theta\D\theta\int_0^{2\pi}\D\phi\frac{\sin\theta}{1-k\sin\theta\cos\phi}=\int_{S^2}\frac{\sqrt{1-Z^2}}{1-kX}\frac{\D\sigma}{4\pi}\notag\\={}&\int_{S^2}\frac{\sqrt{1-X^2}}{1-kZ}\frac{\D\sigma}{4\pi}=\int_{S^2}\frac{\sqrt{1-X^2}}{1-k^{2}Z^{2}}\frac{\D\sigma}{4\pi}=\frac{1}{4\pi}\int_0^\pi\sin\theta\D\theta\int_0^{2\pi}\D\phi\frac{\sqrt{\smash[b]{1-\sin^{2}\theta\cos^{2}\phi}}}{1-k^{2}\cos^{2}\theta}=\frac{1}{\pi}\int_{0}^\pi\frac{\mathbf E(\sin\theta)\sin\theta\D\theta}{1-k^{2}\cos^{2}\theta},\end{align*}which verifies  Eq.~\ref{eq:KE_D_B}. As we multiply both sides of Eq.~\ref{eq:KE_D_B} by $ k^2$, and differentiate in $k$, we obtain the formula stated in Eq.~\ref{eq:EB}:\begin{align*}\frac{k\mathbf E(k)}{1-k^{2}}=\frac{\D }{\D k}[\mathbf K(k)-\mathbf E(k)]=\frac{2}{\pi}\frac{\D }{\D k}\int_{0}^1\frac{\mathbf E(\sqrt{1-\kappa^{2}})k^{2}\D\kappa}{1-k^{2}\kappa^{2}}=\frac{4k}{\pi}\int_{0}^1\frac{\mathbf E(\sqrt{1-\kappa^{2}})\D\kappa}{(1-k^{2}\kappa^{2})^2},\quad 0<k<1.\end{align*}Now, by analytic continuation, we have \begin{align}\mathbf E(k)=\frac{4(1-k^2)}{\pi}\int_{0}^1\frac{\mathbf E(\sqrt{1-\kappa^{2}})\D\kappa}{(1-k^{2}\kappa^{2})^2},\quad k\in\mathbb C\smallsetminus[1,+\infty).\tag{\ref{eq:EB}$^*$}\label{eq:EB_star}\end{align}In particular, by the imaginary modulus transformation for complete elliptic integrals of the second kind \cite[Ref.][item~160.02]{ByrdFriedman}, we obtain\begin{align*}\sqrt{\smash[b]{1+\xi^2}}\mathbf E\left( \frac{\xi}{\sqrt{1+\xi^2}} \right)=\mathbf E(i\xi)=\frac{4(1+\xi^2)}{\pi}\int_{0}^1\frac{\mathbf E(\sqrt{1-\kappa^{2}})\D\kappa}{(1+\xi^{2}\kappa^{2})^2},\quad \xi>0,\end{align*}  which is  equivalent to Eq.~\ref{eq:EB'}. \qed\end{enumerate}\end{proof}
\begin{remark}
Admittedly, the spherical rotation is not the only road to any or all of the identities in  Eqs.~\ref{eq:Beltrami}, \ref{eq:mB}, \ref{eq:LB} and \ref{eq:mLB}. For example, one may still verify Eq.~\ref{eq:Beltrami}   via the method of moments and hypergeometric summations \cite{Wan2012}. The same can be said for several formulae that we will encounter in Proposition~\ref{prop:Ramanujan}. However, it is our hope that the spherical rotations provide clearer geometric motivations than a purely combinatorial approach.

The author thanks an anonymous referee for pointing out a recent contribution of V. Anghel \cite{Anghel2013}, which generalizes the Beltrami rotation to hyperspheres embedded in Euclidean spaces of arbitrary dimensions. It is highly probable that the hyperspherical analog of Beltrami rotation can lead to geometric proofs of various integral identities related to elliptic integrals, which we hope to pursue in a separate publication. \eor\end{remark}

Beltrami transformations allow us to convert one multiple elliptic integral into another. For instance, the following chain of identities \begin{align*}\int_{0}^1\mathbf K(\xi)\D \xi=\int_0^1\frac{\mathbf K(\sqrt{1-\kappa^2})}{1+\kappa}\D\kappa=\int_0^1\frac{\mathbf K(\xi)\xi}{1+\sqrt{1-\xi^{2}}}\frac{\D\xi}{\sqrt{1-\xi^{2}}}=\frac{2}{\pi}\int_0^1\mathbf K\left(\sqrt{1-\kappa^2}\right)\left[ \frac{\kappa\arccos\kappa}{\sqrt{1-\kappa^2}}-\log(2\kappa) \right]\D\kappa\end{align*}result from two applications of Eq.~\ref{eq:mB}, and the transformations\begin{align*}\int_0^1\mathbf E(k)\D k=\frac{2}{\pi}\int_0^1\frac{\mathbf E(\sqrt{1-\kappa^2})}{\kappa^{3}}[-\kappa+(1+\kappa^{2})\tanh^{-1}\kappa] \D\kappa=\int_{0}^{1}\frac{(2+\kappa)\mathbf E(\sqrt{1-\kappa^2}) }{(1+\kappa)^{2}}\D \kappa\end{align*}  can be justified by  Eqs.~\ref{eq:EB} and \ref{eq:EB'}.
More examples will be described  elsewhere.

In the next corollary, we mention a few consequences of Beltrami transformations that are pertinent to later developments in the current work.\begin{corollary}[Some Applications of Beltrami Transformations]\label{cor:Beltrami_app}\begin{enumerate}[label=\emph{(\alph*)}, ref=(\alph*), widest=a]\item For $ 0<u<1$, we have the following duality relations for multiple elliptic integrals:\begin{align}\int_0^{\sqrt{\vphantom{1}u}}\frac{\mathbf K(k)\D k}{\sqrt{1-k^2}\sqrt{u-k^2}}={}&\int_0^1\frac{k\mathbf K(k)\D k}{\sqrt{1-k^2}\sqrt{1-k^2 u}},\label{eq:Q'_duality}\\\int^1_{\sqrt{\vphantom{1}u}}\frac{k\mathbf K(k)\D k}{\sqrt{1-k^2}\sqrt{k^2-u}}={}&\int_0^1\frac{\mathbf K(\sqrt{1-\kappa^2})\D\kappa}{\sqrt{1-\kappa^2}\sqrt{1-\kappa^2 u}}.\label{eq:P'_duality}\end{align}\item\label{itm:Tricomi_P_pm_half} We have the following Cauchy principal values for $ -1<x<1$:\begin{align}\mathcal P\int_{-1}^1\frac{ \mathbf K(\sqrt{(1-\xi)/2})}{\pi(x-\xi)\sqrt{1+\xi}}\D\xi={}&\frac{ \mathbf K(\sqrt{(1+x)/2})}{\sqrt{1+x}},\label{eq:P_minus_half_T}\\\mathcal P\int_{-1}^1\frac{ \mathbf K(\sqrt{(1-\xi)/2})-2 \mathbf E(\sqrt{(1-\xi)/2})}{\pi(x-\xi)\sqrt{1+\xi}}\D\xi={}&\frac{ 2\mathbf E(\sqrt{(1+x)/2})-\mathbf K(\sqrt{(1+x)/2})}{\sqrt{1+x}}.\label{eq:P_minus_half_T1}\end{align} \end{enumerate}\end{corollary}\begin{proof}\begin{enumerate}[label=(\alph*), widest=a]\item With Eq.~\ref{eq:mB}, we may interchange the order of integrations to  compute \begin{align*}\int_0^{\sqrt{\vphantom{1}u}}\frac{\mathbf K(k)\D k}{\sqrt{1-k^2}\sqrt{u-k^2}}=\int_0^{\sqrt{\vphantom{1}u}}\left[ \frac{2}{\pi}\int_0^1 \frac{\sqrt{1-\xi^2}\mathbf K(\sqrt{1-\kappa^2})\D \kappa}{1-\xi^2(1-\kappa^2)}\right]\frac{ \D \xi}{\sqrt{1-\xi^2}\sqrt{u-\xi^2}}=\int_{0}^1\frac{\mathbf K(\sqrt{1-\kappa^2})\D \kappa}{\sqrt{1-u(1-\kappa^{2})}},\end{align*}where the last expression is equivalent to the right-hand side of Eq.~\ref{eq:Q'_duality}, by the correspondence $ \kappa\mapsto \sqrt{1-k^2}$. Similarly, we may prove Eq.~\ref{eq:P'_duality} using Eq.~\ref{eq:Beltrami}:\begin{align*}\int^1_{\sqrt{\vphantom{1}u}}\frac{k\mathbf K(k)\D k}{\sqrt{1-k^2}\sqrt{k^2-u}}=\int_0^{\sqrt{\vphantom{1}u}}\left[ \frac{2}{\pi}\int_0^1 \frac{\mathbf K(\sqrt{1-\kappa^2})\D \kappa}{1-k^{2}\kappa^2}\right]\frac{ k\D k}{\sqrt{1-k^2}\sqrt{k^2-u}}={}&\int_0^1\frac{\mathbf K(\sqrt{1-\kappa^2})\D\kappa}{\sqrt{1-\kappa^2}\sqrt{1-\kappa^2 u}}.\end{align*} \item\label{itm:Tricomi_P_half} For $ z>1$, one can  readily compute (cf. \cite[Ref.][p.~233]{SteinII})\begin{align} \mathbf K(\sqrt{z})={}&\int_{0}^1\frac{\D t}{\sqrt{1-t^2}\sqrt{1-z t^2}}=\int_{0}^{1/\sqrt{z}}\frac{\D t}{\sqrt{1-t^2}\sqrt{1-z t^2}}-i\int_{{1/\sqrt{z}}}^{1}\frac{\D t}{\sqrt{1-t^2}\sqrt{z t^2-1}}\notag\\={}&\frac{1}{\sqrt{z}}\int_{0}^{1}\frac{\D s}{\sqrt{1-\frac{1}{z}s^2}\sqrt{1-s^2}}-\frac{i}{\sqrt{z}}\int_{{1}}^{\sqrt{z}}\frac{\D s}{\sqrt{1-\frac{1}{z}s^2}\sqrt{s^2-1}}=\frac{\mathbf K(\sqrt{1/z})-i\mathbf K(\sqrt{(z-1)/z})}{\sqrt{z}}.\label{eq:inv_mod}\end{align}This is the inverse modulus transformation for complete elliptic integrals of the first kind.

Now, in the identity\begin{align*}\mathbf K(k)=\frac{2}{\pi}\int_0^1\frac{\mathbf K(\sqrt{1-\kappa^2})\D\kappa}{1-k^2\kappa^2},\quad \forall k\in\mathbb C\smallsetminus[1,+\infty),\end{align*}we approach the limit of $ k\pm i0^+\in[1,+\infty)$, read off the real part, and employ the inverse modulus transformation (Eq.~\ref{eq:inv_mod}), to obtain\begin{align*}\mathbf K(k)= \frac{k}{\pi}\mathcal P\int_{-1}^1\frac{\mathbf K(\sqrt{1-\kappa^{2}})\D\kappa}{k^{2}-\kappa^{2}},\quad 0<k<1,\end{align*}which can be reformulated into  Eq.~\ref{eq:P_minus_half_T} after the variable substitutions $ k=\sqrt{(1+x)/2}$, $\kappa=\sqrt{(1+\xi)/2} $.

From Eq.~\ref{eq:P_minus_half_T}, we may directly compute (cf.~\cite[Ref.][Eq.~11.215]{KingVol1})\begin{align*}\mathcal P\int_{-1}^1\frac{ \mathbf K(\sqrt{(1-\xi)/2})\sqrt{1+\xi}\D\xi}{\pi(x-\xi)}={}&\mathcal P\int_{-1}^1\frac{ \mathbf K(\sqrt{(1-\xi)/2})[1+(\xi-x)+x]\D\xi}{\pi(x-\xi)\sqrt{1+\xi}}\notag\\={}&(1+x)\mathcal P\int_{-1}^1\frac{ \mathbf K(\sqrt{(1-\xi)/2})\D\xi}{\pi(x-\xi)\sqrt{1+\xi}}-\frac{1}{\pi}\int_{-1}^1\frac{ \mathbf K(\sqrt{(1-\xi)/2})\D\xi}{\sqrt{1+\xi}}\notag\\={}&\sqrt{1+x}\mathbf K\left( \sqrt{\frac{1+x}{2}} \right)-\frac{1}{\pi}\int_{-1}^1\frac{ \mathbf K(\sqrt{(1-\xi)/2})\D\xi}{\sqrt{1+\xi}},\quad -1<x<1.\end{align*}Here, the last integral can be simplified with the knowledge of $ \int_0^1\mathbf K(\sqrt{1-\kappa^2})\D \kappa=\pi^2/4$ \cite[Ref.][item~6.141.2]{GradshteynRyzhik}, which brings us\begin{align*}\mathcal P\int_{-1}^1\frac{ \mathbf K(\sqrt{(1-\xi)/2})\sqrt{1+\xi}\D\xi}{\pi(x-\xi)}=\sqrt{1+x}\mathbf K\left( \sqrt{\frac{1+x}{2}} \right)-\frac{\pi}{\sqrt{2}},\quad -1<x<1.\end{align*}Applying integration by parts to the Tricomi transform (cf.~\cite[Ref.][Eqs.~11.218 and 11.219]{KingVol1}), one can use the equation above
 to deduce that \begin{align*}\mathcal P\int_{-1}^1\frac{1}{\pi(x-\xi)}\frac{\D}{\D \xi}\left[   \mathbf K\left( \sqrt{\frac{1-\xi}{2}} \right)\sqrt{1+\xi} \right]\D\xi=\frac{\D}{\D \xi}\left[\sqrt{1+x}\mathbf K\left( \sqrt{\frac{1+x}{2}} \right)\right]+\frac{1}{\sqrt{2}(x-1)},\quad -1<x<1.\end{align*}In other words, we have \begin{align*}\mathcal P\int_{-1}^1\frac{ \mathbf K(\sqrt{(1-\xi)/2})- \mathbf E(\sqrt{(1-\xi)/2})}{\pi(x-\xi)(1-\xi)\sqrt{1+\xi}}\D\xi={}&\frac{1}{1-x}\left[\frac{ \mathbf E(\sqrt{(1+x)/2})}{\sqrt{1+x}}-\frac{1}{\sqrt{2}}\right],\quad -1<x<1.\end{align*}Carrying this further, we obtain\begin{align*}&\mathcal P\int_{-1}^1\frac{ \mathbf K(\sqrt{(1-\xi)/2})- \mathbf E(\sqrt{(1-\xi)/2})}{\pi(x-\xi)\sqrt{1+\xi}}\D\xi=\frac{ \mathbf E(\sqrt{(1+x)/2})}{\sqrt{1+x}}-\frac{1}{\sqrt{2}}+\frac{1}{\pi}\int_{-1}^1\frac{ \mathbf K(\sqrt{(1-\xi)/2})- \mathbf E(\sqrt{(1-\xi)/2})}{(1-\xi)\sqrt{1+\xi}}\D\xi\notag\\={}&\frac{ \mathbf E(\sqrt{(1+x)/2})}{\sqrt{1+x}}-\frac{1}{\sqrt{2}}+\frac{1}{\pi}\int_{-1}^1\frac{\D}{\D \xi}\left[   \mathbf K\left( \sqrt{\frac{1-\xi}{2}} \right)\sqrt{1+\xi} \right]\D\xi=\frac{ \mathbf E(\sqrt{(1+x)/2})}{\sqrt{1+x}},\quad -1<x<1.\end{align*}We may combine the relation above with  Eq.~\ref{eq:P_minus_half_T} into a symmetric form\begin{align*}\mathcal P\int_{-1}^1\frac{ \mathbf K(\sqrt{(1-\xi)/2})-2 \mathbf E(\sqrt{(1-\xi)/2})}{\pi(x-\xi)\sqrt{1+\xi}}\D\xi=\frac{2\mathbf E(\sqrt{(1+x)/2})-\mathbf K(\sqrt{(1+x)/2})}{\sqrt{1+x}},\quad -1<x<1,\end{align*}as stated in Eq.~\ref{eq:P_minus_half_T1}.
 \qed\end{enumerate}\end{proof}
\subsection{Ramanujan Rotations\label{subsec:Ramanujan}}The next proposition is based on  reverse engineering of some formulae in Ramanujan's notebooks that were stated without proofs.\begin{proposition}[Ramanujan Transformations]\label{prop:Ramanujan}\begin{enumerate}[label=\emph{(\alph*)}, ref=(\alph*), widest=a]\item \label{itm:Ramanujan1} For $ s,t\in[0,1)$, we may represent $ \mathbf K(\sqrt{\vphantom{t}s})\mathbf K(\sqrt{t})$ and $ [\mathbf K(\sqrt{t})]^2$ as  integrals on the unit sphere $ S^2$:\begin{align}\mathbf K(\sqrt{\vphantom{t}s})\mathbf K(\sqrt{t})={}&\frac{1}{8}\int_{S^2}\frac{\D\sigma}{\sqrt{(1-sX^2)(1-tY^2)-Z^2}},\label{eq:K_prod_t1t2}\\
[\mathbf K(\sqrt{t})]^2={}&\int_{S^2}\frac{2(2-t)\mathbf K(\sqrt{X^{2}+Y^2})}{(2-t)^{2}-t^{2}X^{2}}\frac{\D\sigma}{4\pi}\label{eq:K_sqr}
\\={}&\int_{S^2}\frac{(1+t)\mathbf K(\sqrt{X^{2}+Y^2})}{(1+t)^{2}-4tX^{2}}\frac{\D\sigma}{4\pi},\tag{\ref{eq:K_sqr}$_L$}\label{eq:K_sqr_Landen}
\end{align}
where $ X=\sin\theta\cos\phi,Y=\sin\theta\sin\phi,Z=\cos\theta$ are the Cartesian coordinates on the unit sphere  $ S^2$.
Consequently, the following identities\ hold  for $ 0<u<1$:
\begin{align}\int_0^{\sqrt{\vphantom{1}u}}\frac{\mathbf K(k)\D k}{\sqrt{1-k^2}\sqrt{u-k^2}}={}&\int_{0}^1\frac{k \mathbf K(k)\D k}{\sqrt{1-k^2}\sqrt{1-k^2u}}=\frac{1}{1+\sqrt{\vphantom{1}u}}\left[\mathbf K\left( \sqrt{\frac{2\sqrt{\vphantom{1}u}}{1+\sqrt{\vphantom{1}u}}} \right) \right]^2,\label{eq:Q'u}\\\int_0^{\sqrt{\vphantom{1}1-u}}\frac{\mathbf K(k)\D k}{\sqrt{1-k^2}\sqrt{1-u-k^2}}={}&\int_{0}^1\frac{k \mathbf K(k)\D k}{\sqrt{1-k^2}\sqrt{1-k^2(1-u)}}=\frac{2}{1+\sqrt{\vphantom{1}u}}\left[\mathbf K\left( \sqrt{\frac{1-\sqrt{\vphantom{1}u}}{1+\sqrt{\vphantom{1}u}}} \right) \right]^2.\label{eq:Q'1u}\end{align}  \item\label{itm:Ramanujan2}
We have the integral identity\begin{align}\int_{S^2}\mathbf K\left(\sqrt{X^{2}+Y^2}\right)f(|X|){\frac{\D\sigma}{4\pi}}=\frac{2}{\pi}\int_0^1f(k)\mathbf K\left( \sqrt{\frac{1+k}{2}} \right)\mathbf K\left( \sqrt{\frac{1-k}{2}} \right)\D k\label{eq:S2_int_KK}\end{align}so long as the function  $f(|X|),0\leq|X|\leq1$   assumes non-negative values and the surface integral is convergent.
  This allows us to convert Eq.~\ref{eq:K_sqr} into the following identities for $0< t<1 $:
\begin{align}
[\mathbf K(\sqrt{t})]^2={}&\frac{2}{\pi}\int_0^1\frac{\mathbf K(\sqrt{\mathstrut\mu})\mathbf K(\sqrt{\mathstrut1-\mu})\D \mu}{1-\mu t}, \label{eq:K_sqr_star}\tag{\ref{eq:K_sqr}$^*$}\\{[\mathbf K(\sqrt{1-t})]^2}={}&\frac{8}{\pi}\int_0^1\frac{\mathbf K(\sqrt{\mathstrut\mu})\mathbf K(\sqrt{\mathstrut1-\mu})\D \mu}{(1+\sqrt{t})^{2}-\mu (1-\sqrt{t})^{2}},\label{eq:K_sqr_star_Landen}\tag{\ref{eq:K_sqr}$^*_L$}\end{align}and to deduce the following integral formulae:\footnote{Here, in Eq.~\ref{eq:KiK'}, we have chosen the univalent branches of the square roots such that $ \I \sqrt{(2t-1)^{2}-(1-\kappa^2)}\geq0$, and $ \mathbf K(\sqrt{1-t})\geq0,\mathbf K(\sqrt{t})\geq0$; in Eq.~\ref{eq:KiK''}, it is understood that $  \sqrt{(2t-1)^{2}-(1-\kappa^2)}\geq0$ on the left-hand side and \[\mathbf K(\sqrt{t})=\int_0^{\pi/2}\frac{\D\theta}{\sqrt{1+|t|\cos^2\theta}}\geq0,\quad \I\mathbf K(\sqrt{1-t})=\I\int_0^{\pi/2}\frac{\D\theta}{\sqrt{1-(1-t)\cos^2\theta}}\leq0,\quad\forall t<0.\]} \begin{align}\int_{0}^{1}\frac{\mathbf K(\sqrt{1-\kappa^2})}{\sqrt{(2t-1)^{2}-(1-\kappa^2)}}\D\kappa={}&\frac{[ \mathbf K( \sqrt{1-t} )-i\mathbf K( \sqrt{t}  )]^2}{2},\quad 0<t<\frac{1}{2};\label{eq:KiK'}\\
\int_{0}^{1}\frac{\mathbf K(\sqrt{1-\kappa^2})}{\sqrt{(2t-1)^{2}-(1-\kappa^2)}}\D\kappa={}&\frac{[ \mathbf K( \sqrt{1-t} )+i\mathbf K( \sqrt{t}  )]^2}{2},\quad t<0.\tag{\ref{eq:KiK'}$'$}\label{eq:KiK''}\end{align}

 Moreover, we have the  identity below for $ 0<\lambda\leq1/2$: \begin{align}
\int^1_{1-2\lambda}\frac{k\mathbf K(k)\D k}{\sqrt{1-k^{2}}\sqrt{k^2-(1-2\lambda)^{2}}}=\mathbf K(\sqrt{\lambda})\mathbf K(\sqrt{1-\lambda})={}&\int^1_{(1-\lambda)/(1+\lambda)}\frac{\mathbf K(k)\D k}{\sqrt{1-k^{2}}\sqrt{(1+\lambda)^{2}k^2-(1-\lambda)^2}},\label{eq:KK'_dual_forms}
\intertext{or equivalently,}\int_0^{\pi/2}\frac{\mathbf K(\sin\theta)\D\theta}{\sqrt{\vphantom{\sin^2\theta}1-(1-2\lambda)^2\cos^2\theta}}=\mathbf K(\sqrt{\lambda})\mathbf K(\sqrt{1-\lambda})={}&\int_0^{\pi/2}\D\theta\int_0^{\pi/2}\frac{\D\phi}{\sqrt{\smash[b]{(1+\lambda)^{2}-(1-\lambda)^2\sin^2\theta-4\lambda\sin^2\phi}}}.
\label{eq:KK'_dual_forms'}\tag{\ref{eq:KK'_dual_forms}$'$}
\end{align}\end{enumerate} \end{proposition}\begin{proof}\begin{enumerate}[label=(\alph*),widest=a]\item
Writing $ \D\sigma=\sin\theta\D\theta\D\phi$ for the surface element on the unit sphere, we have\begin{align*}\mathbf{K}(\sqrt{s})\mathbf{K}(\sqrt{t})=\frac{1}{8}\int_0^{\pi}\frac{\D\theta}{\sqrt{1-s\cos^2\theta}}\int_0^{2\pi}\frac{\D\phi}{\sqrt{\smash[b]{1-t\cos^2\phi}}}=\frac{1}{8}\int_{S^{2}}\frac{\D\sigma}{\sqrt{1-sZ^{2}}\sqrt{1-Z^2-tX^2}}\end{align*}where the Cartesian coordinates $ X=\sin\theta\cos\phi,Z=\cos\theta$ were used in the last step. As  in Lemma~\ref{lm:sph_rot}, we interchange the r\^oles of  the $X$- and $Z$-axes,    to deduce\begin{align*}\mathbf{K}(\sqrt{\vphantom{t}s})\mathbf{K}(\sqrt{t})={}&\frac{1}{8}\int_{S^{2}}\frac{\D\sigma}{\sqrt{1-sX^{2}}\sqrt{1-X^2-tZ^2}}=\int_0^{\pi/2}\D\theta\int_0^{\pi/2}\D\phi\frac{\sin\theta}{\sqrt{\smash[b]{1-s\sin^{2}\theta\cos^{2}\phi}}\sqrt{\smash[b]{1-\sin^{2}\theta\cos^{2}\phi-t\cos^2\theta}}}\notag\\={}&\int_0^{\pi/2}\frac{\sin\theta\D \theta}{\sqrt{1-t\cos^2\theta}}\int_0^{\pi/2}\frac{\D\phi}{\sqrt{1-s\sin^{2}\theta\cos^{2}\phi}\sqrt{1-\frac{\sin^{2}\theta}{1-t\cos^2\theta}\cos^{2}\phi}}.\end{align*}For any  two numbers $ a,b\in(0,1)$,  we can verify the identity\begin{align}\int_0^{\pi/2}\frac{\D\phi}{\sqrt{1-a\cos^{2}\phi}\sqrt{1-b\cos^{2}\phi}}=\int_0^{\pi/2}\frac{\D\psi}{\sqrt{1-a-(b-a)\cos^{2}\psi}}\label{eq:sqr_comb}\end{align}with a simple substitution $ \phi=\arctan(\sqrt{1-a}\tan\psi)$.  Now, setting $ a=s\sin^{2}\theta,b=\frac{\sin^{2}\theta}{1-t\cos^2\theta}$, we are led to the formula\begin{align}\mathbf{K}(\sqrt{\vphantom{t}s})\mathbf{K}(\sqrt{t})={}&\int_0^{\pi/2}\sin\theta\D \theta\int_0^{\pi/2}\frac{ \D\phi}{\sqrt{\smash[b]{(1-s)(1-\sin^{2}\theta \cos^{2}\phi)+(s-t)\cos^2\theta+st\cos^2\theta\sin^{2}\theta \sin^{2}\phi}}}\notag\\={}&\frac{1}{8}\int_{S^{2}}\frac{\D\sigma}{\sqrt{(1-s)(1-X^{2})+(s-t)Z^{2}+stY^2 Z^2}}=\frac{1}{8}\int_{S^{2}}\frac{\D\sigma}{\sqrt{(1-sY^{2})(1-tZ^{2})-X^{2}}}\tag{\ref{eq:K_prod_t1t2}$'$}\label{eq:K_prod_t1t2'}\end{align}after a literal transformation from the spherical coordinates to Cartesian coordinates, and an algebraic simplification according to the spherical constraint $X^2+Y^2+Z^2=1$. Clearly, the last term in Eq.~\ref{eq:K_prod_t1t2'} is equivalent to the right-hand side of Eq.~\ref{eq:K_prod_t1t2}, as the integral in question remains invariant under a cyclic shift of variables $ (X,Y,Z)\mapsto(Z,X,Y)$.

When $ s=t$, we may swap  the $X$- and $Z$-axes in the penultimate term of Eq.~\ref{eq:K_prod_t1t2'}, to derive\begin{align*}[\mathbf{K}(\sqrt{t})]^2={}&\frac{1}{8}\int_{S^{2}}\frac{\D\sigma}{\sqrt{\smash[b]{(1-t)(1-Z^{2})+t^2Y^2 X^2}}}=\int_0^{\pi/2}\D \theta\int_0^{\pi/2}\D\phi\frac{ 1}{\sqrt{\smash[b]{1-t+t^{2}\sin^2\theta \sin^{2}\phi\cos^{2}\phi}}}\notag\\={}&\frac12\int_0^{\pi/2}\D \theta\int_0^{\pi}\D\phi\frac{ 1}{\sqrt{{1-t+\frac{t^{2}}{4}\sin^{2}\theta\sin^2\phi }}}=\frac{1}{8}\int_{S^{2}}\frac{\D\sigma}{\sqrt{(1-Z^2)\left(1-t+\frac{t^2}{4}Y^2\right) }}.\end{align*}Switching the $ Y$- and $Z$-axes, we then arrive at
\begin{align*}[\mathbf{K}(\sqrt{t})]^2
={}&\frac{1}{8}\int_{S^{2}}\frac{\D\sigma}{\sqrt{(1-Y^2)\left(1-t+\frac{t^2}{4}Z^2\right) }}
=\int_0^{\pi/2}\sin\theta\D \theta\int_0^{\pi/2}\D\phi\frac{ 1}{\sqrt{1-\sin^{2}\theta\sin^2\phi}\sqrt{1-t+\frac{t^2}{4}\cos^{2}\theta}}\notag\\
={}&\int_0^{\pi/2}\frac{ \mathbf K(\sin\theta)\sin\theta\D \theta}{\sqrt{1-t+\frac{t^2}{4}\cos^{2}\theta}}
=\int_0^\pi\frac{ \mathbf K(\sin\theta)\sin\theta\D \theta}{\sqrt{(2-t)^{2}-t^{2}\sin^{2}\theta}}.\end{align*}Judging from the familiar integral formula\[\int_0^{2\pi}\frac{\D\phi}{2-t+t\sin\theta\cos\phi}=\frac{2\pi}{\sqrt{(2-t)^{2}-t^{2}\sin^{2}\theta}},\]it is evident that
\[[\mathbf{K}(\sqrt{t})]^2
=\frac{1}{4\pi}\int_0^{2\pi}\D\phi\int_0^{\pi}\frac{ 2\mathbf K(\sin\theta)\sin\theta\D \theta}{2-t+t\sin\theta\cos\phi}
=\frac{1}{4\pi}\int_{S^2}\frac{2\mathbf K(\sqrt{X^{2}+Y^2})\D\sigma}{2-t+tX}.\]Pairing up the values of the integrand at $ \pm X$, we obtain
\begin{align*}[\mathbf{K}(\sqrt{t})]^2={}&\frac{1}{4\pi}\int_{S^2}\frac{2\mathbf K(\sqrt{X^{2}+Y^2})\D\sigma}{2-t+tX}=\frac{2-t}{4\pi}\int_{S^2}\frac{2\mathbf K(\sqrt{X^{2}+Y^2})\D\sigma}{(2-t)^{2}-t^{2}X^{2}},\end{align*}
as claimed in  Eq.~\ref{eq:K_sqr}.
Then, Landen's transformation\begin{align}\mathbf K(\sqrt{t})=\frac{1}{1+\sqrt{t}}\mathbf K\left( \frac{2\sqrt[4]{t}}{1+\sqrt{t}} \right),\quad 0\leq t<1\label{eq:LandenL}\end{align}brings us to  Eq.~\ref{eq:K_sqr_Landen}.

On account of Eqs.~\ref{eq:K_sqr} and \ref{eq:K_sqr_Landen}, we may put down the relations\begin{align*}\left[ \mathbf K\left( \sqrt{\frac{2\sqrt{\vphantom{1}u}}{1+\sqrt{\vphantom{1}u}}} \right) \right]^2={}&(1+\sqrt{\vphantom{1}u})\int_{S^2}\frac{\mathbf K(\sqrt{X^2+Y^2})}{1-uX^2}\frac{\D \sigma}{4\pi}=(1+\sqrt{\vphantom{1}u})\int_{0}^1\frac{k \mathbf K(k)\D k}{\sqrt{1-k^2}\sqrt{1-k^2u}},\notag\\\left[ \mathbf K\left( \sqrt{\frac{1-\sqrt{\vphantom{1}u}}{1+\sqrt{\vphantom{1}u}}} \right) \right]^2={}&\frac{1+\sqrt{\vphantom{1}u}}{2}\int_{S^2}\frac{\mathbf K(\sqrt{X^2+Y^2})}{1-(1-u)X^2}\frac{\D \sigma}{4\pi}=\frac{1+\sqrt{\vphantom{1}u}}{2} \int_{0}^1\frac{k \mathbf K(k)\D k}{\sqrt{1-k^2}\sqrt{1-k^2(1-u)}},\end{align*}which confirm, respectively, the right halves of  Eqs.~\ref{eq:Q'u} and \ref{eq:Q'1u}. The left halves of  Eqs.~\ref{eq:Q'u} and \ref{eq:Q'1u} follow from the duality relation in  Eq.~\ref{eq:Q'_duality}.

One may wish to compare  Eqs.~\ref{eq:K_sqr} and \ref{eq:K_sqr_Landen} (applicable to $ 0<u<1$) to the Hobson coupling formula in Eq.~\ref{eq:LegendreP_quarter_spec} (confined to $ 0\leq u\leq 1/2$).\item We start with a spherical rotation\begin{align*}&\int_{S^2}\mathbf K\left(\sqrt{X^{2}+Y^2}\right)f(|X|){\frac{\D\sigma}{4\pi}}=\int_{S^2}\mathbf K\left(\sqrt{Z^{2}+Y^2}\right)f(|Z|){\frac{\D\sigma}{4\pi}}\notag\\={}&\frac{2}{\pi}\int_0^1\D Z\int_{0}^{\pi/2}\D\phi\,\mathbf K\left(\sqrt{Z^{2}+(1-Z^{2}){\sin^2 \phi}}\right)f(|Z|)=\frac{2}{\pi}\int_0^1\D k\int_{0}^{\pi/2}\D\phi\,\mathbf K\left(\sqrt{k^{2}\cos^2{\phi}+\sin^2 {\phi}}\right)f(k)\notag\\={}&\frac{2}{\pi}\int_0^1f(k)\left[ \int_{0}^{\sqrt{1-k^2}}\frac{\mathbf K(\sqrt{1-\kappa^2})}{\sqrt{1-k^{2}-\kappa^2}}\D\kappa\right]\D k=\frac{2}{\pi}\R\int_0^1f(k)\left[ \int_{0}^{1}\frac{\mathbf K(\sqrt{1-\kappa^2})}{\sqrt{1-k^{2}-\kappa^2}}\D\kappa\right]\D k,\end{align*}where we have made the substitution $ \kappa=\sqrt{1-k^2}\cos\phi$  in the last line. Then, we may evaluate\begin{align*}\R\left[ \int_{0}^{1}\frac{\mathbf K(\sqrt{1-\kappa^2})}{\sqrt{1-k^{2}-\kappa^2}}\D\kappa\right]=-\I\left[ \int_{0}^{1}\frac{\mathbf K(\sqrt{1-\kappa^2})}{\sqrt{k^{2}-1+\kappa^2}}\D\kappa\right]\end{align*}by analytic continuation. From  Eq.~\ref{eq:K_sqr}, we have\begin{align*} \int_{0}^{1}\frac{\mathbf K(\sqrt{1-\kappa^2})}{\sqrt{1-z+\frac{z^{2}}{4}\kappa^2}}\D\kappa=\left[ \int_{0}^1\frac{\D t}{\sqrt{1-t^2}\sqrt{1-z t^2}} \right]^2,\quad 0\leq z<1.\end{align*}   After setting $ z=2/(1+k)\geq1$ in the inverse modulus transformation formula (Eq.~\ref{eq:inv_mod}), we obtain\begin{align}(1+k)\int_{0}^{1}\frac{\mathbf K(\sqrt{1-\kappa^2})}{\sqrt{k^{2}-1+\kappa^2}}\D\kappa=\frac{1+k}{2}\left[ \mathbf K\left( \sqrt{\frac{1+k}{2}} \right)-i\mathbf K\left( \sqrt{\frac{1-k}{2}} \right )\right]^2,\quad 0\leq k<1.\label{eq:KiK}\end{align} Reading off the imaginary parts on both sides of the formula above, we arrive at the following equation for $ \lambda=(1-k)/2\in(0,1/2]$:
\begin{align}
\int^1_{1-2\lambda}\frac{k\mathbf K(k)\D k}{\sqrt{1-k^{2}\vphantom{)^2}}\sqrt{k^2-(1-2\lambda)^{2}}}=\mathbf K(\sqrt{\lambda})\mathbf K(\sqrt{1-\lambda}),\label{eq:KK'_usual}
\end{align}as well as the identity claimed in Eq.~\ref{eq:S2_int_KK}.

Applying  Eq.~\ref{eq:S2_int_KK} to  Eq.~\ref{eq:K_sqr}, we obtain the result stated in Eq.~\ref{eq:K_sqr_star}:\begin{align*}[\mathbf K(\sqrt{t})]^2={}&\int_{S^2}\frac{2(2-t)\mathbf K(\sqrt{X^{2}+Y^2})}{(2-t)^{2}-t^{2}X^{2}}{\frac{\D\sigma}{4\pi}}\notag\\={}&\frac{2}{\pi}\left[\int_0^1\frac{1}{2-t+kt}\mathbf K\left( \sqrt{\frac{1+k}{2}} \right)\mathbf K\left( \sqrt{\frac{1-k}{2}} \right)\D k+\int_0^1\frac{1}{2-t-k't}\mathbf K\left( \sqrt{\frac{1+k'}{2}} \right)\mathbf K\left( \sqrt{\frac{1-k'}{2}} \right)\D k'\right]\notag\\={}&\frac{2}{\pi}\left[\int_0^{1/\sqrt{2}}\frac{2\mathbf K(\sqrt{\mathstrut\mu})\mathbf K(\sqrt{\mathstrut1-\mu})}{2-t+(1-2\mu)t}\D\mu+\int^1_{1/\sqrt{2}}\frac{2\mathbf K(\sqrt{\mathstrut\mu'})\mathbf K(\sqrt{\mathstrut1-\mu'})}{2-t-(2\mu'-1)t}\D\mu'\right]=\frac{2}{\pi}\int_0^1\frac{\mathbf K(\sqrt{\mathstrut\mu})\mathbf K(\sqrt{\mathstrut1-\mu})\D \mu}{1-\mu t}.\end{align*} (An equivalent form of  Eq.~\ref{eq:K_sqr_star} has appeared as  Eq.~26 in~\cite{Wan2012}, which was derived combinatorially using generalized hypergeometric series, instead of the geometric interpretation given here.)
One may  derive Eq.~\ref{eq:K_sqr_star_Landen} from Eq.~\ref{eq:K_sqr_star} and Landen's transformation\begin{align*}\mathbf K(\sqrt{1-t})=\frac{2}{1+\sqrt{t}}\mathbf K\left( \frac{1-\sqrt{t}}{1+\sqrt{t}} \right).\end{align*}

  Writing $ t=(1-k)/2\in(0,1/2)$, we may  recast Eq.~\ref{eq:KiK} into Eq.~\ref{eq:KiK'}. By  analytically continuing   Eq.~\ref{eq:KiK'} to negative valued  $t<0$ with the aid of the integral representation of the elliptic integral $ \mathbf K$ for complex-valued modulus, we arrive at Eq.~\ref{eq:KiK''}. Here, the change in the sign from   $ -i\mathbf K( \sqrt{t}  )$ in Eq.~\ref{eq:KiK'} to $ +i\mathbf K( \sqrt{t}  )$ in Eq.~\ref{eq:KiK''} is worthy of special attention. It is ready to appreciate that such a sign change ensures compatibility with the natural choices of the univalent branches of the square roots occurring in all the places.

The leftmost equality in Eq.~\ref{eq:KK'_dual_forms}  has been already proved in Eq.~\ref{eq:KK'_usual}, and its counterpart in Eq.~\ref{eq:KK'_dual_forms'} follows from Eq.~\ref{eq:P'_duality}. To demonstrate the rightmost equality in Eq.~\ref{eq:KK'_dual_forms}, we perform the following computations for $ 0<\lambda\leq1/2$: \begin{align*}\int^1_{(1-\lambda)/(1+\lambda)}\frac{\mathbf K(k)\D k}{\sqrt{1-k^{2}}\sqrt{(1+\lambda)^{2}k^2-(1-\lambda)^2}}={}&\R\int^{\pi/2}_{0}\frac{\mathbf K(\sin\theta)\D \theta}{\sqrt{4\lambda-(1+\lambda)^{2}\cos^2\theta}}=\frac{1}{2\sqrt{\lambda}}\R\left[ \mathbf K\left (\frac{i(1-\sqrt{\lambda})}{2\sqrt[4]\lambda} \right)\mathbf K\left( \frac{(1+\sqrt{\lambda})}{2\sqrt[4]\lambda} \right) \right]\end{align*}by using an analytic continuation of the leftmost term of Eq.~\ref{eq:KK'_dual_forms'} in the last step. According to the imaginary modulus transformation and Landen's transformation, we have\begin{align*}\mathbf K\left (\frac{i(1-\sqrt{\lambda})}{2\sqrt[4]\lambda} \right)=\frac{2\sqrt[4]\lambda}{1+\sqrt{\lambda}}\mathbf K\left( \frac{1-\sqrt{\lambda}}{1+\sqrt{\lambda}} \right)=\sqrt[4]\lambda\mathbf K(\sqrt{1-\lambda});\end{align*} whereas the inverse modulus transformation  (Eq.~\ref{eq:inv_mod})  and Landen's transformation lead us to \begin{align*}\R\left[\mathbf K\left( \frac{(1+\sqrt{\lambda})}{2\sqrt[4]\lambda} \right) \right]=\frac{2\sqrt[4]\lambda}{1+\sqrt{\lambda}}\mathbf K\left( \frac{2\sqrt[4]\lambda}{1+\sqrt{\lambda}} \right)=2\sqrt[4]\lambda\mathbf K(\sqrt{\lambda}).\end{align*}Thus, the rightmost equality in Eq.~\ref{eq:KK'_dual_forms} is verified.  The connection to its counterpart in  Eq.~\ref{eq:KK'_dual_forms'} can be proved as follows:
\begin{align*}&\int_{\sqrt{\vphantom{1}u}}^1\frac{\mathbf K(k)\D k}{\sqrt{1-k^2}\sqrt{k^2-u}}=\int_0^{\pi/2}\frac{1}{\sqrt{1-(1-u)\sin^2\theta}}\mathbf K\left( \sqrt{\frac{u}{1-(1-u)\sin^2\theta}} \right)\D\theta\notag\\={}&\int_0^{\pi/2}\D\theta\int_0^{\pi/2}\frac{\D\phi}{\sqrt{\smash[b]{1-(1-u)\sin^2\theta-u\sin^2\phi}}}.\end{align*}Here,  we have used the variable substitution $ k=\sqrt{\smash[b]{u/[1-(1-u)\sin^2\theta]}}$ to complete the  first line in the equation above, before spelling out $ \mathbf K(\sqrt{u/[1-(1-u)\sin^2\theta]})$ as an integral in $ \phi$ in the second line.

    The last equality in  Eq.~\ref{eq:KK'_dual_forms'} implies the following evaluation for $0<u<1$:
\begin{align}\int_0^{\pi/2}\D\theta\int_0^{\pi/2}\frac{\D\phi}{\sqrt{\smash[b]{1-u\sin^2\theta-(1-u)\sin^2\phi}}}={\frac{2}{1+\sqrt{\vphantom{1}u}}}\mathbf K\left( \sqrt{\frac{1-\sqrt{\vphantom{1}u}}{1+\sqrt{\vphantom{1}u}}} \right)\mathbf K\left( \sqrt{\frac{2\sqrt{\vphantom{1}u}}{1+\sqrt{\vphantom{1}u}}} \right),\label{eq:KK'_dual_forms'_star}\tag{\ref{eq:KK'_dual_forms'}$^*$}\end{align}   so an interchange of the angular variables $ \theta$ and $\phi$ brings us back to the $ u\leftrightarrow1-u$ symmetry displayed in Eq.~\ref{eq:PuP1_u_symm}.

One may also wish to compare the formula\begin{align}\frac{2}{1+\sqrt{\vphantom{1}u}}\mathbf K\left( \sqrt{\frac{2\sqrt{\vphantom{1}u}}{1+\sqrt{\vphantom{1}u}}} \right)\mathbf K\left( \sqrt{\frac{1-\sqrt{\vphantom{1}u}}{1+\sqrt{\vphantom{1}u}}} \right)=\int_{\sqrt{\vphantom{1}u}}^1\frac{\mathbf K(k)\D k}{\sqrt{1-k^2}\sqrt{k^2-u}},\quad 0<u<1\end{align}to the corresponding result from Hobson coupling (Eq.~\ref{eq:LegendreP_quarter_spec'}).     \qed\end{enumerate}\end{proof}\begin{remark}\begin{enumerate}[label=(\arabic*)]
\item
Up to a variable substitution, our intermediate result in part (a)   \begin{align}[\mathbf K(\sqrt{t})]^2=\int_0^{\pi/2}\frac{ \mathbf K(\sin\theta)\sin\theta\D \theta}{\sqrt{1-t+\frac{t^2}{4}\cos^{2}\theta}}=\int_0^{1}\frac{ \mathbf K(\sqrt{1-\kappa^{2}})\D \kappa}{\sqrt{1-t+\frac{t^2}{4}\kappa^{2}}},\quad 0\leq t<1\tag{\ref{eq:K_sqr}$'$}\label{eq:K_sqr'}\end{align}is related to  formula   13.0.18 in \cite{Apelblat}, and item 2.16.5.7 of \cite{PBMVol3}. However, it is highly probable that the ultimate source of this formula is  Ramanujan.

Entry~7(x) in Chapter 17 of Ramanujan's second notebook essentially reads \[\int_0^{\pi/2}\int_0^{\pi/2}\frac{\D\theta\D\phi}{\sqrt{\smash[b]{(1-u\sin^2\theta)(1-u\sin^2\theta\sin^2\phi)}}}=\left( \int_0^{\pi/2}\frac{\D\psi}{\sqrt{\smash[b]{1-u\sin^4\psi}}} \right)^2,\]which was verified by  B. C. Berndt with some highly technical transformations  of the generalized hypergeometric series $ _4F_3$  \cite[Ref.][pp.~110-111]{RN3}. In fact, the left-hand side of Ramanujan's formula equals\begin{align*}\int_0^{\pi/2}\frac{\mathbf K(\sqrt{\vphantom{1}u}\sin\theta)\D\theta}{\sqrt{\smash[b]{1-u\sin^2\theta}}}=\int_0^{\pi/2}\frac{\mathbf K(\sin\theta)\sin\theta\D\theta}{\sqrt{\smash[b]{1-u\sin^2\theta}}}\end{align*}after integration in $ \phi$ and reference to Eq.~\ref{eq:Q'_duality}, while its right-hand side equals $ \pi^2[P_{-1/4}(1-2u)]^2/4$, according to the integral identity in Eq.~\ref{eq:P_quarter_sin4}. In this manner, we see that Ramanujan's formula  is actually equivalent to Eq.~\ref{eq:K_sqr'}, a consequence of spherical rotations.

In \cite[Ref.][pp.~111-112]{RN3}, Berndt followed up with a combinatorial proof of Entry 7(xi) in Ramanujan's notebook: \begin{align*}\int_0^{\pi/2}\int_0^{\pi/2}\frac{k\sin\theta\D\theta\D\phi}{\sqrt{\smash[b]{(1-k^{2}\sin^2\theta)(1-k^2\sin^2\theta\sin^2\phi)}}}={}&\int_0^{\pi/2}\int_0^{\arcsin k}\frac{\D\theta\D\phi}{\sqrt{\smash[b]{1-k^{2}\sin^2\theta-\sin^2\theta\sin^2\phi}}}\notag\\={}&\frac{1}{2}\left\{ \left[\mathbf K\left( \sqrt{\frac{1+k}{2}} \right)\right]^2-\left[\mathbf K\left( \sqrt{\frac{1-k}{2}} \right)\right]^2\right\},\end{align*}which turns out to be exactly the real part of our Eq.~\ref{eq:KiK'}, a geometrically motivated result.

Berndt has described his proofs of these two entries  as ``undoubtedly not those found by Ramanujan''. Judging from the use of spherical coordinates in Ramanujan's presentation, we take leave to think that the technique of spherical  rotations given in the proof of this proposition might  be closer to a rediscovery of  the pathway that Ramanujan has  originally undertaken.       \item Combining our analysis in  part (b) with Eq.~\ref{eq:sqr_comb}, we have effectively shown that spherical geometry entails the following integral formula\begin{align*}\int_{0}^{\pi/2}\,\mathbf K\left(\sqrt{k^{2}\cos^2{\phi}+\sin^2 {\phi}}\right)\D\phi= \int_0^{\pi/2}\frac{ \mathbf K(\sin\theta)\D \theta}{\sqrt{1-k^2\cos^2\theta}}=\mathbf K\left( \sqrt{\frac{1+k}{2}} \right)\mathbf K\left( \sqrt{\frac{1-k}{2}} \right),\end{align*}which is a relation that has nearly 80 years of history. In 2008, Bailey~\textit{et al.}~used the method of Bessel  moments to discover the following integral identity (Eq.~49 in \cite{Bailey2008})\begin{align*}\int_0^{\pi/2}\frac{\mathbf K(\sin\theta)\D \theta}{\sqrt{1-\cos^2\alpha\cos^2\theta}}=\mathbf K\left(\sin\frac{\alpha}{2}\right)\mathbf K\left(\cos\frac{\alpha}{2}\right),\end{align*}and remarked on its equivalence to a formula derived by Glasser in 1976:\begin{align*}\int_0^{\pi/2}\mathbf K\left(\sqrt{1-{\sin\vphantom{1}^2\alpha\cos\vphantom{1}^2\phi}}\right)\D\phi=\mathbf K\left(\sin\frac{\alpha}{2}\right)\mathbf K\left(\cos\frac{\alpha}{2}\right).\end{align*}As pointed out by Zucker \cite{Zucker2011}, the left-hand side in the  formula of Glasser, being a type of generalized Watson integral, can be expressed in terms of  $ \mathfrak F_4$, an Appell hypergeometric function, which in turn,  reduces to the product of two complete elliptic integrals of the first kind, according to a result by Bailey in 1933 \cite{Bailey1933}. Needless to say, all these displayed formulae are equivalent to each other, and echo back to Eq.~\ref{eq:LegendreP_half_spec'}, which was derived earlier from the Hobson coupling formula.  \eor
\end{enumerate}\end{remark}\begin{corollary}[Some Cauchy Principal Values]  For $ -1<x<1$, we have \begin{align}
\mathcal P\int_{-1}^1\mathbf K\left( \sqrt{\frac{1+\xi}{2}} \right)\mathbf K\left( \sqrt{\frac{1-\xi}{2}} \right)\frac{ 2\D \xi}{\pi (x-\xi)}={}&\left[\mathbf K \left(\sqrt{\frac{\vphantom{\xi}1+x}{2}}\right) \right]^2-\left[\mathbf K \left(\sqrt{\frac{\vphantom{\xi}1-x}{2}}\right) \right]^2\label{eq:RamanujanT}
\end{align}and \begin{align}&
\mathcal P\int_{-1}^1\left[\mathbf K\left(\sqrt{\frac{1+\xi}{2}}\right) \mathbf K\left( \sqrt{\frac{1-\xi}{2}} \right)+\mathbf K\left( \sqrt{\frac{1-\xi}{2}} \right) \mathbf E\left(\sqrt{\frac{1+\xi}{2}}\right)-\mathbf K\left(\sqrt{\frac{1+\xi}{2}}\right) \mathbf E\left( \sqrt{\frac{1-\xi}{2}} \right)\right]\frac{ \D \xi}{\pi (x-\xi)}\notag\\={}&-\left[\mathbf K \left(\sqrt{\frac{\vphantom{\xi}1-x}{2}}\right) \right]^2+ \mathbf K\left(\sqrt{\frac{\vphantom{\xi}1-x}{2}}\right)  \mathbf E\left(\sqrt{\frac{\vphantom{\xi}1-x}{2}}\right)+ \mathbf K\left(\sqrt{\frac{\vphantom{\xi}1+x}{2}}\right)  \mathbf E\left(\sqrt{\frac{\vphantom{\xi}1+x}{2}}\right).\label{eq:KE_EK_T}
\end{align}\end{corollary}\begin{proof}From Eq.~\ref{eq:K_sqr_star} and the inverse modulus transformation $ \R\{[\mathbf K(1/\sqrt{t})]^2/t\}=[\mathbf K(\sqrt{t})]^2-[\mathbf K(\sqrt{1-t})]^2$ for $ t\in(0,1)$ (cf.~Eq.~\ref{eq:inv_mod}), we may deduce\begin{align*}[\mathbf K(\sqrt{t})]^2-[\mathbf K(\sqrt{1-t})]^2={}&\frac{2}{\pi}\mathcal P\int_0^1\frac{\mathbf K(\sqrt{\mathstrut\mu})\mathbf K(\sqrt{\mathstrut1-\mu})\D \mu}{t-\mu },\quad 0<t<1,\end{align*}which is equivalent to Eq.~\ref{eq:RamanujanT}.

Now, we proceed as in the proof of  Corollary~\ref{cor:Beltrami_app}\ref{itm:Tricomi_P_half}, and compute\begin{align}
\mathcal P\int_{-1}^1\mathbf K\left( \sqrt{\frac{1+\xi}{2}} \right)\mathbf K\left( \sqrt{\frac{1-\xi}{2}} \right)\frac{ 2(1+\xi)\D \xi}{\pi (x-\xi)}={}&(1+x)\left\{\left[\mathbf K \left(\sqrt{\frac{\vphantom{\xi}1+x}{2}}\right) \right]^2-\left[\mathbf K \left(\sqrt{\frac{\vphantom{\xi}1-x}{2}}\right) \right]^2\right\}-\frac{2}{\pi}\int_{-1}^1\mathbf K\left( \sqrt{\frac{1+\xi}{2}} \right)\mathbf K\left( \sqrt{\frac{1-\xi}{2}} \right)\D \xi\notag\\={}&(1+x)\left\{\left[\mathbf K \left(\sqrt{\frac{\vphantom{\xi}1+x}{2}}\right) \right]^2-\left[\mathbf K \left(\sqrt{\frac{\vphantom{\xi}1-x}{2}}\right) \right]^2\right\}-\frac{\pi^2}{2},\label{eq:KK'_T1}\end{align}  where we have quoted the result $ T_{0,-1/2}=\int_{-1}^{1}P_{-1/2}(\xi)P_{-1/2}(-\xi)\D\xi=\pi$ from a limit scenario of Eq.~\ref{eq:T_0_nu}.  Likewise, we may deduce from the equation above another Cauchy principal value:\begin{align*}
\mathcal P\int_{-1}^1\mathbf K\left( \sqrt{\frac{1+\xi}{2}} \right)\mathbf K\left( \sqrt{\frac{1-\xi}{2}} \right)\frac{ 2(1-\xi^{2})\D \xi}{\pi (x-\xi)}={}&(1-x^{2})\left\{\left[\mathbf K \left(\sqrt{\frac{\vphantom{\xi}1+x}{2}}\right) \right]^2-\left[\mathbf K \left(\sqrt{\frac{\vphantom{\xi}1-x}{2}}\right) \right]^2\right\}+\frac{\pi^2x}{2}.\end{align*}Like what we did in  the proof of  Corollary~\ref{cor:Beltrami_app}\ref{itm:Tricomi_P_half}, we can use the last equation to derive\begin{align}\mathcal P\int_{-1}^1\frac{\D}{\D\xi}\left[(1-\xi^{2})\mathbf K\left( \sqrt{\frac{1+\xi}{2}} \right)\mathbf K\left( \sqrt{\frac{1-\xi}{2}} \right)\right]\frac{ 2\D \xi}{\pi (x-\xi)}={}&\frac{\D}{\D x}\left((1-x^{2})\left\{\left[\mathbf K \left(\sqrt{\frac{\vphantom{\xi}1+x}{2}}\right) \right]^2-\left[\mathbf K \left(\sqrt{\frac{\vphantom{\xi}1-x}{2}}\right) \right]^2\right\}+\frac{\pi^2x}{2}\right).\label{eq:KK'_T2}\end{align} after integration by parts. A little algebra then reveals that Eqs.~\ref{eq:KK'_T1} and \ref{eq:KK'_T2} together entail Eq.~\ref{eq:KE_EK_T}.
\qed\end{proof}

\section{\label{sec:Tricomi}Finite Hilbert Transform and Tricomi Pairing}\subsection{Parseval Identity for Tricomi Transforms}The finite Hilbert transform on the interval $ (-1,1)$, also known as the Tricomi transform, is defined via a Cauchy principal value (see \cite[Ref.][Chap.~4]{TricomiInt} or \cite[Ref.][Chap.~11]{KingVol1}):\begin{align*}(\widehat{\mathcal T} f)(x):=\mathcal P\int_{-1}^1\frac{f(\xi)\D \xi}{\pi(x-\xi)},\quad\text{a.e. } x\in(-1,1).\end{align*}This induces a continuous linear operator $ \widehat {\mathcal T}:L^p(-1,1)\longrightarrow L^p(-1,1)$ for $ 1<p<+\infty$ \cite[Ref.][p.~188]{SteinWeiss}.

We can  derive some new integral relations  from  a Parseval-type identity for the Tricomi transform (see \cite[Ref.][\S4.3, Eq.~2]{TricomiInt} or \cite[Ref.][Eq.~11.237]{KingVol1}):\begin{align}\int_{-1}^1 f(x)(\widehat{\mathcal T}g)(x)\D x+\int_{-1}^1 g(x)(\widehat{\mathcal T}f)(x)\D x={}&0,\label{eq:Tricomi_Parseval1}\end{align}where $ f\in L^p(-1,1),p>1$; $ g\in L^q(-1,1),q>1$ and $ \frac1p+\frac1q<1$. We call such a procedure ``Tricomi pairing'', for which an example in the proposition below puts  finishing touches on a proof for the conjectural identity stated in the introduction. \begin{proposition}[An Application of Tricomi Pairing to Multiple Elliptic Integrals]\label{prop:Tricomi_Paring}There are several integrals that evaluate to the same number $ [\Gamma(\frac{1}{4})]^{8}/(128\pi^2)$:\begin{align}&\int_0^1\left[\mathbf K\left( \sqrt{1-k^2} \right)\right]^{3}\D k=\frac{10}{3}\int_0^1[\mathbf K(k)]^{3}\D k=5\int_0^1[\mathbf K(k)]^{3}k\D k\notag\\={}&3\int_0^1[\mathbf K(k)]^2\mathbf K\left( \sqrt{1-k^2} \right)\D k=2\int_0^1\mathbf K(k)\left[\mathbf K\left( \sqrt{1-k^2} \right)\right]^{2}\D k=6\int_0^1[\mathbf K(k)]^2\mathbf K\left( \sqrt{1-k^2} \right)k\D k.\label{eq:tripleK1}\end{align}\end{proposition}\begin{proof}
If we set $ f(x)=\mathbf K(\sqrt{(1-x)/2})/\sqrt{1+x},-1<x<1$ and  $ g(x)=2\mathbf K(\sqrt{(1+x)/2})\mathbf K(\sqrt{(1-x)/2}),-1<x<1$ in Eq.~\ref{eq:Tricomi_Parseval1}, while recalling the Cauchy principal values given in  Eqs.~\ref{eq:P_minus_half_T} and \ref{eq:RamanujanT},   then we arrive at \[2\int_{-1}^1 \left[\mathbf K\left( \sqrt{\frac{\vphantom{\xi}1+x}{2}} \right)\right]^2\mathbf K\left( \sqrt{\frac{\vphantom{\xi}1-x}{2}} \right)\frac{\D x}{\sqrt{1+x}}=-\int_{-1}^1\mathbf K\left( \sqrt{\frac{1-\xi}{2}} \right)\left\{ \left[\mathbf K \left(\sqrt{\frac{\vphantom{\xi}1+\xi}{2}}\right) \right]^2-\left[\mathbf K \left(\sqrt{\frac{\vphantom{\xi}1-\xi}{2}}\right) \right]^2 \right\}\frac{\D \xi}{\sqrt{1+\xi}}.\]This instantly rearranges into $ \int_0^1[\mathbf K( \sqrt{1-k^2} )]^{3}\D k=3\int_0^1[\mathbf K(k)]^2\mathbf K( \sqrt{1-k^2} )\D k$ (an identity that  Wan conjectured numerically in \cite{Wan2012} without an analytic proof), which  reveals the equivalence between the leading items  of the first two lines  in Eq.~\ref{eq:tripleK1}.

 Writing $ k=(1-\xi)/(1+\xi)$ and employing Landen's transformation $\mathbf K(2\sqrt{\xi}/(1+\xi))=(1+\xi)\mathbf K(\xi) $, one has
  \begin{align*}\int_0^1\left[\mathbf K\left( \sqrt{1-k^2} \right)\right]^{3}\D k=\int_0^1\left[\mathbf K\left( \sqrt{1-\left(\frac{1-\xi}{1+\xi}\right)^2} \right)\right]^{3}\D \frac{1-\xi}{1+\xi}=2\int_0^1[\mathbf K(\xi)]^3(1+\xi)\D\xi;\end{align*}writing $ k=(1-\xi)/(1+\xi)$ and employing Landen's transformation $2\mathbf K((1-\xi)/(1+\xi))=(1+\xi)\mathbf K(\sqrt{1-\xi^{2}}) $, one has
\begin{align*}\int_0^1[\mathbf K(k)]^3\D k=\int_0^1\left[\mathbf K\left( \frac{1-\xi}{1+\xi} \right)\right]^{3}\D \frac{1-\xi}{1+\xi}=\frac{1}{4}\int_0^1\left[\mathbf K\left( \sqrt{1-\xi^2} \right)\right]^{3}(1+\xi)\D\xi=\frac{1}{4}\int_0^1\left[\mathbf K\left( \sqrt{1-k^2} \right)\right]^{3}\D k+\frac{1}{4}\int_0^1[\mathbf K(\eta)]^3\eta\D\eta,\end{align*}where the last step is a trivial substitution $ \xi\mapsto\sqrt{1-\eta^2}$. The  two simultaneous equations displayed above make it possible to eliminate any one among the three quantities  $ \int_0^1[\mathbf K(\sqrt{1-k^2})]^3\D k$, $ \int_0^1[\mathbf K(k)]^3\D k$, $ \int_0^1[\mathbf K(k)]^3k\D k$, and determine the ratio between the  two remaining numbers. Thus, we have verified the chain of identities  in the first line
 of  Eq.~\ref{eq:tripleK1} (cf.~\cite[Ref.][Eq.~29]{Wan2012}).
The relations \textit{within}  the second line of   Eq.~\ref{eq:tripleK1} can  be likewise established by successive Landen's transformations, as detailed in the first paragraph of \cite[Ref.][p.~139]{Wan2012}.

As we recall from Eq.~\ref{eq:T_half_half} that \begin{align*}6\int_0^1[\mathbf K(k)]^2\mathbf K\left( \sqrt{1-k^2} \right)k\D k=\frac{[\Gamma(\frac{1}{4})]^{8}}{128\pi^{2}},\end{align*}the verification is complete. \qed\end{proof}\begin{remark}Sometimes, the output of Tricomi pairing can also be immediately recovered by more straightforward means. For example,  upon setting\begin{align*}f(\xi)={}&2\mathbf K\left( \sqrt{\frac{1+\xi}{2}} \right)\mathbf K\left( \sqrt{\frac{1-\xi}{2}} \right),\notag\\g( \xi)={}&\mathbf K\left(\sqrt{\frac{1+\xi}{2}}\right) \mathbf K\left( \sqrt{\frac{1-\xi}{2}} \right)+\mathbf K\left( \sqrt{\frac{1-\xi}{2}} \right) \mathbf E\left(\sqrt{\frac{1+\xi}{2}}\right)-\mathbf K\left(\sqrt{\frac{1+\xi}{2}}\right) \mathbf E\left( \sqrt{\frac{1-\xi}{2}} \right)\end{align*} in  Eq.~\ref{eq:Tricomi_Parseval1}, one arrives at a vanishing identity:\begin{align*}0={}&\int_0^1\left[ \mathbf K\left( \sqrt{1-k^2} \right) \right]^2\left[\mathbf K\left( \sqrt{1-k^2} \right)\mathbf K(k) +\mathbf K\left( \sqrt{1-k^2} \right) \mathbf E(k)-3\mathbf K(k)\mathbf E\left( \sqrt{1-k^2} \right)\right]k\D k.\end{align*} As pointed out by an anonymous referee, the equation above is anticipated from the fact that the integrand is precisely the derivative of $ k^2(1-k^2)\mathbf K(k)[\mathbf K(\sqrt{1-k^2})]^3$, which makes it less surprising than Eq.~\ref{eq:tripleK1}. It might be still interesting to ask if  Eq.~\ref{eq:tripleK1} can  be likewise reduced into finite steps of algebraic manipulations on elliptic integrals and  applications of the Newton-Leibniz formula (cf.~\S\ref{sec:outlook}).      \eor\end{remark}
 \subsection{\label{subsec:Tricomi_comp}Tricomi Transform of $ P_{\nu}(x)P_\nu(-x)$}The formula in Eq.~\ref{eq:RamanujanT} can be rewritten as \[\mathcal P\int_{-1}^1\frac{2P_{-1/2}(\xi)P_{-1/2}(-\xi)}{\pi(x-\xi)}\D \xi=-\{[P_{-1/2}(x)]^2-[P_{-1/2}(-x)]^2\},\quad -1<x<1.\]This is not accidental. In the next proposition, we will generalize such a Tricomi transform relation to Legendre functions of arbitrary degree $ \nu$.\begin{proposition}[Tricomi Transform of $ P_\nu(\xi)P_\nu(-\xi)$]\label{prop:P_nu_T}For any  $ \nu\in\mathbb C\smallsetminus\mathbb Z$, we have\begin{align}
\mathcal
P\int_{-1}^1\frac{2P_{\nu}(\xi)P_{\nu}(-\xi)}{\pi(x-\xi)}\D \xi=\frac{[P_{\nu}(x)]^2-[P_{\nu}(-x)]^2}{\sin(\nu\pi)},\quad -1<x<1.\label{eq:P_nu_T}
\end{align}For $n\in\mathbb Z_{\geq0} $, there is an identity\begin{align}
\mathcal
P\int_{-1}^1\frac{[P_{n}(\xi)]^{2}}{2(x-\xi)}\D \xi=\mathcal
P\int_{-1}^1\frac{[P_{-n-1}(\xi)]^{2}}{2(x-\xi)}\D \xi=P_n(x)Q_n(x),\quad -1<x<1.\label{eq:P_n_T}
\end{align}\end{proposition}\begin{proof}To verify Eq.~\ref{eq:P_nu_T}, it would suffice to  demonstrate that\begin{align}0={}&\int_{-1}^1 P_\ell(x)\left\{\frac{[P_{\nu}(x)]^2-[P_{\nu}(-x)]^2}{\sin(\nu\pi)}-\mathcal P\int_{-1}^1\frac{2P_{\nu}(\xi)P_{\nu}(-\xi)}{\pi(x-\xi)}\D \xi\right\}\D x\notag\\={}&\frac{1}{\sin(\nu\pi)}\int_{-1}^1 P_\ell(x)\{[P_{\nu}(x)]^2-[P_{\nu}(-x)]^2\}\D x+\frac{4}{\pi}\int_{-1}^1 Q_\ell(x)P_{\nu}(x)P_{\nu}(-x)\D x,\quad \forall\ell\in\mathbb Z_{\geq0}.\label{eq:P_nu_T_Q}\end{align}Here, in the last line, we have used the  Parseval identity of Tricomi pairing (Eq.~\ref{eq:Tricomi_Parseval1}), along with the Neumann integral representation for Legendre functions of the second kind (cf.~\cite[Ref.][Eq.~11.269]{KingVol1} or \cite[Ref.][Table 1.12A, Eq.~12A.26]{KingVol2}):\begin{align}Q_{\ell}(x)=\mathcal P\int_{-1}^1\frac{P_\ell(\xi)\D \xi}{2(x-\xi)},\quad \forall x\in(-1,1),\forall\ell\in\mathbb Z_{\geq0}.\label{eq:Neumann_int_repn}\end{align}

For  a non-negative even number  $ \ell$, both addends in the last line of Eq.~\ref{eq:P_nu_T_Q} vanish because the integrands are odd functions. We may now settle  Eq.~\ref{eq:P_nu_T_Q} for odd numbers $ \ell$ by induction. For $ \ell=1$, we  use Eq.~\ref{eq:LegendreSqrDiff_nu} to compute  \begin{align*}&\frac{1}{\sin(\nu\pi)}\int_{-1}^1 P_1(x)\{[P_{\nu}(x)]^2-[P_{\nu}(-x)]^2\}\D x=\frac{1}{\sin(\nu\pi)}\int_{-1}^1 x\{[P_{\nu}(x)]^2-[P_{\nu}(-x)]^2\}\D x\notag\\={}&-\frac{1}{4\nu(\nu{+1)}\sin(\nu\pi)}\lim_{x\to1-0^+}(1-x^{2})\frac{\D}{\D x}\left[(1-x^{2})\frac{\D\{[P_{\nu}(x)]^2-[P_{\nu}(-x)]^2\}}{\D x}\right]\notag\\&+\frac{1}{4\nu(\nu{+1)}\sin(\nu\pi)}\lim_{x\to-1+0^+}(1-x^{2})\frac{\D}{\D x}\left[(1-x^{2})\frac{\D\{[P_{\nu}(x)]^2-[P_{\nu}(-x)]^2\}}{\D x}\right]=\frac{4\sin(\nu\pi)}{\nu(\nu+1)\pi^{2}}.\end{align*}Meanwhile,  we may check that\begin{align*}&\frac{4}{\pi}\int_{-1}^1 Q_1(x)P_{\nu}(x)P_{\nu}(-x)\D x=\frac{4}{\pi}\int_{-1}^1 \left(-1+\frac{x}{2}\log\frac{1+x}{1-x}\right)P_{\nu}(x)P_{\nu}(-x)\D x\notag\\={}&-\frac{4}{\pi}\int_{-1}^1 P_{\nu}(x)P_{\nu}(-x)\D x-\frac{2}{\pi}\int_{-1}^1 \frac{\log\frac{1+x}{1-x}}{4\nu(\nu{+1)}}\frac{\D}{\D x}\left\{(1-x^{2})\frac{\D}{\D x}\left[(1-x^{2})\frac{\D(P_{\nu}(x)P_{\nu}(-x))}{\D x}\right]+4\nu(\nu+1)(1-x^2)P_{\nu}(x)P_{\nu}(-x)\right\} \D x\notag\\={}&\frac{4}{\pi}\int_{-1}^1 \frac{1}{4\nu(\nu{+1)}}\frac{\D}{\D x}\left[(1-x^{2})\frac{\D(P_{\nu}(x)P_{\nu}(-x))}{\D x}\right]\D x=-\frac{4\sin(\nu\pi)}{\nu(\nu+1)\pi^{2}}.\end{align*}Thus,   Eq.~\ref{eq:P_nu_T_Q} holds for $ \ell=1$. Noting that the Legendre functions of the first and second kinds satisfy similar recursion relations, namely,\begin{align*}&&(2\mu+1)(1-x^{2})\frac{\D P_\mu(x)}{\D x}={}&\mu(\mu+1)[P_{\mu-1}(x)-P_{\mu+1}(x)];& (2\mu+1) x P_\mu(x)={}&(\mu+1)P_{\mu+1}(x)+\mu P_{\mu-1}(x);&&\notag\\&&(2\mu+1)(1-x^{2})\frac{\D Q_\mu(x)}{\D x}={}&\mu(\mu+1)[Q_{\mu-1}(x)-Q_{\mu+1}(x)];& (2\mu+1) xQ_\mu(x)={}&(\mu+1)Q_{\mu+1}(x)+\mu Q_{\mu-1}(x),&&\end{align*}we can deduce\begin{align*}&\frac{\nu(\nu+1)(\ell+1)}{\sin(\nu\pi)}\int_{-1}^1 P_{\ell+1}(x)\{[P_{\nu}(x)]^2-[P_{\nu}(-x)]^2\}\D x+\frac{4\nu(\nu+1)(\ell+1)}{\pi}\int_{-1}^1 Q_{\ell+1}(x)P_{\nu}(x)P_{\nu}(-x)\D x\notag\\&+\frac{\nu(\nu+1)\ell}{\sin(\nu\pi)}\int_{-1}^1 P_{\ell-1}(x)\{[P_{\nu}(x)]^2-[P_{\nu}(-x)]^2\}\D x+\frac{4\nu(\nu+1)\ell}{\pi}\int_{-1}^1 Q_{\ell-1}(x)P_{\nu}(x)P_{\nu}(-x)\D x\notag\\={}&\frac{\nu(\nu+1)(2\ell+1)}{\sin(\nu\pi)}\int_{-1}^1 xP_{\ell}(x)\{[P_{\nu}(x)]^2-[P_{\nu}(-x)]^2\}\D x+\frac{4\nu(\nu+1)(2\ell+1)}{\pi}\int_{-1}^1 xQ_{\ell}(x)P_{\nu}(x)P_{\nu}(-x)\D x,\end{align*} as well as{\allowdisplaybreaks\begin{align*}&\frac{4\nu(\nu+1)}{2\ell+1}\int_{-1}^1 [(\ell+1)P_{\ell+1}(x)+\ell P_{\ell-1}(x)]\{[P_{\nu}(x)]^2-[P_{\nu}(-x)]^2\}\D x=4\nu(\nu+1)\int_{-1}^1 xP_{\ell}(x)\{[P_{\nu}(x)]^2-[P_{\nu}(-x)]^2\}\D x\notag\\={}&-\int_{-1}^1P_{\ell}(x)\frac{\D}{\D x}\left\{(1-x^{2})\frac{\D}{\D x}\left[(1-x^{2})\frac{\D\{[P_{\nu}(x)]^2-[P_{\nu}(-x)]^2\}}{\D x}\right]+4\nu(\nu+1)(1-x^{2})[P_{\nu}(x)]^2-4\nu(\nu+1)(1-x^{2})[P_{\nu}(x)]^2\right\}\D x\notag\\={}&\int_{-1}^1\left\{(1-x^{2})\frac{\D}{\D x}\left[(1-x^{2})\frac{\D\{[P_{\nu}(x)]^2-[P_{\nu}(-x)]^2\}}{\D x}\right]+4\nu(\nu+1)(1-x^{2})[P_{\nu}(x)]^2-4\nu(\nu+1)(1-x^{2})[P_{\nu}(x)]^2\right\}\frac{\D P_{\ell}(x)}{\D x}\D x\notag\\&+\frac{8[1+(-1)^{\ell}]\sin^2(\nu\pi)}{\pi^{2}}\notag\\={}&\ell(\ell+1)\int_{-1}^1(1-x^{2})P_{\ell}(x)\frac{\D\{[P_{\nu}(x)]^2-[P_{\nu}(-x)]^2\}}{\D x}\D x+\frac{4\ell(\ell+1)\nu(\nu+1)}{2\ell+1}\int_{-1}^1 [P_{\ell-1}(x)-P_{\ell+1}(x)]\{[P_{\nu}(x)]^2-[P_{\nu}(-x)]^2\}\D x\notag\\&+\frac{8[1+(-1)^{\ell}]\sin^2(\nu\pi)}{\pi^{2}},\end{align*}}which leads to  a recursion relation\begin{align*}&(\ell+1)^2[(\ell+1)^2-(2\nu+1)^2]\int_{-1}^1 P_{\ell+1}(x)\{[P_{\nu}(x)]^2-[P_{\nu}(-x)]^2\}\D x\notag\\{}&-\ell^{2}[\ell^{2}-(2\nu+1)^2]\int_{-1}^1 P_{\ell-1}(x)\{[P_{\nu}(x)]^2-[P_{\nu}(-x)]^2\}\D x=-\frac{8[1+(-1)^{\ell}](2\ell+1)\sin^2(\nu\pi)}{\pi^{2}},\end{align*} and {\allowdisplaybreaks\begin{align*}&\frac{4\nu(\nu+1)}{2\ell+1}\int_{-1}^1 [(\ell+1)Q_{\ell+1}(x)+\ell Q_{\ell-1}(x)]P_{\nu}(x)P_{\nu}(-x)\D x=4\nu(\nu+1)\int_{-1}^1 xQ_{\ell}(x)P_{\nu}(x)P_{\nu}(-x)\D x\notag\\={}&-\int_{-1}^1Q_{\ell}(x)\frac{\D}{\D x}\left\{(1-x^{2})\frac{\D}{\D x}\left[(1-x^{2})\frac{\D(P_{\nu}(x)P_{\nu}(-x))}{\D x}\right]+4\nu(\nu+1)(1-x^{2})P_{\nu}(x)P_{\nu}(-x)\right\}\D x\notag\\={}&\int_{-1}^1\left\{(1-x^{2})\frac{\D}{\D x}\left[(1-x^{2})\frac{\D(P_{\nu}(x)P_{\nu}(-x))}{\D x}\right]+4\nu(\nu+1)(1-x^{2})P_{\nu}(x)P_{\nu}(-x)\right\}\frac{\D Q_{\ell}(x)}{\D x}\D x\notag\\={}&\ell(\ell+1)\int_{-1}^1(1-x^{2})Q_{\ell}(x)\frac{\D(P_{\nu}(x)P_{\nu}(-x))}{\D x}\D x+\frac{4\ell(\ell+1)\nu(\nu+1)}{2\ell+1}\int_{-1}^1 [Q_{\ell-1}(x)-Q_{\ell+1}(x)]P_{\nu}(x)P_{\nu}(-x)\D x\notag\\&-\frac{2[1+(-1)^{\ell}]\sin(\nu\pi)}{\pi},\end{align*}}which brings us another recursion relation\begin{align*}&(\ell+1)^2[(\ell+1)^2-(2\nu+1)^2]\int_{-1}^1 Q_{\ell+1}(x)P_{\nu}(x)P_{\nu}(-x)\D x\notag\\{}&-\ell^{2}[\ell^{2}-(2\nu+1)^2]\int_{-1}^1 Q_{\ell-1}(x)P_{\nu}(x)P_{\nu}(-x)\D x=\frac{2[1+(-1)^{\ell}](2\ell+1)\sin(\nu\pi)}{\pi}.\end{align*}Here, we have used the identities\begin{align*}&\lim_{x\to-1+0^+}(1-x^2 )\frac{\D}{\D x}\left[ 2(1-x^2 )P_{\nu}(x)\frac{\D P_{\nu}(x) }{\D x}\right]=\frac{8\sin^2(\nu\pi)}{\pi^{2}};\quad P_\ell(1)=(-1)^\ell P_\ell(-1)=1,\forall\ell\in\mathbb Z_{\geq0}\end{align*}to take care of boundary contributions to the integral concerning $ P_\ell(x)$, and resorted to  the limit behavior\begin{align*}\lim_{x\to1-0^+}(1-x^{2})^2\frac{\D Q_{\mu}(x)}{\D x}\frac{\D(P_{\nu}(x)P_{\nu}(-x))}{\D x}={}&-\frac{2\sin(\nu\pi)}{\pi}\notag\\\lim_{x\to-1+0^+}(1-x^{2})^2\frac{\D Q_{\mu}(x)}{\D x}\frac{\D(P_{\nu}(x)P_{\nu}(-x))}{\D x}={}&\frac{2\cos(\mu\pi)\sin(\nu\pi)}{\pi}\end{align*} for the integration by parts involving $ Q_\ell(x)$. Therefore, the last line of     Eq.~\ref{eq:P_nu_T_Q} satisfies a homogeneous recursion\begin{align*}&(\ell+1)^2[(\ell+1)^2-(2\nu+1)^2]\left\{\frac{1}{\sin(\nu\pi)}\int_{-1}^1 P_{\ell+1}(x)\{[P_{\nu}(x)]^2-[P_{\nu}(-x)]^2\}\D x+\frac{4}{\pi}\int_{-1}^1 Q_{\ell+1}(x)P_{\nu}(x)P_{\nu}(-x)\D x\right\}\notag\\={}&\ell^{2}[\ell^{2}-(2\nu+1)^2]\left\{\frac{1}{\sin(\nu\pi)}\int_{-1}^1 P_{\ell-1}(x)\{[P_{\nu}(x)]^2-[P_{\nu}(-x)]^2\}\D x+\frac{4}{\pi}\int_{-1}^1 Q_{\ell-1}(x)P_{\nu}(x)P_{\nu}(-x)\D x\right\},\quad \ell\in\mathbb Z_{\geq0}\end{align*} and the truthfulness of    Eq.~\ref{eq:P_nu_T_Q}  for $ \ell=1 $ entails any scenario with a  larger odd number $ \ell$. This completes the verification of Eq.~\ref{eq:P_nu_T} for $ \nu\in\mathbb C\smallsetminus\mathbb Z$.

For a fixed $x\in (-1,1)$, both sides of  Eq.~\ref{eq:P_nu_T} represent continuous functions of $ \nu$, so we may handle Eq.~\ref{eq:P_n_T} by  investigating the  $ \nu\to n$ limit where $n\in\mathbb Z_{\geq0}$. Suppose that $ n$ is even and non-negative, then we have\begin{align*}\lim_{\nu\to n}\frac{[P_{\nu}(x)]^2-[P_{\nu}(-x)]^2}{\sin(\nu\pi)}={}&2P_{n}(x)\lim_{\nu\to n}\frac{P_{\nu}(x)-P_{\nu}(-x)}{\sin(\nu\pi)}=2P_{n}(x)\lim_{\nu\to n}\frac{\cos(\nu\pi)P_{\nu}(x)-P_{\nu}(-x)}{\sin(\nu\pi)}\notag\\={}&\frac{4}{\pi}P_n(x)\lim_{\nu\to n}Q_\nu(x)=\frac{4}{\pi}P_n(x)Q_{n}(x)=\frac{4}{\pi}P_n(-x)Q_{n}(x).\end{align*}If we start from an odd and positive $n$ instead, we will end up with\begin{align*}\lim_{\nu\to n}\frac{[P_{\nu}(x)]^2-[P_{\nu}(-x)]^2}{\sin(\nu\pi)}={}&2P_{n}(x)\lim_{\nu\to n}\frac{P_{\nu}(x)+P_{\nu}(-x)}{\sin(\nu\pi)}=-2P_{n}(x)\lim_{\nu\to n}\frac{\cos(\nu\pi)P_{\nu}(x)-P_{\nu}(-x)}{\sin(\nu\pi)}\notag\\={}&-\frac{4}{\pi}P_n(x)\lim_{\nu\to n}Q_\nu(x)=-\frac{4}{\pi}P_n(x)Q_{n}(x)=\frac{4}{\pi}P_n(-x)Q_{n}(x).\end{align*}Hence,  we have proved \begin{align*}\mathcal
P\int_{-1}^1\frac{2P_{n}(\xi)P_{n}(-\xi)}{\pi(x-\xi)}\D \xi=\mathcal
P\int_{-1}^1\frac{2P_{-n-1}(\xi)P_{-n-1}(-\xi)}{\pi(x-\xi)}\D \xi=(-1)^{n}\frac{4}{\pi}P_n(x)Q_n(x)=\frac{4}{\pi}P_n(-x)Q_n(x),\quad n\in\mathbb Z_{\geq0},
\end{align*}which is equivalent to Eq.~\ref{eq:P_n_T}. In fact,  Eq.~\ref{eq:P_n_T} is not particularly surprising. It is just a special case of the stronger statement that (cf.~\cite[Ref.][Eqs.~11.280 and 11.281]{KingVol1} or  \cite[Ref.][Table 1.12A, Eq.~12A.27]{KingVol2})
\begin{align}
\label{eq:Pn_p_T}
\mathcal P\int_{-1}^1\frac{P_{n}(\xi)p(\xi)}{2(x-\xi)}\D \xi=Q_{n}(x)p(x),\quad \deg p(x)\leq n,
\end{align}
 where $ p(x)$ is any polynomial whose degree does not exceed $n$.
\qed\end{proof}

  We follow up with some additional examples involving the product of four elliptic integrals in the integrands.
\begin{corollary}[Another Application of Tricomi Pairing]We have an integral identity\begin{align}
\int_{-1}^1 x[P_\nu(x)]^3P_\nu(-x)\D x=\frac{\sin(2\nu\pi)\cos(\nu\pi)}{(2\nu+1)^{2}\pi},\quad \nu\in\mathbb C\smallsetminus\{-1/2\},\label{eq:xPPPP}
\end{align}which leads to {\allowdisplaybreaks\begin{align}-\frac{\pi}{2}={}&\int_{-1}^1 x[P_{-1/2}(x)]^3P_{-1/2}(-x)\D x\notag\\={}&\frac{32}{\pi^{4}}\int_0^1(1-2t)[\mathbf K(\sqrt{t})]^3\mathbf K(\sqrt{1-t})\D t=-\frac{32}{\pi^{4}}\int_0^1(1-2t)[\mathbf K(\sqrt{1-t})]^3\mathbf K(\sqrt{t})\D t,\label{eq:xPPPP_half}\\-\frac{9 \sqrt{3}}{4 \pi }={}&\int_{-1}^1 x[P_{-1/3}(x)]^3P_{-1/3}(-x)\D x\notag\\={}&\frac{216}{\sqrt{3}\pi^{4}}\int_0^1\frac{(1-p^2)p(2+p)}{1+2p}\left[1-2\frac{27p^2(1+p)^2}{4(1+p+p^2)^3}\right]\left[ \mathbf K\left( \sqrt{\frac{p^{3}(2+p)}{1+2p}} \right) \right]^3\mathbf K\left( \sqrt{1-\frac{p^{3}(2+p)}{1+2p}} \right)\D p\notag\\={}&-\frac{72}{\sqrt{3}\pi^{4}}\int_0^1\frac{(1-p^2)p(2+p)}{1+2p}\left[1-2\frac{27p^2(1+p)^2}{4(1+p+p^2)^3}\right]\left[ \mathbf K\left( \sqrt{1-\frac{p^{3}(2+p)}{1+2p}} \right) \right]^3\mathbf K\left( \sqrt{\frac{p^{3}(2+p)}{1+2p}} \right)\D p,\label{eq:xPPPP_third}\\-\frac{2 \sqrt{2}}{\pi}={}&\int_{-1}^1 x[P_{-1/4}(x)]^3P_{-1/4}(-x)\D x\notag\\={}&\frac{32\sqrt{2}}{\pi^4}\int_0^1\frac{1-2u}{(1+\sqrt{u})^2}\left[ \mathbf K\left( \sqrt{\frac{2\sqrt{\vphantom{1}u}}{1+\sqrt{\vphantom{1}u}}} \right) \right]^3\mathbf K\left( \sqrt{\frac{1-\sqrt{\vphantom{1}u}}{1+\sqrt{\vphantom{1}u}}} \right)\D u\notag\\={}&-\frac{64\sqrt{2}}{\pi^4}\int_0^1\frac{1-2u}{(1+\sqrt{u})^2}\left[ \mathbf K\left( \sqrt{\frac{1-\sqrt{\vphantom{1}u}}{1+\sqrt{\vphantom{1}u}}} \right) \right]^3\mathbf K\left( \sqrt{\frac{2\sqrt{\vphantom{1}u}}{1+\sqrt{\vphantom{1}u}}} \right)\D u,\label{eq:xPPPP_quarter}\\-\frac{27}{16\pi}={}&\int_{-1}^1 x[P_{-1/6}(x)]^3P_{-1/6}(-x)\D x\notag\\={}&\frac{54}{\pi^{4}}\int_0^1\frac{t(1-t)(1+t)(2-t)(1-2t)}{(1-t+t^2)^3}[\mathbf K(\sqrt{t})]^3\mathbf K(\sqrt{1-t})\D t\notag\\={}&-\frac{54}{\pi^{4}}\int_0^1\frac{t(1-t)(1+t)(2-t)(1-2t)}{(1-t+t^2)^3}[\mathbf K(\sqrt{1-t})]^3\mathbf K(\sqrt{t})\D t.\label{eq:xPPPP_sixth}\end{align}} \end{corollary}\begin{proof}For integer degrees $ \nu\in\mathbb Z$, the left-hand side of  Eq.~\ref{eq:xPPPP} represents the integration of an odd function over the interval $ [-1,1]$, hence vanishing. This is consistent with the corresponding behavior  on the right-hand side.

We now turn our attention to  the scenarios where $ \nu\in\mathbb C\smallsetminus(\mathbb Z\cup\{-1/2\})$. From the Tricomi transform formula in Eq.~\ref{eq:P_nu_T}, one may readily deduce the following relation\begin{align*}\mathcal
P\int_{-1}^1\frac{2\xi P_{\nu}(\xi)P_{\nu}(-\xi)}{\pi(x-\xi)}\D \xi=\mathcal
P\int_{-1}^1\frac{2[x-(x-\xi)] P_{\nu}(\xi)P_{\nu}(-\xi)}{\pi(x-\xi)}\D \xi=x\frac{[P_{\nu}(x)]^2-[P_{\nu}(-x)]^2}{\sin(\nu\pi)}-\frac{2}{\pi}\int_{-1}^{1}P_{\nu}(\xi)P_{\nu}(-\xi)\D \xi.\end{align*} Here, according to Eq.~\ref{eq:T_0_nu}, we have\begin{align*}T_{0,\nu}=\int_{-1}^{1}P_{\nu}(\xi)P_{\nu}(-\xi)\D \xi=\frac{2\cos(\nu\pi)}{2\nu+1}, \end{align*} so we may combine the formula \begin{align*}\mathcal
P\int_{-1}^1\frac{2\xi P_{\nu}(\xi)P_{\nu}(-\xi)}{\pi(x-\xi)}\D \xi=x\frac{[P_{\nu}(x)]^2-[P_{\nu}(-x)]^2}{\sin(\nu\pi)}-\frac{4\cos(\nu\pi)}{(2\nu+1)\pi}\end{align*}with  Eq.~\ref{eq:P_nu_T} for the implementation of the following Tricomi pairing:\begin{align*}\int_{-1}^1x P_\nu(x)P_\nu(-x)\{[P_\nu(x)]^2-[P_\nu(-x)]^2\}\D x={}&\int_{-1}^1x P_\nu(x)P_\nu(-x)\left[\mathcal P\int_{-1}^1\frac{2 P_{\nu}(\xi)P_{\nu}(-\xi)\sin(\nu\pi)}{\pi(x-\xi)}\D \xi\right]\D x\notag\\={}&-\int_{-1}^1\left[ \mathcal
P\int_{-1}^1\frac{2\xi P_{\nu}(\xi)P_{\nu}(-\xi)\sin(\nu\pi)}{\pi(x-\xi)}\D \xi \right]P_\nu(x)P_\nu(-x)\D x\notag\\={}&-\int_{-1}^1x P_\nu(x)P_\nu(-x)\{[P_\nu(x)]^2-[P_\nu(-x)]^2\}\D x+\frac{2\sin(2\nu\pi)}{(2\nu+1)\pi}\int _{-1}^1P_\nu(x)P_\nu(-x)\D x.\end{align*}After rearrangement, we obtain\begin{align*}4\int_{-1}^1 x[P_\nu(x)]^3P_\nu(-x)\D x=2\int_{-1}^1x P_\nu(x)P_\nu(-x)\{[P_\nu(x)]^2-[P_\nu(-x)]^2\}\D x=\frac{2\sin(2\nu\pi)}{(2\nu+1)\pi}\int _{-1}^1P_\nu(x)P_\nu(-x)\D x,\end{align*} which entails the claimed identity in Eq.~\ref{eq:xPPPP}.

The left-hand side of  Eq.~\ref{eq:xPPPP} extends to be a continuous function in $ \nu\in\mathbb C$.  Taking the $ \nu\to-1/2$ limit, one arrives at Eq.~\ref{eq:xPPPP_half}. The special cases $ \nu=-1/3,-1/4,-1/6$ correspond to Eqs.~\ref{eq:xPPPP_third}, \ref{eq:xPPPP_quarter} and \ref{eq:xPPPP_sixth}. \qed\end{proof}

A key step in the proof of Proposition~\ref{prop:Tricomi_Paring} hinges on the identity\begin{align*}\int_0^1\left[\mathbf K \left(\sqrt{1-k^2} \right)\right]^{3}\D k=3\int_0^1[\mathbf K(k)]^2\mathbf K\left( \sqrt{1-k^2} \right)\D k,\quad \textit{i.e. }\int_{-1}^1\frac{[P_{-1/2}(x)]^3}{\sqrt{1+x}}\D x=3\int_{-1}^1\frac{P_{-1/2}(x)[P_{-1/2}(-x)]^2}{\sqrt{1+x}}\D x.\end{align*}This result is actually just a special case within a family of identities satisfied by multiple elliptic integrals involving the product of three complete elliptic integrals (of the first and second kinds). To flesh out, for any integer $ n\in\mathbb Z$, we have\begin{align}
\mathcal P\int_{-1}^1\frac{P_{(2n+1)/2}(\xi)}{\sqrt{1+\xi}}\frac{\D \xi}{2(x-\xi)}=\frac{Q_{(2n+1)/2}(x)}{\sqrt{1+x}}\equiv(-1)^{n+1}\frac{\pi}{2}\frac{P_{(2n+1)/2}(-x)}{\sqrt{1+x}},\quad \forall x\in(-1,1),\label{eq:P_half_int_T}
\end{align}which  generalizes Corollary~\ref{cor:Beltrami_app}\ref{itm:Tricomi_P_pm_half}, and Eq.~\ref{eq:P_half_int_T} entails \begin{align}
\int_{-1}^1\frac{[P_{(2n+1)/2}(x)]^3}{\sqrt{1+x}}\D x=3\int_{-1}^1\frac{P_{(2n+1)/2}(x)[P_{(2n+1)/2}(-x)]^2}{\sqrt{1+x}}\D x\label{eq:triple_law}
\end{align}after Tricomi pairing with Eq.~\ref{eq:P_nu_T}. In particular, for $ n=0$, Eq.~\ref{eq:triple_law} specializes to  \begin{align*}\int_0^1\left[2\mathbf E\left(\sqrt{1-k^2} \right)-\mathbf K \left(\sqrt{1-k^2} \right)\right]^{3}\D k=3\int_0^1[2\mathbf E(k)-\mathbf K(k)]^2 \left[2\mathbf E\left(\sqrt{1-k^2} \right)-\mathbf K \left(\sqrt{1-k^2} \right)\right]\D k.\end{align*}

The identity stated in   Eq.~\ref{eq:P_half_int_T} is the Tricomi transform of $ P_{(2n+1)/2}(x)/\sqrt{1+x}\equiv P_{-(2n+3)/2}(x)/\sqrt{1+x},n\in\mathbb Z$, which extends the Neumann integral representation of Legendre functions $ Q_\ell$ for $\ell\in\mathbb Z_{\geq0}$ (Eq.~\ref{eq:Neumann_int_repn}). In the next proposition, we shall prove    Eq.~\ref{eq:P_half_int_T} in an even broader context, which in turn, also gives rise to an independent verification of Eq.~\ref{eq:P_nu_T} in Proposition~\ref{prop:P_nu_T}.
\begin{proposition}[Generalized Neumann Integrals]
For $ n\in\mathbb Z_{\geq0}$ and $ \R(\nu-n)>-1$, one has\begin{align}
\mathcal
P\int_{-1}^1(1+\xi)^{\nu-n} P_{\nu}(\xi)\frac{\D \xi}{2(x-\xi)}={}&(1+x)^{\nu-n} Q_{\nu}(x),\quad -1<x<1.\label{eq:P_nu_T00}
\end{align}\end{proposition}\begin{proof} We note that there is a  standard moment  formula for Legendre functions \cite[Ref.][item~7.127]{GradshteynRyzhik}:\begin{align}\int_{-1}^1(1+x)^\sigma P_\nu(x)\D x=\frac{2^{\sigma+1}[\Gamma(\sigma+1)]^2}{\Gamma(\sigma+\nu+2)\Gamma(1+\sigma-\nu)},\quad \R\sigma>-1.\label{eq:alg_P_nu_int}\end{align}{One can verify  Eq.~\ref{eq:alg_P_nu_int} by termwise integration over the Taylor series expansion for $ P_{\nu}(x)={_2}F_1\left( \left.\begin{smallmatrix}-\nu,\nu+1\\1\end{smallmatrix}\right| \frac{1-x}{2}\right)$ with respect to $ (1-x)/2$ and a reduction of the generalized hypergeometric series $ _3F_2$.}

We now consider the Mellin inversion of Eq.~\ref{eq:alg_P_nu_int}:\begin{align}
\frac{1}{2\pi i}\int_{c-i\infty}^{c+i\infty}\frac{[\Gamma(s)]^22^s}{\Gamma(s+\nu+1)\Gamma(s-\nu)}\frac{\D s}{(1+\xi)^s}=\begin{cases}P_{\nu}(\xi), & -1<\xi<1 \\
0, & \xi>1 \\
\end{cases}\label{eq:Mellin_Pnu}
\end{align}where $ c>0$ and the integration $ \int_{c-i\infty}^{c+i\infty}:=\lim_{T\to+\infty}\int_{c-iT}^{c+iT}$ is carried out along  a vertical line $ \R s=c$.  

To facilitate analysis, we momentarily assume that $ \R(\nu-n)>-\frac{1}{2}$ and pick $ c=\R(\nu -n)+\frac{1}{2}$, so that $ c>0$ and $ \R(\nu-n-c)=-\frac{1}{2}$.
For $-1<x<1$, one may enlist   Eq.~\ref{eq:Mellin_Pnu} to compute\begin{align}\mathcal
P\int_{-1}^1\frac{(1+\xi)^{\nu-n}P_{\nu}(\xi)\D \xi}{2(x-\xi)}={}&\frac{1}{2\pi i}\int_{\R(\nu -n)+\frac{1}{2}-i\infty}^{\R(\nu -n)+\frac{1}{2}+i\infty}\frac{[\Gamma(s)]^22^s}{\Gamma(s+\nu+1)\Gamma(s-\nu)}\left[\lim_{\varepsilon\to0^+}\int_{-1}^\infty\left( \frac{1}{x-\xi+i\varepsilon} +\frac{1}{x-\xi- i\varepsilon}\right)\frac{(1+\xi)^{\nu-n-s}\D \xi}{4}\right]\D s\notag\\={}&\frac{(1+x)^{\nu-n}}{4i}\int_{\R(\nu -n)+\frac{1}{2}-i\infty}^{\R(\nu -n)+\frac{1}{2}+i\infty}\frac{[\Gamma(s)]^22^{s}\cot(\nu\pi-s\pi)}{\Gamma(s+\nu+1)\Gamma(s-\nu)}\frac{\D s}{(1+x)^s},\label{eq:Mellin_Tricomi}\end{align} where the integration over $ \xi\in(-1,\infty)$ is an elementary exercise in complex analysis.
From Eq.~\ref{eq:Mellin_Tricomi}, it is straightforward to verify the Legendre differential equation:\begin{align*}&\frac{\D}{\D x}\left[ (1-x^2)\frac{\D}{\D x} \mathcal
P\int_{-1}^1\frac{(1+\xi)^{\nu-n}P_{\nu}(\xi)\D \xi}{2(1+x)^{\nu-n}(x-\xi)}\right]+\nu(\nu+1)\mathcal P\int_{-1}^1\frac{(1+\xi)^{\nu-n}P_{\nu}(\xi)\D \xi}{2(1+x)^{\nu-n}(x-\xi)}\notag\\={}&\frac{1}{4i}\left\{\int_{\R(\nu -n)+\frac{1}{2}-i\infty}^{\R(\nu -n)+\frac{1}{2}+i\infty}\frac{[\Gamma(s)]^22^{s}\cot(\nu\pi-s\pi)}{\Gamma(s+\nu+1)\Gamma(s-\nu)}\frac{2s^{2}\D s}{(1+x)^{s+1}}-\int_{\R(\nu -n)+\frac{1}{2}-i\infty}^{\R(\nu -n)+\frac{1}{2}+i\infty}\frac{[\Gamma(s)]^22^{s}\cot(\nu\pi-s\pi)}{\Gamma(s+\nu+1)\Gamma(s-\nu)}\frac{(s+\nu)(s-\nu-1)\D s}{(1+x)^{s}}\right\}\notag\\={}&\frac{1}{4i}\left(\int_{\R(\nu -n)+\frac{1}{2}-i\infty}^{\R(\nu -n)+\frac{1}{2}+i\infty}-\int_{\R(\nu -n)-\frac{1}{2}-i\infty}^{\R(\nu -n)-\frac{1}{2}+i\infty}\right)\frac{[\Gamma(s)]^22^{s}\cot(\nu\pi-s\pi)}{\Gamma(s+\nu+1)\Gamma(s-\nu)}\frac{2s^{2}\D s}{(1+x)^{s+1}}=0\end{align*}by noting that $ s=0$ and $ s=\nu-n$ are both removable singularities for the integrand in the last line.  Thus, we are sure that the left-hand side of Eq.~\ref{eq:P_nu_T00} can be written as $ (1+x)^{\nu-n}[a_{\nu,n} P_\nu(x)+b_{\nu,n} Q_\nu(x)]$ for some coefficients $ a_{\nu,n}$ and $ b_{\nu,n}$, provided   that $ \R(\nu-n)>-\frac{1}{2}$. 

 Instead of directly coping with  $ a_{\nu,n}$ and $ b_{\nu,n}$ for generic $\nu$ and $n$, we impose a temporary constraint that $ n=0,-\frac{1}{2}<\nu<0$. Again, by  Eq.~\ref{eq:Mellin_Tricomi}, we can compute two moment integrals\begin{align*}&\int_{-1}^1[a_{\nu,0} P_\nu(x)+b_{\nu,0} Q_\nu(x)]\D x=\frac{1}{2\pi i}\int_{\nu+\frac{1}{2}-i\infty}^{\nu+\frac{1}{2}+i\infty}\frac{[\Gamma(s)]^2\pi\cot(\nu\pi-s\pi)}{\Gamma(s+\nu+1)\Gamma(s-\nu)}\frac{\D s}{1-s}\notag\\={}&\frac{\cos(\nu\pi)}{\nu(\nu+1)}+\sum _{m=1}^{\infty } \frac{[\Gamma (m+\nu )]^2}{(1-m-\nu ) \Gamma (m) \Gamma (m+2 \nu +1)}=-\frac{2}{\nu(\nu+1)}\sin^2\frac{\nu\pi}{2};\notag\\&\int_{-1}^1[a_{\nu,0} P_\nu(x)+b_{\nu,0} Q_\nu(x)](1+x)\D x=\frac{1}{2\pi i}\int_{\nu+\frac{1}{2}-i\infty}^{\nu+\frac{1}{2}+i\infty}\frac{[\Gamma(s)]^2\pi\cot(\nu\pi-s\pi)}{\Gamma(s+\nu+1)\Gamma(s-\nu)}\frac{2\D s}{2-s}\notag\\={}&-\frac{2\cos(\nu\pi)}{(\nu-1 ) \nu  (\nu +1) (\nu +2)}+\sum _{m=1}^{\infty } \frac{2[ \Gamma (m+\nu )]^2}{(2-m-\nu) \Gamma (m) \Gamma (m+2 \nu +1)}=-\frac{2 [\nu ^2+\nu +\cos (\pi  \nu )-1]}{(\nu -1) \nu  (\nu +1) (\nu +2)}\end{align*}by closing the contours to the right. These two moment integrals reveal that $ a_{\nu,0}=0,b_{\nu,0}=1$ for $  -\frac{1}{2}<\nu<0$. For any fixed $ x\in(-1,1)$,  the integral\begin{align*}\int_{\R\nu+\frac{1}{2}-i\infty}^{\R\nu+\frac{1}{2}+i\infty}\frac{[\Gamma(s)]^22^{s}\cot(\nu\pi-s\pi)}{\Gamma(s+\nu+1)\Gamma(s-\nu)}\frac{\D s}{4i(1+x)^s}=\frac{1}{(1+x)^{\nu}}\mathcal
P\int_{-1}^1\frac{(1+\xi)^{\nu}P_{\nu}(\xi)\D \xi}{2(x-\xi)}\end{align*}is analytic for  $ \R\nu>-\frac{1}{2}$ and is equal to  $ Q_\nu(x)$ for  $  -\frac{1}{2}<\nu<0$, so it can be identified with $ Q_\nu(x)$ for $ \R\nu>-\frac{1}{2}$. Moreover, so long as $ \R(\nu-n)>-\frac{1}{2}$ for some $ n\in\mathbb Z_{\geq0}$, we have \begin{align*}\left(\int_{\R(\nu -n)+\frac{1}{2}-i\infty}^{\R(\nu -n)+\frac{1}{2}+i\infty}-\int_{\R\nu+\frac{1}{2}-i\infty}^{\R\nu +\frac{1}{2}+i\infty}\right)\frac{[\Gamma(s)]^22^{s}\cot(\nu\pi-s\pi)}{\Gamma(s+\nu+1)\Gamma(s-\nu)}\frac{\D s}{(1+x)^s}=0,\end{align*}as the ratio $\cot(\nu\pi-s\pi)/\Gamma(s-\nu) $ remains bounded at all the removable singularities enclosed in the contour. This allows us to simplify  Eq.~\ref{eq:Mellin_Tricomi} into Eq.~\ref{eq:P_nu_T00} for $ \R(\nu-n)>-\frac{1}{2},n\in\mathbb Z_{\geq0}$.      
  By analytic continuation in $ \nu$, one can extend the applicability to $ \R(\nu-n)>-1,n\in\mathbb Z_{\geq0}$.\qed\end{proof}\begin{remark}

One may well recognize that  Eq.~\ref{eq:P_nu_T00}  incorporates the  classical result in Eq.~\ref{eq:Pn_p_T} as a special case:\[\mathcal P\int_{-1}^1\frac{P_{n}(\xi)p(\xi)}{2(x-\xi)}\D \xi=Q_{n}(x)p(x),\quad \deg p(x)\leq n\in\mathbb Z_{\geq0}.\]

To deduce  Eq.~\ref{eq:P_nu_T} from  Eq.~\ref{eq:P_nu_T00}, we need  the Hardy-Poincar\'e-Bertrand formula (see \cite[Ref.][\S4.3, Eq.~4]{TricomiInt} or \cite[Ref.][Eq.~11.52]{KingVol1}):
\begin{align}
\widehat{\mathcal T}[f(\widehat{\mathcal T}g)+g(\widehat{\mathcal T}f)]=(\widehat{\mathcal T}f)(\widehat{\mathcal T}g)-fg,\label{eq:Tricomi_convolution}
\end{align}which applies to the scenarios where  $ f\in L^p(-1,1),p>1$; $ g\in L^q(-1,1),q>1$ and $ \frac1p+\frac1q<1$. By   Eqs.~\ref{eq:P_nu_T00} and \ref{eq:alg_P_nu_int}, we have the following identities valid for $ -1<\nu<0$:\begin{align}(1+x)^{-\nu-1} Q_{-\nu-1}(x)={}&\frac{\pi}{2}\mathcal P\int_{-1}^1(1+\xi)^{-\nu-1} P_{-\nu-1}(\xi)\frac{\D \xi}{\pi(x-\xi)},\label{eq:PQ_spec1}\\(1+x)^{\nu+1} Q_\nu(x)-\frac{2^{\nu} [\Gamma (\nu +1)]^2}{\Gamma (2 \nu +2)}={}&\frac{\pi}{2}\mathcal P\int_{-1}^1(1+\xi)^{\nu+1} P_\nu(\xi)\frac{\D \xi}{\pi(x-\xi)}.\label{eq:PQ_spec2}\end{align}Taking the last pair of equations (Eqs.~\ref{eq:PQ_spec1} and \ref{eq:PQ_spec2}) as inputs for the  Hardy-Poincar\'e-Bertrand formula (Eq.~\ref{eq:Tricomi_convolution}), we obtain the following identity for $ -1<\nu<0$ and $ -1<x<1$:\begin{align*}&\frac{\pi}{2}\mathcal P\int_{-1}^1P_{\nu}(\xi)[Q_\nu(\xi)+Q_{-\nu-1}(\xi)]\frac{\D \xi}{\pi(x-\xi)}=Q_\nu(x)Q_{-\nu-1}(x).\end{align*}According to  the relation between $ Q_\nu$ and $ P_\nu$ (Eq.~\ref{eq:def_Qnu}),  the  last equation  reduces to the Tricomi transform of $ P_\nu(\xi)P_\nu(-\xi)$ (Eq.~\ref{eq:P_nu_T}) upon analytic continuation in $\nu$. 
    \eor\end{remark}

\section{\label{sec:outlook}Discussion and Outlook}

In our evaluations of some generalized Clebsch-Gordan integrals (Eqs.~\ref{eq:T_half_half}-\ref{eq:T_sixth_sixth}), the expressions involving  Euler's gamma function actually correspond to certain special values of complete elliptic integrals as well. For example, the knowledge of $\mathbf K(1/\sqrt{2})=[\Gamma(\frac14)]^2/(4\sqrt{\pi}) $ \cite[Ref.][item~8.129.1]{GradshteynRyzhik}
allows us to recast the formula $ [\Gamma(\frac{1}{4})]^{8}/(128\pi^2)=\int_0^1[\mathbf K( \sqrt{1-k^2} )]^{3}\D k$ into the following form:\begin{align}2\left[\int_0^1\frac{\D s}{\sqrt{1-s^2}\sqrt{1-(s^2/2)}}\right]^4=\int_0^1\left[\int_0^1\frac{\D t}{\sqrt{1-t^2}\sqrt{1-(1-k^{2})t^{2}}}\right]^{3}\D k.\label{eq:alg_id}\end{align}Here, both sides are integrations of algebraic functions over algebraic domains, which qualify them as members in the ``ring of periods'' defined by Kontsevich and Zagier \cite{KontsevichZagier}. It is generally believed that identities for periods can be proved  by ``algebraic means'' \cite{KontsevichZagier}, namely, relying on nothing else than additivity of the integral, algebraic change of variables,  and the Newton-Leibniz-Stokes formula.
However, the analytic proof we produced  for the  generalized Clebsch-Gordan integrals does not  fall into such a category: we have invoked Bessel functions, which are ``exponential periods'' \cite{KontsevichZagier}. We would like to see  purely algebraic evaluations of the  generalized Clebsch-Gordan integrals when the choice of degree $\nu$ leads to  special values of complete elliptic integrals. In particular, it could be interesting to search for potential  connections to high degree modular equations and the Chowla-Selberg theory.

After communicating the first version of this manuscript to Prof.\ Jonathan M. Borwein in January 2013, I was sent a preview of a forthcoming book  \cite{LatticeSum2013} that contains a sketched proof for Eq.~\ref{eq:alg_id} using lattice sums and modular forms, totally independent of the methods presented in my work. Later this January,  James G. Wan also wrote me about his plan to give a fuller account for the proofs of his own conjectures in a joint work  with  Rogers and  Zucker, currently available as \cite{RogersWanZucker2013preprint}. Some new integrals in  \cite{RogersWanZucker2013preprint} have inspired me to compose a short sequel  \cite{Zhou2013Int3Pnu}   to the current work, in which there are further applications of the   Hardy-Poincar\'e-Bertrand formula (Eq.~\ref{eq:Tricomi_convolution}) and the Tricomi transform of $ (1+\xi)^{\nu-n} P_{\nu}(\xi),n\in\mathbb Z_{\geq0},\R(\nu-n)>-1$ (Eq.~\ref{eq:P_nu_T00}), as well as an extension of the Hansen-Heine scaling analysis in Proposition~\ref{prop:gen_CG} to the computations of other multiple elliptic integrals.

The spherical methods in this work (Legendre functions and spherical rotations) enable us to  handle a large variety of multiple elliptic integrals, but they are by no means a cure-all. The methods of Bessel moments \cite{Bailey2008,BNSW2011,BSWZ2012}, geometric transformations of Watson type \cite{Zucker2011}, hypergeometric summations \cite{Wan2012} still play  fundamental r\^oles in our quantitative understandings for integrals over elliptic integrals.

With a synthesis of various techniques, it is sometimes possible to handle integrals over the product of more than three Legendre functions of the same degree $ \nu\in\mathbb C$, such as\begin{align}
\int_{-1}^1 x[P_\nu(x)]^4\D x=\lim_{z\to\nu}\frac{2\sin^4(\pi z)[\psi^{(2)}(z+1)+\psi^{(2)}(-z)+28\zeta(3)]}{(2z+1)^2\pi^4},\quad\text{where }\psi^{(2)}(z):=\frac{\D^3}{\D z^3}\log\Gamma(z).\label{eq:4Pnu_prod}
\end{align} One may wish to compare Eq.~\ref{eq:4Pnu_prod} to  the integrals $ \int_{-1}^1 x[P_\nu(x)]^3P_\nu(-x)\D x,\nu\in\mathbb C$  (Eq.~\ref{eq:xPPPP}) that reduce to elementary functions. Clearly,
special cases of   Eq.~\ref{eq:4Pnu_prod}
bring us some interesting evaluations of multiple elliptic integrals. For example, we have the following integral representations for Ap\'ery's constant  $ \zeta(3)=\sum_{n=1}^\infty n^{-3}$:\begin{align*}\zeta(3)=-\frac{\pi^{4}}{243}\int_{-1}^1 x[P_{-1/3}(x)]^4\D x=-\frac{\pi^{4}}{168}\int_{-1}^1 x[P_{-1/4}(x)]^4\D x=-\frac{2\pi^{4}}{189}\int_{-1}^1 x[P_{-1/6}(x)]^4\D x\end{align*}as well as a critical scenario involving $ \zeta(5)=\sum_{n=1}^\infty n^{-5}$:\[ \zeta(5)=-\frac{\pi^{4}}{372}\int_{-1}^1 x[P_{-1/2}(x)]^4\D x=\frac{8}{93}\int_{0}^1(2t-1)[\mathbf K(\sqrt{t})]^4\D t.\]

\noindent \textbf{Acknowledgements} The author thanks two anonymous referees for their thoughtful comments that helped improve the presentation of this paper.
This work was partly supported by the Applied Mathematics Program within the Department of Energy (DOE) Office of Advanced Scientific Computing Research (ASCR) as part of the Collaboratory on Mathematics for Mesoscopic Modeling of Materials (CM4).
The author thanks Prof.~Weinan E (Princeton University) for his encouragements.
\bibliography{Pnu}
\bibliographystyle{unsrt}

\end{document}